\documentclass[12pt, letterpaper, amsart, leqno]{article}
\usepackage{amsmath,amstext}
\usepackage{amssymb}
\usepackage{amsfonts}
\usepackage{amsthm}
\usepackage[english]{babel}
\usepackage[pdftex]{graphicx}
\usepackage[applemac]{inputenc}
\usepackage{float}
\usepackage{pdfsync}
\usepackage[all]{xy}
\usepackage[percent]{overpic}
\usepackage{scrtime}
\usepackage{marginnote}
\usepackage{enumerate}
\usepackage[useregional]{datetime2}
\usepackage{siunitx}
\usepackage{cclicenses}
\usepackage{hyperref}

\numberwithin{equation}{section}

\DeclareMathOperator{\vol}{vol}
\DeclareMathOperator{\Id}{Id}
\DeclareMathOperator{\supp}{supp}
\DeclareMathOperator{\Image}{Im}
\newcommand {\Ham}{{\mathfrak {Ham}}}
\newcommand {\DHam}{{Ham}}
\newcommand{\gclDHam}[1]{\widehat{\DHam}(#1)}

\newcommand{\gclHamc}[1]{\widehat{\Ham}_c(#1)}
\newcommand{\gclDHamc}[1]{\widehat{\DHam}_c(#1)}
\def\GFQI   {GFQI\,}
\def \F {F}
\def \varpsi {\varphi}
\def \phi {\varphi}

\def \a {a}
\def \A {A}

\def \dispdot {}

\begin{document}
 \renewcommand{\labelenumi}{\bf (\arabic{enumi})}
 \renewcommand{\theenumi}{(\arabic{enumi})}

\theoremstyle{plain}
\newtheorem{theorem}{Theorem}[section]
\newtheorem{lemma}[theorem]{Lemma}
\newtheorem{proposition}[theorem]{Proposition}
\newtheorem{corollary}[theorem]{Corollary}
\newtheorem{conjecture}{Conjecture}
\newtheorem*{conjecture-non}{Conjecture}
\newtheorem{definition}[theorem]{Definition}
\theoremstyle{definition}
\newtheorem{example}[theorem]{Example}
\newtheorem{prop-def}[theorem]{Proposition \& Definition}
\theoremstyle{remark}
\newtheorem{remark}[theorem]{Remark}
\newtheorem{remarks}[theorem]{Remarks}
\newtheorem{question}[theorem]{Question}

\title{Symplectic Homogenization}
\author{  \\{ \sc Claude Viterbo}
\\ DMA,  UMR  8553 du CNRS
\\ \'Ecole Normale Sup\'erieure, PSL University\\ 45 Rue d'Ulm\\ 75230 Paris Cedex 05,  France
\thanks{Current address: Laboratoire Math\'ematique d'Orsay, B\^atiment 307, Facult\'e des sciences d'Orsay,  Universit\'e de Paris-Saclay, 91405 Orsay Cedex. Email: Claude.Viterbo@universite-paris-saclay.fr}}
\date{ \today \; at \DTMcurrenttime}
 \maketitle

\begin{abstract}
Let  $H(q,p)$ be a Hamiltonian on $T^*T^n$. Under suitable assumptions on $H$, we show that the sequence $(H_{k})_{k\geq 1}$ defined by $H_{k}(q,p)=H(kq,p)$ converges in the $\gamma$-topology -defined in \cite{Viterbo-STAGGF}- to an integrable continuous Hamiltonian  $\overline{H}(p)$. This is extended to the case of non-autonomous Hamiltonians, and the more general setting in which only some of the variables are homogenized: we consider  the sequence $H(kx,y,q,p)$ and prove it has a $\gamma$-limit  ${\overline H}(y,q,p)$, thus yielding an ``effective Hamiltonian''.
The goal of this paper is to prove  convergence of the above sequences,  state the first properties of the homogenization operator, and give some applications to  solutions of Hamilton-Jacobi equations, construction of quasi-states, etc...
We also prove that when $H$ is convex in $p$, the function $\overline H$ coincides with Mather's $\alpha$ function  defined in \cite{Ma} and associated to the Legendre dual of $H$. This  gives a new proof of its symplectic invariance first discovered  by P. Bernard  in \cite{Bernard2}.
\newline
Keywords: Homogenization, symplectic topology,  Hamiltonian flow, Hamilton-Jacobi equation, variational solutions. AMS Classification: primary: 37J05,  53D35 secondary:  35F20, 49L25, 37J40, 37J50.
\end{abstract}
 \tableofcontents
 \listoffigures

\section{Introduction}
The aim of this paper is to define the notion of homogenization for a
Hamiltonian diffeomorphism of $T^{*}T^{n}$. In other words, given a compactly supported Hamiltonian 
$H(t,q,p)$ on $S^{1}\times T^{*}T^{n}$, we shall study
whether  the sequence $H_{k}$ defined by $H_{k}(t,q,p) = H(kt,kq,p)$ converges
to some Hamiltonian $\overline{H}$, necessarily of the form ${\overline H}(q,p)= {\overline h} (p)$.

The convergence of $(H_{k})_{k\geq 1}$ to $\overline H$ should be understood as the convergence for the symplectic metric $\gamma$ defined  in \cite{Viterbo-STAGGF} (see also Section \ref{sec-3.1}), of
the  flow of $H_{k}$, $\varphi^{t}_{k}$ to the flow
of $\overline{H}$, $\overline{\varphi}^{\;t}$.
This convergence is necessarily rather weak, since for example  $C^{0}$-convergence for the flows essentially never holds. 

However, such $\gamma$-convergence implies the $C^0$-convergence for  the variational solution (see \cite{Viterbo-Ottolenghi} for the definition) of  Hamilton-Jacobi equations

\begin{equation}
\tag{$HJ_{k}$}
\left\lbrace
\begin{array}{ll}
& \frac{\partial}{\partial t} u (t,q) + H (k\cdot t,k\cdot q, \frac{\partial}{\partial
q}u (t,q)) = 0\\
& u (0,q) = f(q)
\end{array}
\right.
\end{equation}
to the variational solutions of
 \begin{equation}
\tag{$\overline{HJ}$}
\left\lbrace
\begin{array}{ll}
& \frac{\partial}{\partial t} u (t,q) + \overline H (\frac{\partial}{\partial q}u (t,q)) = 0\\
& u (0,q) = f(q).
\end{array}
\right.
\end{equation}

It is important to stress that this notion of convergence {\bf does not} imply any kind of
pointwise or almost everywhere convergence\footnote{However $C^{0}$ convergence of the flows implies
$\gamma$-convergence as we proved in \cite{Viterbo-STAGGF}, we
refer to Humili\`ere's work in \cite{Humiliere} for stronger statements, i.e. weaker assumptions.} for $\varphi_k^t$ or $H_k$, but is rather
related to variational notions of convergence, similar to $\Gamma$-convergence
(see \cite{de Giorgi}, \cite{Dal Maso}), that  was  already used in
homogenization theory for studying viscosity solutions of
Hamilton-Jacobi equations. Homogenization using this method was for example used in the work of Lions-Papanicolaou and Varadhan (see  \cite{L-P-V} and also  \cite{Evans}, \cite{Fathi1}), or the rescaling of metrics on $T^{n}$ (see \cite{Acerbi-Buttazzo} and \cite{Gromov}).

\medskip

All the above-mentioned papers can be considered as special cases of ``symplectic homogenization'' that is presented here. We believe some
of the advantages of this unified treatment are

\begin{enumerate}

\item The removal of any convexity or even coercivity  assumption on $H$ in the $p$
direction (this is required in \cite{L-P-V} and \cite{Evans, Fathi1, Fathi3}), usually needed to define $\overline H$  because of the use of minimization
techniques for the Lagrangian. In fact our homogenization is defined on compactly supported objects, and then extended to a number of non-compactly supported situations.

\item  The natural extension of homogenization to cases where $H$ has very
little regularity (less than continuity is needed).

\item  A well defined and common definition of the convergence of $H_{k}$
to $\overline{H}$ or $\varphi_{k}$ to $\overline\varphi$ that applies to flows, Hamilton-Jacobi equations, etc.

\item The symplectic invariance of the homogenized Hamiltonian extending the invariance results proved in \cite{Bernard} for Mather's $\alpha$ function, making his constructions slightly less mysterious.

\item The geometric properties of the function $\overline H$ (see proposition \ref{Main-properties}, \ref{Main-properties-5}),
yielding  computational methods extending those  obtained in the one-dimensional case  in \cite{L-P-V} or in other cases (see for example  \cite{Concordel}).

\end{enumerate}
\medskip

This paper will address these fundamental questions,
some other applications will be dealt with in subsequent papers (see for example \cite{NCMT} for an approach to Mather's theory in the non-convex setting).

\section{Acknowledgments}
Part of this work was done while the author was at Centre de Math\'ematiques  Laurent Schwartz, UMR 7640 du CNRS, \'Ecole Polytechnique - 91128 Palaiseau, France. Supported also by ANR projects GRST, Symplexe, Floer Power, KAMFaible, Microlocal,  and by the NSF under agreement DMS-0635607 and DMS-0603957. 

\section{Some Notation}
\begin{tabular}{l l l }
$\lambda$ &:& the Liouville form $pdq$ defined on $T^*M$\\
$ \Ham_{c}(T^{*}M)$&: &  compactly  supported time-dependent one periodic Hamiltonians,  i.e. elements \\ && in $C^{\infty}_c(  {\mathbb R}/ {\mathbb Z} \times T^*M, {\mathbb R})$.\\
$\DHam_c(T^{*}M)$&: &  set of time one flows of Hamiltonians in $ \Ham_{c}(T^*M)$.\\

 \GFQI     &: & Generating function quadratic at infinity.\\
$\mathcal{L}(T^{*}M)$&: &    the set of  images of the zero section under the action \\ &&of the Hamiltonian diffeomorphisms  in  $\DHam_c(T^{*}M)$.\\
$c(\alpha , S)$&: &  critical value obtained by minimax on $S$ with the \\ &&cohomology class $\alpha$.\\
$\gamma(L)$ and $\gamma (\varphi)$&: &  $c(\mu,L)-c(1,L)$ and $c(\mu,\varphi)-c(1,\varphi)$ the metrics\\&& on ${\mathcal L}(T^{*}M)$ and $\gclDHamc{T^{*}M}$ \\
$\gclHamc {T^{*}M}$&: &   completion for the metric $\gamma$ of     $ {\Ham}_{c}(T^{*}M)$.\\
$\gclDHamc{T^{*}M}$&: &  completion for the metric  $\gamma$ of   $\DHam_c(T^{*}M)$.\\
$c(1(x)\otimes \alpha, S)$&: &   the number $c(\alpha ,S_x)$ where $S_x(q;\xi)=S(x,q;\xi)$.\\
$\sup (f(A))$ &: & for $f : X \to {\mathbb R} $ and $A\subset X$ is defined as $\sup_{x\in A} f(x)$
\end{tabular}\\

\subsection*{Acknowledgments} This paper originated in the attempt to understand the physical phenomenon of self-adaptative (or self-tuning) oscillators described in \cite{Couderc}, that the author learned from a paper by A. Boudaoud, published in ``Images de la Physique'' (see \cite{Boudaoud}). We hope to be able to explain in a future publication the connection between Symplectic homogenization and self-adaptative phenomena in vibrating structures.
We wish to thank not only  the authors of \cite{Couderc} for writing their paper and its popularized version \cite{Boudaoud}, but also CNRS for maintaining\footnote{This journal has since disappeared, but has been replaced by ``Reflets de la physique''.} journals like ``Images de la physique''.

I  wish to thank Franco Cardin, Sergei Kuksin, and Pierre Pansu for useful discussions. Martino Bardi helped me to improve the bibliography (I am solely responsible for remaining omissions),   Vincent Humili\`ere for answering many questions concerning properties of the completions of Hamiltonian functions and maps, and pointing out several inaccuracies and finally  anonymous referees for their careful reading of the manuscript, and detailed and helpful reports.  I am grateful to  the University of Beijing for hospitality during the preparation of this work, to the organizers of the Yashafest in Stanford, and the Conference on symplectic topology in Kyoto in July 2007 for giving me the opportunity to present preliminary versions of this work.  I also warmly thank Claudine Harmide for the efficient and speedy typing of the original version.

\subsection{Note}\label{Note}
Previous versions of this paper crucially used the case $M=T^n$ of the following 
 \begin{conjecture}\label{conjecture} There exists a constant $C_n$ such that any Lagrangian submanifold $L$ of $T^*M$ contained in the unit disc bundle $\{(q,p)\in T^*M \mid \Vert p \Vert \leq 1\}$ that is
  Hamiltonian isotopic to the zero section satisfies $\gamma (L)\leq C_M$. 
\end{conjecture}
The content of section \ref{new-section} replaces the use of this conjecture in the proof of the main theorem. 
Added in revision: the conjecture has been recently  proved, in particular for $M=T^n$, by Shelukhin in \cite{Shelukhin} (see also later proofs in \cite{Guillermou-Vichery, Viterbo-support}). Also a number of papers using the present paper or its ideas
are \cite{M-V-Z}, \cite{M-Z}, \cite{Sorrentino-V}, \cite{NCMT}, \cite{Bisgaard}, \cite{Viterbo-stochastic}.

\section{A crash course on generating function metric}\label{Sec-gf-summary}
This section is devoted to defining the metric $\gamma$, stating some of its main properties and explaining the relationship with Hamilton-Jacobi equations. The reader familiar with the $\gamma$-metric may  skip this section and start directly from section \ref{sec-3}, possibly returning here for reference.

\subsection{Generating functions, the calculus of critical values and the \texorpdfstring{$\gamma$}{g}-metric}\label{sec-3.1}
Let  $M$
be an $n$-dimensional closed manifold, $L$ be a Lagrangian submanifold in $T^{*}M$
Hamiltonian isotopic to the
zero section $0_{M}$ (i.e. such that there is a Hamiltonian isotopy
$\varphi_{t}$ such that $\varphi_{1}(0_{M})=L$).

\begin{definition}\label{Def-3.1} The smooth function $S:M\times
  \mathbb{R}^{k}\to \mathbb{R}$ is a generating function quadratic at
  infinity (\GFQI     for short) for $L$ if:

\begin{enumerate} \item \label{i} there is a non-degenerate quadratic form $B$ on $\mathbb{R}^{k}$
such that
$$
\vert \nabla_{\xi}S(q;\xi) - \nabla B (\xi) \vert \leq C
$$

\item \label{ii}  zero is a regular value of the map
$$
(q;\xi) \mapsto \frac{\partial S}{\partial \xi} (q;\xi)
$$

\item \label{iii}  by \ref{i} and \ref{ii}, $\Sigma_{S}= \{ (q;\xi) \mid \frac{\partial
  S}{\partial \xi} (q;\xi)=0\}$ is a smooth compact submanifold in $M\times
\mathbb{R}^{k}$. The map
$$
i_{S} : \Sigma_{S} \to T^{*}M
$$
$$
(q;\xi) \mapsto (q, \frac{\partial S}{\partial q} (q;\xi))
$$
sends diffeomorphically $\Sigma_{S}$ to $L$.
\end{enumerate} 
\end{definition}

\begin{remarks} We point out that :
 \begin{enumerate}
\item Throughout this paper, we shall use a semicolon to separate the ``base variables'' $q$ from the
``fiber variables'', $\xi$, and we shall abbreviate ``generating function quadratic at infinity'' by ``\GFQI     ''.

\item When $k=0$, i.e. there are no fiber variables, $L$ is just the graph of the differential $dS(q)$.
 \end{enumerate}
\end{remarks}
When $L$ is Hamiltonian isotopic to the zero section,  any two \GFQI     of $L$ are equivalent by the equivalence relation generated by the following three elementary operations associating $S_{1}$ to $S_{2}$ (see
\cite{Theret}, theorem 3.2, page 254,\cite{Viterbo-STAGGF}, prop 1.5, page 688):
\begin{enumerate} 
\item \label{stab} (Stabilization) $S_{2}(x,\xi, \eta)=S_{1}(x,\xi)+ q(\eta)$ where $q$ is a non-degenerate quadratic form.
\item \label{diff}(Diffeomorphism )  $S_{2}(x,\xi)=S_{1}(x, \varphi (x,\xi))$ where $(x,\xi) \longrightarrow (x,\varphi (x,\xi))$ is a fiber-preserving diffeomorphism.
\item \label{trans}(Translation)  $S_{2}(x,\xi)=S_{1}(x,\xi)+c$.
\end{enumerate} 
 Moreover, an elementary computation using the K\"unneth isomorphism shows
that denoting by $S^{\lambda}$ the set
$$
\{ (q;\xi) \in M \times \mathbb{R}^{k} \mid S (q;\xi) \leq \lambda \}
$$
we have for $c$ large enough that
$$
H^{*} (S^{c} , S^{-c}) \simeq H^{*} (M) \otimes H^{*} (D^{-}, \partial D^{-})
$$
where $D^{-}$ is the unit disc of the negative eigenspace of
$B$. In the sequel we denote by $S^{\infty}$, (resp. $S^{-\infty}$) the set $S^{c}$ (resp. $S^{-c}$) for $c$ large enough.   Therefore, to each cohomology class $\alpha$ in
$H^{*}(M)$ we may associate the image of $\alpha \otimes
T$ ($T$ is a chosen generator of $H^{*}(D^{-}, \partial
D^{-}) \simeq \mathbb{Z}$) in $H^*(S^{\infty},S^{-\infty})$, and for $\alpha \neq 0$,  by homological minimax, a critical level $c(\alpha
, S)$ (see \cite{Viterbo-STAGGF} section 2, page 690-693).

Note that according to \cite{Viterbo-STAGGF}, if $S_{1}, S_{2}$ are related by \ref{stab} or \ref{diff}, then $H^{*}(S_{1}^{\mu}, S_{1}^{\lambda}) = H^{*}(S_{2}^{\mu}, S_{2}^{\lambda})$, while if they are related by  \ref{trans}, we have   $H^{*}(S_{1}^{\mu}, S_{1}^{\lambda}) = H^{*}(S_{2}^{\mu+c}, S_{2}^{\lambda+c})$. As a result the minimax critical levels are well-defined up to a constant shift (i.e. a shift by a constant independent from the cohomology class). 
\begin{definition}\label{Def-3.3}(see \cite{Viterbo-STAGGF}, definition 2.1)
Let $L$ be Hamiltonian isotopic to the zero section, $S$ a \GFQI     for $L$. For any non-zero cohomology  class $\alpha$ we define $$c(\alpha, S)= \inf\{\lambda \mid \text{the image of}\; T\otimes \alpha\; \text{in}\;\; H^{*}(S^{\lambda}, S^{-\infty})\; \text{is non-zero}\}.$$We denote by $1$ the generator of $H^0(M)$, $\mu$ the generator of $H^n(M)$ and set 
\begin{eqnarray*}
c_{-}(S) & = & c(1,S) \\
c_{+}(S) & = & c(\mu,S)\\
\gamma (L) & =& c(\mu, S) - c (1,S)
\end{eqnarray*}
\end{definition}

 \begin{remarks}\label{rem-3.4}
 We notice that 
\begin{enumerate} 
\item According to \cite{Viterbo-STAGGF, Theret} the numbers $c(\alpha,S)$ indeed only depends on $L$ and not on $S$ up to a global shift : replacing $S$ by $S+c$ generates the same Lagrangian and this shifts all the $c(\alpha,S)$ by $c$. 
\item \label{rem-3.4-1} Note that translating $S$ by a constant shifts
$c(\alpha, S)$ by the same constant, so that, provided we
normalized $S$ in some way, $c(\alpha, L)$ is now well defined as the common value of $c(\alpha, S)$ for $S$ a \GFQI     for $L$. For example,  if we specify the Hamiltonian $H$ yielding the isotopy between
the zero section and $L$, we may normalize $S$ by requiring that its critical values coincide with the critical values of the
action $A_{H}$ defined on ${\mathcal P}=\{ c : [0,1] \mapsto T^{*}M \mid c(t)=(q(t),p(t)), \;p(1)=0\}$ by

$$A_{H}(c)=\int_{0}^{1} \left [ p(t)\dot q(t) -H(t,q(t),p(t))\right ] dt$$
Thus $c_{\pm}(H)=c_\pm(S)$ is well defined. When $\phi^{1}$ is generated by some compactly supported Hamiltonian, we may normalize the \GFQI  so that the fixed point at infinity, which is a critical point of $S$, has critical value zero. We may thus define $c (\alpha, \phi(0_M))$ for $\phi \in {\DHam}_{c}(T^*M)$.
\item \label{rem-trivial-c} Note that when $S$ has no fiber variable, we have $c_+(S)=\sup_{x\in M} S(x), c_-(S)=\inf_{x\in M} S(x)$
 
\item \label{rem-homology-level} There is a similar definition for a homology class instead of cohomology class (see \cite{Viterbo-STAGGF}, p. 692). For $u \in H_*(M)$ we have
$$c(u, S)= \inf\{\lambda \mid \text{the image of}\; T\otimes u\; \text{is in the image of}\;\; H_{*}(S^{\lambda}, S^{-\infty}) \}.$$
\item  We shall sometimes deal with the case $M=\mathbb{R}^{n}$. Then we
need quadraticity at infinity of $S$ for both  the $\xi$ and $x$ variable, so that \ref{i} in Definition \ref{Def-3.1} should be replaced by
\begin{enumerate}
\item [(1')] there exists a nondegenerate quadratic form $B(q,\xi)$ on $M \times
\mathbb{R}^{k} ( = \mathbb{R}^{n}\times \mathbb{R}^{k})$ such that 
$$
\vert \nabla S (q;\xi) - \nabla B (q;\xi) \vert \leq C\dispdot
$$
\end{enumerate} 
The map $\gamma$ is well defined on the set $\mathcal{L}(T^{*} {\mathbb R} ^{n})$ of
Lagrangian submanifolds Hamiltonian isotopic to the zero section, where the Hamiltonian is assumed to be  compactly supported by compactifying $L$ and $ {\mathbb R}^n$ to $S^n$ so that we set $\gamma (L)\overset{def}=\gamma(L\cup \{(\infty,0)\})$ where $$L\cup \{(\infty,0)\} \subset T^*S^n=T^* {\mathbb R}^n \cup ( \{\infty\} \times ({\mathbb R}^n)^*) \dispdot $$
\end{enumerate} 
 \end{remarks}

It follows from \cite{Viterbo-STAGGF},  that  $\gamma$  defines a metric on $\mathcal{L}(T^{*}M)$ by
setting

\begin{definition}\label{Def-4.5} Let $S_{1}$ and $S_{2}$ be \GFQI     for $L_{1}$ and $L_{2}$ in ${\mathcal L}(T^{*}M)$. Then  $\gamma (L_{1},L_{2}) = c(\mu , S_{1} \ominus S_{2})
  - c (1,S_{1}\ominus S_{2})$ where $(S_{1}\ominus S_{2}) (q; \xi_{1},\xi_{2}) = S(q;\xi_{1})
  - S (q; \xi_{2})$. The function $\gamma$ defines a metric on  ${\mathcal L}(T^{*}M)$.
  We denote by $\widehat {\mathcal L}(T^*M)$ the completion\footnote{Usually called the Humili\`ere completion, see \cite{Humiliere}.} of ${\mathcal L}(T^{*}M)$ for $\gamma$. 
\end{definition}

That $\gamma$  is indeed a metric on $\mathcal L(T^*M)$  is a
consequence of Lusternik-Shnirelman's theory (see \cite{Viterbo-STAGGF}).
Note that $\gamma(L)=\gamma(L, 0_{M})$ so our use of $\gamma$ is a slight abuse of notation. 

\medskip

Our goal  is to define a metric on $\Ham_{c} (T^{*}M)=C_{c}^{\infty}([0,1]\times T^*M, {\mathbb R})$ the set of compactly supported, time dependent
Hamiltonians of $T^{*}M$, and on ${\DHam}_{c}(T^{*}M)$ the group of time-one maps of Hamiltonians in $\Ham_{c} (T^{*}M)$. 
\label{Define-j}  For $M = T^{n}$, the graph of  $\varphi \in \DHam_c(T^*T^n)$ ,
$$
\Gamma (\varphi) = \{ (z, \varphi(z)) \mid z \in T^{*}T^{n}\}
$$
is a Lagrangian submanifold of $T^{*}T^{n}\times \overline{T^{*}
  T^{n}}$ (where $\overline{T^{*}M}$ is $T^{*}M$ with the symplectic
form of opposite sign~: $-dp \wedge dq$).

But $T^{*}T^{n} \times \overline{T^{*}T^{n}}$ is covered by
$T^{*}(\Delta_{T^{*}T^{n}})$,  where $\Delta_{T^{*}T^{n}}$ is the
diagonal, through the symplectic covering map 
\begin{align*}j: & \; T^*(T^n\times {\mathbb R} ^n) \longrightarrow T^*T^n \times \overline {T^*T^n} \\
 & (u,v,U,V) \longrightarrow (u-V,v,u,v-U) \end{align*}

Here $(q,p)\in T^*T^n, (Q,P) \in \overline {T^*T^n}$ and  the symplectic form is $$dp\wedge dq-dP\wedge dQ$$ while $(u,v) \in T^n\times {\mathbb R} ^n$, $(u,v,U,V) \in T^*(T^n\times {\mathbb R} ^n) $ with symplectic form $dU\wedge du+ dV\wedge dv$. 

The inverse of $j$,  $$(q,p,Q,P) \longrightarrow (q,P,p-P,Q-q)=(u,v,U,V)$$ is not well-defined : since $Q,q$ are only defined modulo $ {\mathbb Z}^n$, so is $V$. It is thus multivalued, but we may  lift $\Gamma(\varphi)$ to
$\widetilde\Gamma(\varphi) \subset j^{-1}(\Gamma(\varphi))$, which is now a Lagrangian submanifold in
$T^{*}(\Delta_{T^{*}T^{n}})$.

In other words it $\varphi^t$ is a Hamiltonian isotopy of $T^*T^n$ and $\widetilde \varphi^t$ the  lift to $T^* {\mathbb R}^n$ such that $\widetilde \varphi^0=\Id$, we have, setting $[q]$ to be the class of $q\in {\mathbb R}^n$ in $T^n$, 
$$\widetilde \Gamma (\varphi)=\left\{ ([q],P, p-P,Q-q) \in T^*(T^n\times {\mathbb R}^n) \mid \widetilde \varphi(q,p)=(Q,P)\right\}$$

When $\varphi$ has compact support, we may compactify both
$\widetilde{\Gamma}(\varphi)$ and $\Delta_{T^{*}T^{n}}$ and we get a
Lagrangian submanifold $\overline\Gamma(\varphi)$ in $T^{*}(T^{n}\times
S^{n})$. We may then set

\begin{definition} [\cite{Viterbo-STAGGF}, page 697]
For $M=T^n$, the maps $c_-,c_+$ and $\gamma$ are defined by
$$c_-(\varphi)=c (1\otimes 1 , \overline\Gamma(\varphi))$$
$$c_+(\varphi)=c ( \mu_{T^{n}} \otimes \mu_{S^{n}} , \overline\Gamma
(\varphi))
$$
and
$$
\gamma (\varphi) = \gamma(\overline\Gamma
(\varphi))= c ( \mu_{T^{n}} \otimes \mu_{S^{n}} , \overline\Gamma
(\varphi)) - c (1\otimes 1 , \overline\Gamma(\varphi))\dispdot
$$
We also set 
$$c_\pm(\varphi, \psi)=c_\pm(\varphi \psi^{-1})$$ and $$\gamma(\varphi, \psi)=\gamma (\varphi\psi^{-1})$$
 \end{definition}

\begin{proposition}\label{Prop-3.9} (see \cite{Viterbo-STAGGF} )

The map $\gamma$ defines a bi-invariant metric on ${\DHam}_{c}(T^{*}T^n)$
since

\begin{enumerate} 
\item it is non-degenerate $\gamma(\varphi)=0 \Longleftrightarrow
\varphi = id$
\item it is invariant by conjugation $\gamma(\psi \varphi \psi^{-1})
= \gamma(\varphi)$ for any $\psi$ in $\DHam_c(T^{*}T^n)$.
\item it satisfies the triangle inequality
$$
\gamma (\varphi\psi) \leq \gamma(\varphi)+\gamma (\psi)
$$
for any $\varphi , \psi$ in $\DHam_c(T^{*}T^n)$.
\end{enumerate} 
\end{proposition}
\begin{proof}  
Then according to \cite{Viterbo-Montreal}, proposition 2.11, $\gamma $ is a distance on the set of Lagrangian
 Hamiltonian isotopic to the zero section. We apply this to $\overline\Gamma(\varphi) \subset T^*(T^n\times S^n)$. Property $(1)$ follows from the non-degeneracy of the 
 metric on $\mathcal{L}(T^{*}(T^n\times S^n))$. 
For property $(2)$ the proof is identical to the proof of corollary 4.3 in \cite{Viterbo-STAGGF} for the case $M= {\mathbb R}^n$).  
 The last property follows easily from the triangle inequality  from corollary 3.6 page 696 of \cite{Viterbo-STAGGF}. 
\end{proof} 
\begin{remarks} 
\item The metric $\gamma$ can be extended on one hand to $\DHam_c(T^*M)$ by setting $\widehat \gamma (\varphi)=\sup \{ \gamma (\varphi(L),L) \mid L\in {\mathcal L} (T^*M)\}$ and on the other hand to general symplectic manifolds (see \cite{Oh-spectrum1, Schwarz}) using Floer cohomology instead of generating function homology. 
\item A vector field $Z$ is called a Liouville vector field, if 
$Z$ is conformal (i.e. the flow $\psi_{t}$ of $Z$
satisfies $\psi_{t}^{*} \omega = e^{t}\omega$) . We then have, according to \cite{Viterbo-STAGGF} (corollary 4.3 page 698)
 \begin{equation} \label{conf}
\gamma (\psi_{t}\varphi   \psi^{-1}_{t}) = e^{t} \gamma (\varphi) \dispdot
\end{equation}
\end{remarks} 

On the set  $\Ham_{c}(T^*T^n)$ the  metric $ \gamma$ is defined as follows\footnote{Remember that elements of $\Ham_{c}(T^*M)$  are time-dependent Hamiltonians.}

\begin{definition}
Let $H(t,z)$ be a Hamiltonian in $\Ham_{c}(T^*M)$ with flow $\varphi_{H}^t$. We denote by 
$$\gamma (H,K)=\sup \left\{\gamma (\varphi_{H}^t \circ(\varphi_{K}^t)^{-1}) \mid t \in [0,1] \right\}
$$

\end{definition}
Finally we state two convergence  criteria for  the
$\gamma$-metric.

\begin{proposition}\label{Prop-3.12} Let $M = \mathbb{R}^{n}$ or $T^{n}$.
Let  $\varphi_{H}, \varphi_{K}$ be the time-one maps of the flows associated to $H,K$ in $\Ham_{c}(T^*M)$. We have $$ \gamma (H,K) \leq \Vert H-K \Vert _{C^0([0,1]\times T^{*}M, {\mathbb R}) }\dispdot $$ As a result, if the sequence $(H_{k})_{k\geq 1}$ of Hamiltonians on $T^{*}M$ with fixed
support, $C^{0}$-converges  to $H$,  then $(H_{k})_{k\geq 1}$ converges, in the
metric $\gamma$, to $H$.

Similarly we have a constant $C$ such that 
$$\gamma (\varphi, \psi) \leq C\cdot d_{C^0}(\varphi, \psi) \dispdot$$
\end{proposition}

 \begin{proof} The first statement  follows immediately from \cite{Viterbo-STAGGF}
(Proposition 4.6, page 699 and Proposition 4.14 page 707). It is stated explicitly for example as proposition
2.15 in \cite{Viterbo-Montreal} or 
in \cite{Humiliere} p. 378 proposition 2.4, d). 
The second one follows for Hamiltonian maps supported in the unit disc bundle from Theorem 5 in \cite{Seyfaddini}. The general case follows from the homogeneity (by the dilation $t\cdot (q,p)\mapsto (q,t\cdot p)$) of both sides. 
\end{proof}

We may therefore define the
completion $\gclHamc{T^{*}M}$ of ${\Ham}_{c}(T^{*}M)$ for $\gamma$, as Humili\`{e}re did (see \cite{Humiliere}, section 4, page 388).  For example $\gclDHamc{T^{*}M}$  is the set of equivalence classes of Cauchy limits  for $\gamma$ of sequences of elements of $\DHam_c (T^*M)$, two sequences being equivalent if their $\gamma$-distance converges to $0$. For $\gclDHamc{T^{*}M}$ we assume the sequence is supported in a fixed compact set. Note that the support of an element of  $\gclDHam{T^{*}M}$ can be defined directly (see \cite{Humiliere}), so this is the same as the subset of elements of $\gclDHam{T^{*}M}$ having compact support. 
From the above proposition we deduce that a sequence of Hamiltonians $C^{0}$-converging uniformly on compact sets 
 will be a Cauchy sequence for $\gamma$, hence defines an
element in $\gclHamc{T^{*}M}$. We thus get

\begin{proposition}[see \cite{Humiliere}, proposition 1.3] \label{Prop-3.13}
There is a continuous inclusion map
$$
C_{c}^{0} ([0,1] \times T^{*} M , \mathbb{R}) \to \gclHamc{T^{*}M}\dispdot
$$
Similarly if $\overline{\DHam_c(T^{*}M)}$ is the $C^{0}$-closure of
${\DHam}_{c}(T^{*}M)$ we have a continuous  inclusion
$$
\overline{\DHam_c(T^{*}M)}\to \gclDHamc{T^{*}M} \dispdot
$$
\end{proposition}

Finally, we claim that the spectral numbers $c(\alpha , L)$ are well defined on the $\gamma$-completions of the above metric spaces, for example in  $\gclDHamc{T^{*}M}$.

\begin{proposition} \label{Prop-3.14}
For $\alpha \in H^*(T^n\times S^n)\setminus \{0\}$  the map $\varphi \mapsto c(\alpha, \varphi)$ uniquely extends as a continuous map (for the metric $\gamma$) defined on  $\gclDHamc{T^{*}T^{n}}$.
\end{proposition} 
\begin{proof} 
Since   $\vert c(\alpha,\varphi_1) - c(\alpha, \varphi_2) \vert  \leq \gamma( \varphi_1, \varphi_2)$ according to Appendix \ref{Appendix-B}, Proposition \ref{Prop-B3}, this is just an application of the general statement that a Lipschitz map defined on a metric space has a unique extension to its completion. 
\end{proof} 

\subsection{Variational solutions of Hamilton-Jacobi equations}\label{Sect-3.2}

Let  $S(x,y;\xi)$ be a \GFQI     , where $(x,y)\in X\times Y$. Define $S_y(x,\xi)=S(x,y;\xi)$ and for $\alpha \in H^*(X)$ set $c(\alpha\otimes1(y),S)=c(\alpha, S_y)$. The notation indicates that $H^*(X\times \{y\})=H^*(X) \times H^*(\{y\})$ and $H^*(\{y\})$ is one-dimensional generated by an element denoted $1(y)$. 

Let $\phi^{t}$ be the Hamiltonian flow of $H(q,p)$, and ${\widetilde \Gamma} (\phi^{t})$ the lift of its  graph in $T^*(T^n\times {\mathbb R}^n)$ as in Definition \ref{Def-4.5}. Let $S_{t}(q,P,\xi)$ be a \GFQI      for $\widetilde\Gamma(\phi^{t})$. 
We thus define $c(1(q,P), S_{t})=c(1,S_{t,q,P})$. Then $u_{t}(q,P)=c(1(q,P),S_{t})$ is, by definition, the  variational solution of

 \begin{gather*}\left \{ \begin{array}{ll} \frac{\partial}{\partial t}u_{t}(q,P)+H(q,P+\frac{\partial}{\partial q}u_{t}(q))=0 \\
 u_{0}(q,P)=0 \dispdot \end{array}\right.
 \end{gather*}

 We refer to \cite{Viterbo-Ottolenghi}, \cite{Viterbo-HJ} and \cite{Cardin-Viterbo} for more information on variational solutions, in particular the fact that it does not depend on the choice of $S$ and Appendix 2 (theorem 13.1) of \cite{NCMT} for the proof that variational solutions satisfy the equation outside a closed set of  zero measure. 
 \section{Statement of the main results}\label{sec-3}~
We shall first give our results in the case of homogenization with respect to all variables, then present the case of partial homogenization, and finally the 
applications to variational solutions of Hamilton-Jacobi equations.

\subsection{Standard homogenization}

\begin{theorem}[Main theorem]\label{Main-theorem} \hskip 0pt

Let $H(t,q,p)$ be a compactly supported,  smooth Hamiltonian, $1$-periodic in $t$ on  $T^*T^n$. Then the following holds:
 \begin{enumerate}
\item \label{Main-theorem-1} There exists a Hamiltonian $\overline H \in C_{c}^0({\mathbb R}^n, {\mathbb R} )$  such that  the sequence $H_{k} (t,q,p) = H(kt,kq,p)$ $\gamma$-converges to
$\overline{H}(t,q,p)= \overline H(p)$.

\item \label{Main-theorem-2} The function $\overline H$ only depends on $\varphi^1$, the time-one map associated to $H$
(i.e. it does not depend on the isotopy $(\varphi^t)_{t\in [0,1]}$).

\item \label{Main-theorem-3}The map
$$
\mathcal{A} : C_{c}^{\infty} ({\mathbb R}/ \mathbb Z \times T^{*} T^{n} , \mathbb{R}) \to C_{c}^{0}
(\mathbb{R}^{n},\mathbb{R})
$$
defined by $\mathcal{A}(H) = \overline{H}$ extends to a nonlinear projector (i.e.  a surjective map satisfying ${\mathcal A}^2= {\mathcal A}$) with Lipschitz constant 1
$$
\mathcal{A} : \gclHamc { T^{*} T^{n}} \to
C_{c}^{0}(\mathbb{R}^{n},\mathbb{R})
$$
where the metric on $\gclHamc{ T^{*} T^{n}}$ is given by $\gamma$, and the metric on
$C^{0}(\mathbb{R}^{n},\mathbb{R})$ is the $C^{0}$-metric.

Moreover $\mathcal A$ sends Lipschitz Hamiltonians to Lipschitz maps. 
\end{enumerate}
\end{theorem}

The next theorem states  some properties of the homogenization map $\mathcal{A}$.

 \begin{theorem}[Main properties of symplectic homogenization]  \label{Main-properties} \hskip 0pt

 The map $\mathcal A$ defined in the above theorem satisfies the following properties:
 \begin{enumerate}
\item \label{Main-properties-1}It is monotone, i.e. if $H_1 \leq H_2$, then ${\mathcal A}(H_1) \leq {\mathcal A}(H_2)$.

 \item \label{Main-properties-2} It is invariant under the action of a  Hamiltonian symplectomorphism: \newline${\mathcal{A}} (H\circ \psi)= {\mathcal{A}}(H)$ for all $\psi \in \DHam (T^*T^n)$. 

 \item \label{Main-properties-3}  Setting $H_c(t,q,p)=c\cdot H(c\cdot t , q,p)$, we have ${\mathcal A}(H_c)=c{\mathcal A}  (H)$ for any $c\in {\mathbb R}$. In particular if $H$ is autonomous, $\mathcal A (c H)= c \mathcal A( H)$.
 \item \label{Main-properties-4} The map $\mathcal{A}$ extends to  a map (still denoted by $\mathcal{A}$) between ${\mathcal P}(T^*T^n)$, the set of subsets of $T^*T^n$,  to $ {\mathcal P}({\mathbb R}^n)$, the set of subsets of ${\mathbb R}^n$. This map is bounded from below by the symplectic shape of Sikorav (see \cite{Benci, Sikorav,Eliashberg}), i.e. $$\mathop{shape}(U)=\{p_0\in {\mathbb R}^n \mid \exists \psi \in\DHam_c ( T^{*} T^{n}), \psi(T^n\times \{p_0\}) \subset U \}
  \subset {\mathcal{A}}(U).$$
 \item \label{Main-properties-5}  If $L$ is a Lagrangian Hamiltonian isotopic to $L_{p_{0}}=\{(q,p_{0})\in {\rm T^*T^{n}}\} $ and $\sup_{(q,p)\in L} H(q,p)\geq h$ (resp. $\inf_{(q,p)\in L} H(q,p)\leq h$) we have ${\mathcal A}(H)(p_{0})\geq h$ (resp. $\leq h$).
 \item \label{Main-properties-6} We have $$\lim_{k\to \infty} \frac{1}{k} c_{+}(\varphi ^k)= \sup_{p\in {\mathbb R} ^n} {\overline H}(p)$$
$$\lim_{k\to \infty} \frac{1}{k} c_{-}(\varphi ^k) =\inf_{p\in {\mathbb R} ^n} {\overline H}(p).$$
\item \label{Main-properties-7} For any sequence of non-negative compactly supported functions, $(H_n)_{n\geq 1}$, converging uniformly to $1$ on compact sets, we have  $\lim_{n}\zeta (H_{n})=1$.

 \item \label{Main-properties-8} Given any Radon measure $\mu$ on ${\mathbb R}^n$ the map $$\zeta (H)= \int_{{\mathbb R}^n} {\mathcal{A}}(H)(p) d\mu(p)$$ satisfies all the properties of a {\bf symplectic quasi-state}\footnote{see \cite{E-P} for the definition and properties of quasi-states in the symplectic framework, inspired by  \cite{Aarnes}.} except for normalization (i.e. $\zeta(1)=1$) which is however satisfied in a weak sense according to  \ref{Main-properties-7}. In particular we have ${\mathcal A}(H+K) ={\mathcal A}(H)+ {\mathcal A}(K)$ whenever $H$ and $K$ Poisson-commute (i.e. $\{H,K\}=0$).
 \end{enumerate}
 \end{theorem}

 \begin{remarks}\label{rem-4.3} \begin{enumerate}  
 Here are some comments:
 \item The function $\overline H$ will be defined in the autonomous case as $\overline H(p)=\lim_k c(\mu\otimes 1(p), H_k)$. Of course there is a lot to prove, starting from the existence of this limit. 
 
 \item   In \ref{Main-properties-1} the assumption could be replaced by the property that $H_1 \preceq H_2$ in the sense of \cite{Viterbo-STAGGF} (i.e. $c_{-}(\varphi_{H_{1}}^{-1}\circ \varphi_{H_{2}}^1)=0$, see \cite{Viterbo-STAGGF}, definition 4.9, page 701).
\item \label{rem-4.3-3} Note that it is not true that characteristic functions belong to $\gclHamc{ T^{*} T^{n}}$, for example an integrable Hamiltonian is in the completion if and only if it is continuous, so 
for example for $U=S^1\times [-1,1]$, $H_U$ is not in $\gclHamc{ T^{*} T^{n}}$. However $\mathcal A$ can be extended to any $H$ that is a limit of a decreasing sequence of continuous functions -i.e. any upper semi-continuous function-by setting   $H=\lim_{k} H_k$, by setting $\overline H= \lim_{k} \overline H_k$. It is easy to show that this does not depend on the choice of the sequence $(H_k)$. So $\mathcal A$ extends to the class of upper semi-continuous functions. However it is not clear what properties do still hold in such situation, since for example $-H$ is not upper semi-continuous. 

 \item Property \ref{Main-properties-3} is essentially trivial for $c>0$. The non-trivial fact is $\mathcal A (-H)= - \mathcal A( H)$ (see Remark \ref{rem-4.7}).
 \item As a result of \ref{Main-properties-5} if
 $u$ is a smooth subsolution of the stationary Hamilton-Jacobi equation, that is $H(x, p+du(x))\leq h$, then $\overline H (p) \leq h$. Similarly if $u$ is a smooth supersolution, that is $H(q,p+du(q))\geq h$, then
 $\overline H(p) \geq h$.
 \item From \ref{Main-properties-5}, we get the following statement: let $$E_{c}^+=\{p_{0} \in {\mathbb R}^n \mid \exists L \; \text{Hamiltonian isotopic to}\; L_{p_{0}}, \inf_{(q,p)\in L} H(q,p)\geq c\}$$
 $$E_{c}^{-}=\{p_{0} \in {\mathbb R}^n \mid \exists L \; \text{Hamiltonian isotopic to}\; L_{p_{0}}, \sup_{(q,p)\in L} H(q,p)\leq c\}$$
As a result,  if $p \in \overline E_{c}^{+} \cap \overline E_{c}^{-}$, we have $\overline H(p)=c$.  
\end{enumerate} 
  \end{remarks}
More generally we have the following
\begin{corollary} 
 Let $H$ be an autonomous Hamiltonian. Then if $c=\overline H(p_0)$, then $H^{-1}(c)$
 intersects all images of $L_{p_0}$ by a Hamiltonian map.
 \end{corollary} 
 \begin{proof} 
 Indeed, if this was not the case we could find a Hamiltonian image of $L_{p_0}$ which is contained in either $H<c$ or $H>c$. In the first case, this implies $\overline H(p_0) >c$ in the second $\overline H(p_0) <c$, contradicting our assumption.  
 \end{proof} 

 \subsection{Partial homogenization}

 We now consider the case of a Hamiltonian $H$  defined on $T^*T^n\times M$, where $M$ is some symplectic manifold. We shall only deal with the case where $M=T^*T^m$, but the general case can be easily adapted. We want to understand the limit of $H_k(t,x,y,z)=H(kt, kx,y,z)$ for $(x,y)\in T^*T^n, z \in M$. 
Indeed, we shall show that  it is sufficient  to define the homogenization for  $\varphi^t$ for small $t$. But denoting by $\varphi_k^t$ the flow of $H_k$,  the graph of $\varphi_{k}^t$ for $t$ small  lives in a neighborhood of $T^*(T^n\times {\mathbb R}^n)\times \Delta_{M}$, hence in
$T^*(T^n\times {\mathbb R}^n)\times T^*\Delta_{M}$. Since the graph lives in a cotangent bundle, we shall see that it can again be described, using  the theory of generating functions\footnote{Indeed, the main advantage of $T^{n}$ over general closed manifolds, is that $T^{*}T^{n}\times \overline {T^{*}T^{n}}$ is covered by  $T^{*}\Delta$, while for general $M$ this only holds in a tubular neighborhood of the diagonal.}. 
We then have the following extension of Theorem \ref{Main-theorem}, which corresponds to the case $m=0$: 
 \begin{theorem} [Main theorem, partial homogenization case] \label{Thm-4.5}
 Let $H(t,x,y,q,p)$ be a compact supported Hamiltonian on $T^*T^{n+m}$.
 Then \begin{enumerate}
  \item \label{Thm-4.5-1} The sequence $(H_{k})_{k\geq 1}$ defined by $$H_k(t,x,y,q,p)=H(kt,kx,y,q,p)$$ $\gamma$-converges to
a continuous function $\overline H$ of the form $\overline H (y,q,p)$.
 \item \label{Thm-4.5-2}The map $${\mathcal A}_x : C_{c}^{\infty} ([0,1]\times T^{*} T^{n+m} , \mathbb{R}) \to C_c^{0}
(\mathbb{R}^{n}\times T^*T^m,\mathbb{R})$$
given by $\mathcal{A}_{x}(H) = \overline{H}$ extends to a projector (i.e. it is surjective and  satisfies ${\mathcal A}_{x}^2= {\mathcal A}_{x}$) with Lipschitz constant 1
$$\mathcal{A}_x:  \gclHamc {T^*T^{m+n}}\to
 C_c^0({\mathbb{R}^{n} \times T^*T^m)}$$
where the metric on $\gclHamc {T^*T^{m+n}}$ is $\gamma$.
\item \label{Thm-4.5-3} If $H_{(q,p)}(x,y)=H(x,y,q,p)$, we have $$\mathcal {A}_{x}(H) (y,q,p)= \mathcal {A}(H_{(q,p)})(y)$$
Thus partial homogenization is obtained by freezing the non-homogenized variables. 
 \end{enumerate}
   \end{theorem}

 \begin{remarks} \begin{enumerate}
 \item In \ref{Thm-4.5-2} we identify $C_c^{0}(\mathbb{R}^{n}\times T^*T^m,\mathbb{R})$ to an element in $\gclHamc {T^*T^{m+n}}$ as we did in the previous section.
\item  The Hamiltonian $\overline {H}(y,q,p)$ is called the effective Hamiltonian. In case it is smooth, its flow is given by ${\overline \Phi}(x_{0},y_{0},q_{0},p_{0})=(x(t),y(t),q(t),p(t))$
   \begin{gather*} y(t)=y_{0},\; x(t)=x_0+\int_{0}^{t} \frac{\partial \overline H}{\partial y}(y_{0},q(t),p(t)) dt, \\ \dot q (t)= \frac{\partial \overline H}{\partial p}(y_{0},q(t),p(t)), \; \dot p (t)= -\frac{\partial \overline H}{\partial q}(y_{0},q(t),p(t))
 \end{gather*}
  \item It is not true anymore that $\overline H$ depends only on the time-one map of $H$. 
  \item More generally, using Theorem \ref{Thm-4.5} \ref{Thm-4.5-3}, we may prove  properties of  $\mathcal A_{x}$  analogous to the properties of $\mathcal A$ stated in Theorem \ref{Main-theorem}. The projector ${\mathcal A}_{x}$ is not invariant by symplectic maps. It is however invariant by fiber-preserving Hamiltonian symplectic maps: if $\psi(x,y,q,p)=(\psi_{(q,p)}(x,y), \psi_{2}(q,p))$ we have $${\mathcal A}_{x}(H\circ \psi)(y,q,p)={\mathcal A}_{x}(H)(y,\psi_{2}(q,p))$$
\end{enumerate} \end{remarks}
\subsection{Homogenized Hamilton-Jacobi equations}
Our theorem has some interesting applications to generalized
solutions of evolution Hamilton-Jacobi equations. Consider the equation:
\begin{equation} \label{Eq-HJ}
\tag{$HJ$}
\left\lbrace
\begin{array}{ll}
& \frac{\partial}{\partial t} u (t,q) + H (t,q, \frac{\partial}{\partial
q}u (t,q)) = 0\\
& u (0,q) = f(q)\dispdot
\end{array}
\right.
\end{equation} 
where $t\in \mathbb{R} , q \in T^{n}$ and $H\in C^\infty( {\mathbb R}/ \mathbb Z \times T^*T^n) $.

\medskip

Smooth solutions to such equations are only defined for $t$ less than
some $T_{0}$. In general, solutions exhibit shocks: $\left \Vert D^2u(q)
\right \Vert_{C^{0}([0,T]\times T^{n},\mathbb{R}, {\mathbb R} )}$ blows-up as $t$ goes to
$T_{0}$.

\medskip

There are essentially two types of generalized solutions for such
equations~: viscosity solutions (cf. \cite{Crandall-Lions}, \cite{Barles}\cite{Bardi-Capuzzo-Dolcetta}) and variational
solutions (cf. \cite{Sikorav-pc}, \cite{Chaperon1}, \cite{Viterbo-Ottolenghi}, \cite{Viterbo-HJ}). These two solutions do not coincide in
general, with one notable exception: when the Hamiltonian is
convex in $p$ (cf. \cite{Zhukovskaia-2}, \cite{Roos, WQ1}).

\medskip

From \cite{L-P-V} it follows that if $H$ is coercive in $p$, and $u_{k}$ is
the viscosity solution  of
$$
(HJ_{k})
\left\lbrace
\begin{array}{ll}
& \frac{\partial}{\partial t} u_{k} (t,q) + H (kt, kq,
\frac{\partial}{\partial
q} u_{k}(t,q)) = 0\\
& u_{k} (0,q) = f(q).
\end{array}
\right.
$$
the sequence $(u_{k})_{k\geq 1}$ converges to $\overline u$, the viscosity solution of
$$
(\overline{HJ})
\left\lbrace
\begin{array}{ll}
& \frac{\partial}{\partial t} \overline{u} (t,q) + \overline{H}
(\frac{\partial}{\partial q} \overline{u} (t,q)) = 0\\
& \overline{u} (0,q) = f(q).
\end{array}
\right.
$$
Our theorem, together with results by  Humili\`ere (cf. \cite{Humiliere} section 6, in particular Proposition 6.1) implies that this extends to the non-coercive case, provided
$u_{k}$ is the variational solution and $\overline{H}$ is given by our
main theorem.
We now state the more general proposition, yielding the analog of \cite{L-P-V} :
\begin{proposition} \label{cor-3.5}
Let $H \in C^{0}([0,1]\times T^{*} T^{n+m}, {\mathbb R} )$ be either coercive (i.e. $\lim_{ \vert (y,p) \vert \to \infty} H(t,x,y,q,p) = +\infty$ ) or compactly supported, $f\in
C^{0}(T^{n+m}, {\mathbb R} )$ and $u_{k}$  the variational solution of ($HJP_{k}$):
 \begin{gather*}
 \tag{$HJP_{k}$}
\left\lbrace
\begin{array}{ll}
& \frac{\partial}{\partial t} u_{k} (t,x,q) + H (kt,kx,q,
\frac{\partial}{\partial
x} u_{k}(t,x,q), \frac{\partial}{\partial
q} u_{k}(t,x,q)) = 0\\
& u_{k} (0,x,q) = f(x,q).
\end{array}
\right.
\end{gather*}
where $(x,q)\in T^n \times T^m$
Then $\mathop{\lim}\limits_{k\to+\infty} u_{k}(t,x,q) = \overline{u} (t,x,q)$ where
convergence is uniform on compact time intervals and $\overline u$ is the variational
solution of $(\overline{HJP})$.
 \begin{gather*}
\tag{$\overline{HJP}$}
\left\lbrace
\begin{array}{ll}
& \frac{\partial}{\partial t} \overline u (t,x,q) + \overline H (q,
\frac{\partial}{\partial
x} \overline u(t,x,q), \frac{\partial}{\partial
q} \overline u_{k}(t,x,q)) = 0\\
& \overline u (0,x,q) = f(x,q).
\end{array}
\right.
\end{gather*}

More precisely, for $(x,q)$ in a bounded set, there is a sequence $ \varepsilon _{k}$ going to zero, such that for all $t\geq 1$
$$ \vert u_{k}(t,x,q)- {\overline u}(t,x,q) \vert \leq \varepsilon_{k} t\dispdot $$

\end{proposition}
The next three sections will be devoted to the proof of our main
theorem, first in the ``standard case'', then in the ``partial homogenization'' setting.

\begin{remark} \label{rem-4.7}
Note that according to  \ref{Main-properties-3} in Theorem \ref{Main-properties}   we have ${\mathcal A}(-H)= -{ \mathcal A }(H)$. This is a statement that typically   does  not  hold in the case of viscosity solutions, since if $u(t,x)$ is a viscosity solution associated to $H$, $u(-t,x)$ is not in general a viscosity solution associated to $-H$. 
\end{remark} 
\section{Proof of the main theorem}\label{Main}

Let us introduce the reader to the main steps of the proof.
We denote by $\phi_{k}^t$ the flow of $H_k(t,q,p)=H(kt, kq,p)$ for $H \in \Ham_c(T^*T^n)$. In the first part of subsection \ref{subsec:reform},  we shall construct a \GFQI     of the graph of the $\phi_{k}^t$, starting from a \GFQI     of the graph of the  flow $\phi^{t}=\phi_{1}^t$.

Our proof will then be  split in two parts:

\begin{itemize}
\item Finding a candidate $\overline \phi^{\; t}$ for the $\gamma$-limit of $\phi_{k}^{\; t}$, or equivalently a candidate $\overline H$ for the $\gamma$-limit of $H_k$

\item  Showing that the $\gamma$-limit of $\phi_{k}^{t}$ is indeed $\overline \phi^{\; t}$.

\end{itemize}

Remember from Section \ref{Sect-3.2} that if $S(q,p,\xi)$ is a \GFQI     for $\Gamma (\varphi)$, $c(\alpha \otimes 1(p), \varphi)$ is defined as $c(\alpha, S_p)$ where $S_p(q,\xi)=S(q,p,\xi)$. 

The first step goes along the following lines: if $H$ does  not depend on  $(t,q)$, then $c(\mu_{q}\otimes 1(p),\phi^1)=H(p)$, so\footnote{Remember that according to Proposition \ref{Prop-3.14}, spectral numbers are well defined on $\gclDHamc {T^*T^n}$.} if $H_k$ is a sequence of autonomous Hamiltonians $\gamma$-converging to $\overline H$, we must have that  $$\lim_{k\to \infty }c(\mu_{q}\otimes 1(p),\phi_{k}^1)={\overline H}(p)\dispdot $$  We shall thus try to define $\overline H$ using this formula, and we shall first prove that this limit exists. This is the object of the second part of subsection \ref{subsec:reform} and is proved in proposition \ref{4.6}. We thus get a candidate ${\overline\varphi}^{1}$ for the $\gamma$-limit of $\phi_{k}^1$.

The second step is more delicate, and is dealt with in subsection \ref{eop}. The formula obtained for the \GFQI     of $\phi_{k}^{t}$ yields the following inequality   valid for any Hamiltonian map $\alpha$

$$\lim_{k \to \infty}\inf c (\mu\otimes 1 , \varphi_{k} \alpha ) \leq c (\mu\otimes 1 ,\overline\varphi
\alpha )$$ and proved in proposition \ref{Prop-5.15}.

We must then prove the reverse inequality. This requires us to use the previous inequality with $\varphi^{-1}$ instead of $\varphi$, the difficulty being that we do not know {\it a priori} that  $\overline{\varphi^{-1}}=\overline \varphi^{-1}$. The proof of this equality is the object of section \ref{new-section}. 

Finally note that we do not know\footnote{In the case of continuous Hamiltonians and flows, i.e. the case of the group $Hameo$, this is proved in \cite{Viterbo-unique} (see also \cite{B-S}). I owe this remark to Vincent Humili\`ere and Nicolas Vichery.}, in  general, whether the  convergence of a sequence  of compactly supported Hamiltonians $(H_{k})_{k \geq 1}$ follows from the convergence of
the sequence of flows $(\varphi_{k}^t)_{k\geq 1}$. Indeed, it could happen that in the completion $\gclDHamc{ T^*T^n}$,  the same family of maps $\overline\varphi^t$, is obtained from two different Hamiltonians i.e. is the image of two different elements of $\gclHamc{T^*T^n}$ by the extension to the completions of the exponential map $H \longrightarrow \varphi_{H}^t$. However, two distinct compactly supported continuous integrable Hamiltonians $H_{1}(p), H_{2}(p)$ cannot have the same flow, as is proved in corollary \ref{Cor-app-A}  in Appendix \ref{Appendix-A} (this also follows from a much more general theorem of Humili\`ere, Leclercq and Seyfaddini, see \cite{H-L-S}).

\subsection{Reformulating the problem and finding the homogenized
  Hamiltonian} \label{subsec:reform}

First of all, we shall assume we are dealing with a compactly supported  autonomous
Hamiltonian $H\in \Ham_c (T^*T^n)$. We shall see in subsection \ref{naut}  that the general case
reduces to this one.

Similarly let $\varphi^{t}_{k}$ be the flow associated to $H_{k}(q,p)=
H(kq,p)$ set $\varphi_{k} = \varphi_{k}^{1}$ and $\varphi=\varphi_1^1$. 

 We first compute
$\varphi_{k}^t$ as a function of $\varphi^t$.
The map $\rho_k(q,p)=(kq,p)$ defined on $T^*T^n$ is not invertible. Nevertheless its lift $\widetilde \rho_k: T^* {\mathbb R}^n \longrightarrow T^* {\mathbb R}^n$ is invertible. 
If we denote by $\widetilde \varphi$ any lift of $\varphi$, then $\widetilde \rho_k^{-1}\widetilde\varphi \widetilde \rho_k$ is $ \mathbb Z^n$-equivariant, that is 
$$\widetilde \rho_k^{-1}\widetilde\varphi \widetilde \rho_k (q+\nu,p)=\widetilde \rho_k^{-1}\widetilde\varphi \widetilde \rho_k(q,p)+ (\nu,0)$$
for $\nu \in \mathbb Z^n$. It therefore descends to a diffeomorphism of $T^*T^n$. However this diffeomorphism depends on the choice of the lift $\widetilde \varphi$. Since 
$\varphi$ is the time-one of a Hamiltonian isotopy, we may choose for $\widetilde\varphi$ the lift starting from the identity. With this choice we shall write by abuse of notation  $\rho_k^{-1}\varphi  \rho_k$ for the symplectomorphism of $T^*T^n$ induced by $\widetilde \rho_k^{-1}\widetilde\varphi \widetilde \rho_k$.

\begin{lemma}\label{Lemma-5.1} Let $\rho_{k}: T^*T^n \longrightarrow T^*T^n$ be defined by $\rho_k(q,p) = (kq,p)$, then with the above convention, we have
  $\varphi_{k}^{t} = \rho_k^{-1} \varphi^{kt} \rho_{k}$.
\end{lemma}
 \begin{proof}  The map $\rho_{k}$ is conformally symplectic, hence
$$
dH_{k}(z)\xi = dH(\rho_{k}(z)) d\rho_{k}(z) \xi = \omega
\left(X_{H}(\rho_{k}(z)), d\rho_{k} (z)\xi\right)
$$
$$
= (\rho^{*}_{k}\omega) (d\rho_{k}(z)^{-1} X_{H} (\rho_{k}(z)),\xi)
$$
Since $\rho^{*}_{k}\omega = k\omega$, we get
\begin{eqnarray*}
X_{H_{k}} (z) & = & k \left((\rho_{k})_{*} X_{H}\right) (z)\\
& = & (\rho_{k})^{*} (k X_{H})(z)
\end{eqnarray*}
The flow of $kX_{H}$ is $\varphi^{kt}$, hence the flow of
$(\rho_{k})^{*} (k X_{H})$ is $\rho^{-1}_{k}
\varphi^{kt}\rho_{k}$. \end{proof}

From now on we shall write $\varphi_k= \rho_k^{-1}\varphi^k \rho_k$.
We are thus looking for the $\gamma$-limit of
$\rho^{-1}_{k}\varphi^{k}\rho_{k}$.

Note that we may replace $\varphi = \varphi^{1}$ by
$\varphi^{1/r}$ for some fixed integer $r$. Indeed, if $\rho^{-1}_{k}
\varphi^{k/r} \rho_{k}$ $\gamma$-converges to $\psi$, we have that
$$
\rho^{-1}_{k} \varphi^{k}\rho_{k} = \rho_{r} \left( \rho_{kr}^{-1}
  \varphi^{kr/r} \rho_{kr} \right) \rho^{-1}_{r}
$$
$\gamma$-converges to $\rho_{r} \psi \rho^{-1}_{r}$. If our theorem is
proved for $\varphi^{1/r}$, $\psi$ will be generated by a
Hamiltonian depending only on the $p$ variable. We easily check that
in this case
$$
\rho_{r} \psi \rho^{-1}_{r} = \psi^{r}.
$$
In other words, $\rho^{-1}_{k} \varphi^{k}\rho_{k}$ $\gamma$-converges to
$\psi^{r}$.

\begin{remark} 
Note that $\rho_k\varphi\rho_k^{-1}$ is not well defined since $\widetilde\rho_k\widetilde\varphi\widetilde\rho_k^{-1}$ is not $ \mathbb Z^n$ equivariant. However the 
conjugacy $\rho_r(\rho_{kr}^{-1}\varphi\rho_{kr})\rho_r^{-1}$ is well defined : one may check that the lift is indeed $ \mathbb Z^n$-equivariant. Similarly for  $\psi$ the flow of an integrable Hamiltonian, $h(p)$, the map $\rho_r \psi\rho_r^{-1}$ is indeed well-defined, and equal to $\psi^r$ as stated. Of course by a continuity argument this extends to the case where $h$ is only continuous. 
\end{remark} 

For simplicity we shall assume from now on  that $\varphi$ is $C^{1}$-close to the
identity, so that it lifts to a Hamiltonian diffeomorphism $\widetilde\varphi$ of
$T^{*}\mathbb{R}^{n}$, $C^{1}$-close to the identity. The graph of the map $\widetilde\varphi$ has a generating function
 $S$  defined on $T^{*}T^{n}$ and compactly supported. We identify $q \in {\mathbb R}^n$ with its projection in $T^n$, and $P\in {\mathbb R} ^n$. By abuse of notation, we shall still denote by $S$ the lift of $S$ to $T^* {\mathbb R}^n$. 

The function  $S$ defines
$\widetilde\varphi$ by the relation

$$
\widetilde\varphi \left(q, P-\frac{\partial S}{\partial q} (q,P)\right )
=\left (q- \frac{\partial S}{\partial P} (q,P), P\right)
$$

This means that the graph $\widetilde\Gamma (\varphi)$ of $\widetilde\varphi$  in
$\overline{T^{*} \mathbb{R}^{n} }\times T^{*} \mathbb{R}^{n} \simeq
T^{*} \Delta_{\mathbb{R}^{2n}}$, has the generating function ${S}(q,p)$ that is the lift of a compactly supported function defined on $T^{*}T^{n}$.

In other words, setting $\widetilde \varphi (q,p)=(Q,P)$ we have

\begin{gather*} \left \{
\begin{array}{ll} P-p= \frac{\partial {S}}{\partial q}
  (q,P) \\ q-Q=\frac{\partial {S}}{\partial P}
  (q,P)
\end{array} \right .
\end{gather*}

Similarly when $\varphi$ is not $C^1$-close to the identity. Then $S$ has extra variable $\xi \in V$  and $S(x,P;\xi)$ is quadratic at infinity. Then for $(q,P;\xi)$ such that $\frac{\partial S}{\partial \xi }(q,P;\xi)=0$ we have
\begin{gather*} 
\widetilde\varphi \left(q, P-\frac{\partial S}{\partial q} (q,P;\xi)\right )
=\left (q- \frac{\partial S}{\partial P} (q,P;\xi), P\right) 
\end{gather*} 

We now give the composition law for generating functions, due to Chekanov (cf. \cite{Chekanov}).

\begin{lemma}\label{Lem-5.2} Let $\varphi_{1} , \varphi_{2}$ be
  Hamiltonian maps having ${S}_{1} ,{S}_{2}$ defined respectively on $ T^n\times {\mathbb R} ^n\times V_1$ and $ T^n\times {\mathbb R} ^n\times V_2$ as
  generating functions,  compactly supported functions on $T^*T^n$.
  Then $\widetilde{\varphi_{1}\circ \varphi_{2}}$ has the
  generating function
 \begin{gather*}
{S}(q_{1}, p_{2} ; q_{2} ,p_{1},\xi_1, \xi_2 ) ={S}_{1} (q_{1},p_{1},\xi_1) + {S}_{2} (q_{2}
,p_{2},\xi_2) + \langle p_{1} - p_{2} , q_{2} - q_{1} \rangle 
\end{gather*}
where $q_{1}\in {\mathbb R}^n, q_{2}\in {\mathbb R}^n, p_{1},p_{2}\in {\mathbb R} ^n$. 
\end{lemma}

\begin{remark} Let us point out the following
\begin{enumerate} \label{rem-5.4}
\item \label{GFQI-real} The above function $S$ is a \GFQI    , since the sum of two functions (in different variables) with derivatives at bounded distance from the derivative of a quadratic form
is at bounded distance from the derivative of a quadratic form. Note that according to Brunella (\cite{Brunella}, see also \cite{Viterbo-Montreal}, prop. 1.6), $S$ may be deformed to a generating function that is equal to a quadratic form outside a compact set. 
 
\item Note that in the above lemma, $q_{2}\in {\mathbb R} ^n$ is identified with its projection on $T^n$. 
We may remark that if we set $q_{2} = q_{1} + u$, we have
$$S'(q_1,p_2;u,p_2)=
{S} (q_{1} , p_{2} ; q_{1} + u , p_{2} ) =
{S}_{1} (q_{1} , p_{2} ) + {S}_{2} (q_{1}+ u ,
p_{2}) + \langle p_2-p_1, u\rangle
$$
so that ${S'}$ is a \GFQI     of the same Lagrangian as $S$ (since we did a fiber-preserving change of variable), and is the lift of a function defined on $T^{*}T^{n}\times {\mathbb R}^{2n}$ ($q_{1}\in T^{n}, p_{2}\in {\mathbb R}^{n}, u,v \in {\mathbb R}^{n}$).
\end{enumerate} 
\end{remark} 
\begin{proof} This is a straightforward  computation and we refer to \cite{Chekanov}, Theorem 4.1 (see also \cite{Chaperon2} for example). 

\end{proof}

More generally, we get:
\begin{lemma}\label{4.3}
The map $\widetilde{\varphi}^{k}$ has the generating function
 \begin{gather*}
{S}_{k}(q_{1} , p_{k} ; p_{1} , q_{2},p_{2},\cdots, q_{k-1} , p_{k-1}
,q_{k}) =\\
{\Sigma}_{k} (q_{1} , p_{k} ; p_{1} , q_{2}, p_{2} ,\cdots,
q_{k-1} , p_{k-1} , q_{k}) + B_{k} (q_{1}, p_{k} ; p_{1} , q_{2} ,p_{2},\cdots, q_{k-1} , p_{k-1} , q_{k})
\end{gather*}
where
\begin{gather*}
{\Sigma}_{k} (q_{1}, p_{k} ; p_{1} ,q_{2},p_{2} , \cdots ,
q_{k-1}, p_{k-1} , q_{k}) =
\sum_{j=1}^{k} {S}(q_{j} , p_{j})
\end{gather*}
and
$$
B_{k} (q_{1}, p_{k}; p_{1} , q_{2} ,\cdots, q_{k-1} , p_{k-1} ,
q_{k}) =
\sum_{j=1}^{k-1} \langle p_{j}, q_{j+1}-q_{j} \rangle + \langle p_k, q_1-q_k\rangle
$$

\end{lemma}

\begin{proof}  By induction from  the formula in Lemma \ref{Lem-5.2} 
 \begin{gather*}
{S}(q_{1}, p_{2} ; q_{2} ,p_{1} ) ={S}_{1} (q_{1},p_{1} ) + {S}_{2} (q_{2}
,p_{2}) + \langle p_{1} - p_{2} , q_{2} - q_{1} \rangle .
\end{gather*}

\end{proof}

\begin{remarks}
\begin{enumerate} 
\item  In the sequel  $(q_{j},p_{j})\in T^* {\mathbb R} ^n$, even though,  by our usual abuse of notations, we identify $q_{j}$ with its projection on $T^n= {\mathbb R} ^n/ {\mathbb Z}^n$.
  Note however that for $\nu \in {\mathbb Z}^n$, we have that  $\Sigma_{k}$ and $S_k$ are invariant by 
$(q_{j})_{1\leq j \leq n} \longrightarrow (q_{j}+ \nu )_{1\leq j \leq n}$. We emphasize that the $q_j$ and $p_j$ are vector coordinates (i.e. each $q_j$ and each $p_j$ is a vector in $ {\mathbb R}^n$).
\item Note that this formula is a discretization of the action functional (up to sign). The sum of the $S(q_j,p_j)$ corresponds to $\int Hdt$ while the quadratic term corresponds to $\int p\dot q dt$. 

\end{enumerate} 
\end{remarks} 
Again, ${\Sigma}_{k}$ is defined on
$(T^{*}T^{n})^{k}$, while ${B}_{k}$ is defined on
$(T^{*}\mathbb{R}^{n})^{k}$. Note also that $B_k$ is the discretization of $\int_{S^1}p\dot q dt$, so that our expression is the discretization of the Maupertuis action $\int_{S^1}p\dot q - H dt$.  Finally we have

\begin{lemma}\label{Lemma-5.7}
Let $\varphi$ a Hamiltonian diffeomorphism of $T^*T^n$  generated by $S(q,p)$. Then $\varphi_{k} =
\rho_{k}^{-1} \varphi^{k} \rho_{k}$ is generated by ${F}_{k}$ given by
$$
{F}_{k} (q_{1} ,p_{k} ; p_{1} ,\cdots, q_{k-1} , p_{k-1} , q_{k} ) =
$$
$$
\frac{1}{k}{\Sigma}_{k} (k q_{1}, p_{k} ; p_{1} ,\cdots, k
q_{k-1} , p_{k-1} , k q_{k} )
+  {B}_{k} (q_{1}, p_{k} ; p_{1} ,\cdots, q_{k-1}, p_{k-1} ,q_{k}) \dispdot
$$
\end{lemma}
\noindent
\begin{proof}  Indeed an elementary computation shows  that if $S(q,p;\xi)$ is a generating function for
$\psi : T^* {\mathbb R}^n \longrightarrow  T^* {\mathbb R}^n$, then $\frac{1}{k} S (kq , p; \xi)$ generates  $\rho^{-1}_{k} \psi \rho_{k}$.

\medskip

Thus in our case, we expect the generating function

$$
\frac{1}{k} {\Sigma}_{k} (k q_{1} , p_{k} ; p_{1} ,
q_{2},\cdots, q_{k-1} , p_{k-1} , q_{k})
$$
$$
+ \frac{1}{k} {B}_{k} (k q_{1} , p_{k} ; p_{1} , q_{2},\cdots, q_{k-1} ,
p_{k-1} , q_{k})\dispdot
$$
But the fiber-preserving change of variable $q_{j} \mapsto kq_{j} \; (j\geq
2)$ transforms this generating function into
$$
F_{k} (q_{1} , p_{k} ; p_{1} ,q_{2},\cdots, q_{k-1},p_{k-1} ,q_{k}) =
$$
$$
\frac{1}{k} {\Sigma}_{k} (kq_{1} , p_{k} ; p_{1}, kq_{2}
,\cdots, kq_{k-1} , p_{k-1} , q_{k})
$$
$$
+ \frac{1}{k} {B}_{k} (kq_{1} ,p_{k}; p_{1} , k q_{2} ,\cdots, kq_{k-1}
, p_{k-1} , q_{k}) \dispdot
$$
Because the second term is quadratic, we easily check that it is equal to
$$
 {B}_{k} (q_{1} ,p_{k} ; p_{1} , q_{2} , \cdots , q_{k-1} , p_{k-1} ,
q_{k})\dispdot$$

\end{proof}
\begin{definition}\label{4.5}  We set to simplify our notations
$$
x = q_{1} , y = p_{k} \ ,\ \xi = (p_{1} , q_{2},\cdots, q_{k-1} ,
p_{k-1}, q_{k})\dispdot
$$
We define  
$$F_k: {\mathbb R}_x^n\times ({\mathbb R}^n)_y^* \times {\mathbb R}_{\xi}^{2n(k-1)}  \longrightarrow {\mathbb R} $$
 \begin{gather*}  \label{def-F} F_k(x,y;\xi)=\\
  \frac{1}{k}\left [ S(kx,p_{1})+ \sum_{j=2}^{k-1} S(kq_{j},p_{j}) + S(kq_{k},y)\right ]+B_{k}(x,y,\xi)\end{gather*} where $B_k(x,y,\xi)$ is defined by 
  $$B_k(q_1,p_k;p_2,...,q_k)=\sum_{j=1}^{k-1}\langle p_{j}, q_{j+1}-q_{j}\rangle + \langle p_{k},q_{1}-q_{k}\rangle$$

Let $\mu_x$ be the fundamental class in the torus $T^n$ (which variable is denoted here by $x$). We then define
  $${h}_{k} (y) = c (\mu_{x} , {F}_{k,y})=c(\mu_{x} \otimes 1(y) ,{F}_{k})$$
  where
${F}_{k,y} (x;\xi) = {F}_{k} (x,y; \xi)$. 
\end{definition}

\begin{remarks} \label{red-normal}
Let us make the following remarks
\begin{enumerate} \item As long as we write $c(\mu_{x} \otimes 1(y) , S)$ for a generating function $S$, there is no ambiguity.
However, if $\Lambda$ is the Lagrangian associated to $S$, and $(\Lambda)_y$ is the reduction of $\Lambda$ at $y$, having \GFQI     $S_y$, 
 writing an expression like $c(\alpha , (\Lambda )_{y})$ requires some care, since $S$ is defined up to a
constant, and this constant yields a coherent choice of a \GFQI     for
$(\Lambda)_{y}$ for each $y$, so that the $c(\alpha, (\Lambda)_{y})$ are well-defined up to the same constant for all values of the parameter $y$, and not up to a function of $y$ as one could expect.
Indeed, a choice of $S$, a  \GFQI     for $\Lambda$  defines a \GFQI     for $(\Lambda)_{y}$, by $S_{y}= S(y, \bullet )$.

 Moreover, for similar reasons, even $c(\beta, (\Lambda)_y)-c(\alpha, (\Lambda)_y)$ depends on  the global $\Lambda$ and not only on its reduction $(\Lambda)_y$. 
 
\item Note also that for $\nu=(\nu_{1},...,\nu_{k})=(\nu_{1},\overline \nu)\in {\mathbb Z}\oplus {\mathbb Z} ^{k-1}$, we have  $F_{k}(x,y;\xi)=F_{k}(x+\nu_{1},y, \overline \nu + \xi )$, where $\overline\nu+ \xi=(p_{1},q_{2}+\nu_{2},...., p_{k-1},q_{k}+\nu_{k})$. Indeed, this periodicity is obvious for the terms containing $S$, and the quadratic term is
$$B_k(q_1,p_k;p_2,...,q_k)=\sum_{j=1}^{k-1}\langle p_{j}, q_{j}-q_{j+1}\rangle + \langle p_{k},q_{1}-q_{k}\rangle$$ remembering that $x=q_{1},y=p_{k}$  for which the periodicity is easily checked. 

\item Note that since $\varphi_{k}^{t}$ equals the identity outside a compact set of the cotangent bundle, the function $\xi \mapsto F_{k}(x,y,\xi)$ will have, for $y$ large enough, a single critical point with critical value equal to zero (by normalization).  According to remark \ref{rem-5.4}, \ref{GFQI-real}, we could deform $F_{k}$ to be exactly quadratic for $y$ large enough, but this is not really useful, since for such values of $y$ topologically there is no way to distinguish $\xi \mapsto F_{k}(x,y,\xi)$ from a quadratic form (the topology of the sub-level sets will coincide).

\item \label{red-normal-4} The analogue of the above formula still holds if $\varphi$ is not assumed to be $C^1$-small. Then its graph $\overline \Gamma (\varphi)$ has a generating function $S(q,p;\zeta)$, on $ {\mathbb R}_x^n\times ({\mathbb R}^n)_y^*\times E \longrightarrow {\mathbb R} $   and  $\varphi_k^1$ has generating function 
$$F_k: {\mathbb R}_x^n\times ({\mathbb R}^n)_y^* \times {\mathbb R}_{\xi}^{2n(k-1)} \times E^k  \longrightarrow {\mathbb R} $$ where $\xi$ is as in Definition \ref{4.5}
\begin{gather*}  \label{def-F2} F_k(x,y;\xi, \zeta_1,..., \zeta_k)= \\   \frac{1}{k}\left [ S(kx,p_{1},\zeta_1)+ \sum_{j=2}^{k-1} S(kq_{j},p_{j},\zeta_j) + S(kq_{k},y,\zeta_k)\right ]+B_{k}(x,y,\xi) \dispdot\end{gather*} 
 \end{enumerate} 
\end{remarks}

Our first step will be to prove
\begin{proposition}\label{4.6}
The sequence $({h}_{k})_{k\geq 1}$ is a precompact sequence for the  $C^0$ topology.
 \end{proposition}
The proposition will  follow from Ascoli-Arzela's
theorem, once we prove the following
\begin{lemma}\label{Lemma-5.10}
The sequence $({h}_{k})_{k\geq 1}$ is equicontinuous and uniformly bounded.
\end{lemma}

\begin{proof}
Indeed let $\widetilde\varphi_k$ be the lift of $\varphi_{k}=\rho^{-1}_{k}
\varphi^{k}\rho_{k}$ to $T^{*}\mathbb{R}^{n}$. It has support in some
tube
$$
T^{*}_{A} \mathbb{R}^{n} = \left\{ (q,p) \in T^{*}\mathbb{R}^{n} \mid \
  \vert   p\vert   \leq A\right\}\dispdot
$$
Now remember from \cite{Viterbo-STAGGF} (section 2, pages 690-693) that $c(\alpha, F_{k,y})$ is  a critical value of $F_{k,y}$. Thus for each $y$ there exists $x(y) , \xi(y)$ such that
$$
\frac{\partial {F}_{k}}{\partial x} (x(y),y; \xi(y))=0, \;\;\quad
\frac{\partial {F}_{k}}{\partial \xi} (x(y),y,\xi(y))=0
$$
$$
{F}_{k}(x(y),y,\xi(y)) = {h}_{k}(y)\dispdot
$$
Moreover, we may assume $\varphi$ is generic, so that  the map $y\mapsto (x(y),\xi(y))$ is
smooth on a set $W$,  the complement of some codimension one subset (see \cite{Viterbo-Ottolenghi} and also \cite{Roos, WQ1}, Appendix 2 (theorem 13.1) of \cite{NCMT}). Thus for $y$ in $W$
$$
d {h}_{k}(y) = \frac{\partial}{\partial y} {F}_{k}
(x(y),y, \xi(y))
$$
$$
= x (y) - X_{k} (x(y),y)
$$
where $X_{k}$ is defined by
$$
\widetilde\varphi_{k}(x,y) = (X_{k} (x,y) , Y_{k} (x,y))\dispdot
$$
The quantity $x(y)-X_{k}(x(y),y)$ can be estimated as follows~: the
first coordinate of the flow $\widetilde\varphi_{k}^{t}$ satisfies
$$
\dot{x}_{k}(t)= \frac{\partial H}{\partial y} (k x_{k}(t), y_{k}(t))
$$
hence $\vert   \dot{x}_{k}(t)\vert   $ is bounded by $C =\sup\left\{ \left\vert
  \frac{\partial H}{\partial p} (x,y) \right \vert \mid  (x,y) \in T^{*}T^{n}\right\}$  which is finite since $H$ is compactly supported. 
This implies that  $\vert x - X_{k} (x,y)\vert \leq C$ hence in the complement of $\Sigma$ the inequality
$$
\left \vert  d {h}_{k}(y)  \right \vert= \left \vert   \frac{\partial}{\partial y} {F}_{k}
(x(y),y,\xi(y)) \right \vert \leq C
$$
holds. Since ${h}_{k}$ is continuous, this implies that  it is $C$-Lipschitz.

For the uniform boundedness, let $C$ be a bound for $ \vert S \vert $. Then according to Definition \ref{4.5} $ \vert F_k(x,y;\xi)-B_k(x,y,\xi)\vert \leq C$. Since $c(\mu_x, B_k)=0$ we get 
$ \vert  c(\mu_x, F_k) \vert =\vert h_k (y) \vert \leq C$. 
\end{proof}

From Ascoli-Arzela's theorem and the above Lemma \ref{Lemma-5.10}, we infer that the sequence $({ h}_k)_{k\geq 1}$ is relatively compact in the $C^0$ topology. In other word it has a $C^0$-converging subsequence, and so does any of its subsequences. We now argue as follows : consider a subsequence $(h_{k_{\nu}})_{\nu \geq 1}$ of $({ h}_k)_{k\geq 1}$ $C^0$-converging to $h_{\infty}$. We are going to prove that $(\varphi^{1}_{k_{\nu}})_{\nu\geq 1}$ $\gamma$-converges to $\overline\varphi_{\infty}$, the time-one flow of $h_{\infty}$. We still need to prove that the whole sequence $(\varphi^{1}_{k})_{k \geq 1}$ $C^0$-converges to $\overline\varphi_{\infty}$, but this follows\footnote{We use here the fact that in a metric space if every subsequence of a sequence $(x_k)_{k\geq 1}$ has a converging sub-subsequence with limit $x$, then the sequence itself converges to $x$.} from :
\begin{proposition}\label{Prop-5.12}
If a subsequence of $(\rho_{k}^{-1} \varphi^{k}\rho_{k})_{k\geq 1}$ has a $\gamma$-limit $\overline \varphi$, then any other $\gamma$-converging subsequence has the same $\gamma$-limit. 
\end{proposition}

\begin{proof}
We start with 
\begin{lemma}\label{Lem-5.13}  For any $\varphi, \psi$ in ${\DHam_c}(T^{*}T^{n})$:
\begin{equation} \label{1of4.7}
\gamma (\varphi^{k}\psi^{k}) \leq k \gamma (\varphi \psi)
\end{equation} 
\begin{equation} \label{2of4.7}
\gamma (\rho^{-1}_{k} \varphi \rho_{k}) = \frac{1}{k} \gamma (\varphi)
\dispdot 
\end{equation} 
\end{lemma} 
\begin{proof} 
Indeed, we may write
$$
\varphi^{k} \psi^{k} = \varphi \psi \left (\psi^{-1} (\varphi \psi) \psi\right )
\left (\psi^{-2} (\varphi \psi) \psi^{2}\right ) ... \left (\psi^{-(k-1)} (\varphi\psi) \psi^{k-1}\right )
$$
Since each factor is conjugate to $\varphi \psi$, and we have $k$
factors, property (\ref{1of4.7}) follows immediately from conjugation invariance of $\gamma$ and the triangle inequality. 
 Property (\ref{2of4.7}) follows from  the fact that if $S(q,P;\xi)$ is the \GFQI for $\varphi$ then $R_k(q,P;\xi)=\frac{1}{k}S(k\cdot q, P;\xi)$ is the \GFQI for $\rho_k^{-1}\varphi\rho_k$ and  it is easy to check that $c(\alpha , R_k)= \frac{1}{k}c(\alpha, S(q,P;\xi))$ (this is a manifestation of the scaling property of $\gamma$ by conformal conjugation, see equation (\ref{conf}) in section \ref{Sec-gf-summary}). 
\end{proof} 
Now we prove that the sequence $(\rho_{k}^{-1}\varphi^{k}\rho_{k})_{k\geq 1}$ cannot have two distinct limit points. 
Indeed, let us assume we have two infinite sets of integers, $A,B$ such that there exists $\overline \sigma_{1}\neq \overline \sigma_{2}$ and a sequence $ \varepsilon_{k}$ converging to $0$with the property 
$$\forall k\in A\; \gamma (\rho_{k}^{-1}\varphi^{k}\rho_{k},\overline\sigma_{1})\leq \varepsilon _{k}$$
$$\forall k\in B\; \gamma (\rho_{k}^{-1}\varphi^{k}\rho_{k},\overline\sigma_{2})\leq \varepsilon _{k}$$

Then for any integer $q$,  $$\gamma (\rho_{kq}^{-1} \varphi^{kq}\rho_{kq}, \overline \sigma_1)= \gamma (\rho_{q}^{-1}(\rho_{k}^{-1}\phi^{k}\rho_{k})^{q}\rho_{q}(\rho_{q}^{-1} \overline\sigma_1^{-q}\rho_{q}))$$ since $$\overline \sigma_1= \rho_{q}^{-1} \overline\sigma_1^{q}\rho_{q}$$
This last equality follows from the fact that for an integrable Hamiltonian, $H(p)$, we obviously have  $H_k=H$. 

But using  (\ref{2of4.7}), we get for $k\in A$ \begin{gather*}
\gamma (\rho_{kq}^{-1} \varphi^{kq}\rho_{kq}, \overline \sigma_{1})= \gamma (\rho_{q}^{-1}(\rho_{k}^{-1}\phi^{k}\rho_{k}) ^q\rho_{q}(\rho_{q}^{-1} \overline\sigma_{1}^{-q}\rho_{q})) =
  \gamma (\rho_{q}^{-1}((\rho_{k}^{-1}\phi^{k}\rho_{k}) ^q \overline\sigma_{1}^{-q}) \rho_{q})
 \leq \\ \frac{1}{q} \gamma ((\rho_{k}^{-1}\phi^{k}\rho_{k})^{q} \overline\sigma_{1}^{-q}) \leq  q \cdot \frac{1}{q} \gamma (\rho_{k}^{-1}\phi^{k}\rho_{k}\overline \sigma_{1}^{-1})\leq \varepsilon_{k}
\end{gather*}

Similarly we get that for $k$ in $B$, and any $q$, we have 

$$\gamma (\rho_{kq}^{-1} \varphi^{kq}\rho_{kq}, \overline \sigma_{2})\leq \varepsilon _{k}$$

As a result, for $k\in A, m\in B$ we have 

$$\gamma (\rho_{km}^{-1} \varphi^{km}\rho_{km}, \overline \sigma_{1})\leq \varepsilon _{k}$$
$$\gamma (\rho_{km}^{-1} \varphi^{km}\rho_{km}, \overline \sigma_{2})\leq \varepsilon _{m}$$

This  implies $\gamma (\overline\sigma_{1}, \overline \sigma_{2}) \leq  \varepsilon _{m}+ \varepsilon _{k}$ and since the right hand side goes to zero as $k,m$ go to infinity, we get that $\overline \sigma_{1}=\overline \sigma_{2}$ and this concludes the proof. 
\end{proof}

\subsection{The Main steps of the proof}\label{eop}

In the previous section we obtained a continuous  function ${h}_{\infty}(p)$, as the limit of some subsequence $(h_{k_{\nu}}(p))_{\nu \geq 1}$. Since $h_{\infty}$ is
continuous, according to
Humili\`{e}re (cf. our proposition \ref{Prop-3.13} above, or \cite{Humiliere} proposition 1.3) it has a ``generalized flow''  that is generates a one-parameter subgroup  in $\gclDHamc{T^*T^n}$. In other words   the map $H \longrightarrow \varphi_{H}^t$ extends to a map between the $\gamma$-completions, $$ \gclHamc{T^*T^n} \longrightarrow \gclDHamc{T^*T^n}\dispdot $$ Note that $\gclDHamc{T^*T^n}$ inherits the group structure of ${\DHam}_{c}(T^*T^n)$. We denote by ${\overline\varpsi}_\infty^{\; t}$ the
``generalized flow'' associated to $h_{\infty}$. Note that the element $\overline\varpsi_\infty^{t}$ is not a map: it is only an element in $\gclDHamc{T^*T^n}$.
\begin{proposition}\label{Prop-5.14}
For all $t\geq 0$, the element $\overline\varpsi_\infty^{\;t}$ is the  $\gamma$-limit of $(\varphi_{k}^{t})_{k\geq 1}$ defined by $\varphi_{k}^{t}=
\rho^{-1}_{k} \varphi^{kt} \rho_{k}$ : we have
$$
\lim_{k\to +\infty} \gamma (\varphi_{k}^{t}, \overline\varpsi_\infty^{\;t} )=0\dispdot
$$
\end{proposition}
From now on we assume $t=1$, and denote $\varphi_{k}=\varphi_{k}^{1}$, and $\overline\varpsi_\infty=\overline\varpsi_\infty^{\;1}$. The proof of Proposition \ref{Prop-5.14} will be based on the following  two propositions

\begin{proposition}\label{Prop-5.15} There exists a sequence $(k_{\nu})_{\nu\geq 1}$ going to infinity such that for any $\alpha$ in
$\gclDHamc{T^*T^n}$, we have
$$
\limsup_{\nu \to \infty}c_{+}(\varphi_{k_{\nu}} \alpha ) \leq c_{+}(\overline\varpsi_\infty
\alpha )
$$
\end{proposition}

\begin{proposition}\label{Prop-5.16}

Consider a subsequence of $(\varphi_{k_{\nu}})_{\nu\geq 1}$ such that (all limits are uniform in $p$)
$$
\lim_{\nu \to \infty} c (\mu \otimes 1 (p) , \varphi_{k_{\nu}}^{}) =\lim_{\nu} {h}_{k_{\nu}} (p)
= {h}_{\infty} (p)
$$
Then we have
$$
\lim_{\nu \to \infty} c (\mu \otimes 1 (p) , \varphi^{-1}_{k_{\nu}}) = - {h}_{\infty}
(p)
$$
\end{proposition}
  \begin{proof} [ Proof that Proposition \ref{Prop-5.15} and \ref{Prop-5.16} imply
Proposition \ref{Prop-5.14}.]

Indeed take $\alpha = \overline\varpsi_\infty ^{-1}$, where $\overline \varpsi_\infty=\overline\varphi^{1}$ is the limit associated by the previous subsection to some subsequence of $(k_{\nu})_{\nu \geq 1}$ still denoted  $(k_{\nu})_{\nu \geq 1}$.  By Proposition \ref{Prop-5.15}, we get
$$
\limsup_{\nu} c_+(\varphi_{k_{\nu}}  {\overline\varpsi_\infty}^{-1}) \leq c_+( \Id)= 0.
$$
and since for all $\beta$ in $\gclDHamc{T^*T^n}$,  $c_+(\beta ) \geq 0$ we get,

$$
\limsup_{\nu} c _+(\varphi_{k_{\nu}} {\overline\varpsi_\infty}^{-1})=0$$

Now we must prove $\liminf_{\nu} c _-(\varphi_{ k_{\nu}}  \overline\varpsi_\infty^{-1}) =0$, and it is enough to show that
$$
\liminf_{\nu} c_-(\varphi_{k_{\nu}}\alpha) \geq c_-(\overline\varphi_\infty \alpha)
$$
for any $\alpha$ in $\gclDHamc{T^{*}T^{n}}$.
\medskip

But according to 
 \cite{Viterbo-STAGGF}( (2) of proposition 4.2 page 697) and invariance by conjugation of $c$ (see \cite{Viterbo-STAGGF},  corollary 4.3 page 698) the formulas $$c_+(\varphi^{-1})=-c_-(\varphi)$$ and $$c_\pm (\psi\varphi\psi^{-1})=c_\pm(\varphi)$$ hold in $\DHam(T^*T^n)$. Since $c_\pm$ are obviously continuous for the $\gamma$-topology, the same formulas hold in $\gclDHamc{T^*T^n}$. 
We can thus write
$$
c_-(\varphi_{  k_{\nu}}\alpha) =
- c_+( \alpha^{-1}\varphi_{k_{\nu}}^{-1})
= - c_+( \varphi^{-1}_{k_{\nu}} \alpha^{-1})\dispdot
$$
We then apply proposition \ref{Prop-5.15} to the sequence
$(\varphi^{-1}_{k_{\nu}})$. According to proposition \ref{Prop-5.16}
$$
\lim_{\nu} c (\mu \otimes 1 (p) , \varphi^{-1}_{k_{\nu}}) = - {h}_{\infty}(p)
$$
and according to Appendix A, Corollary \ref{Cor-app-A}, $-{ h}_{\infty}(p)$ has flow $\overline\varphi_\infty^{-1}$ in the completion 
$\gclDHamc{T^*T^n}$.

\medskip
  
As a result
$$
\liminf_{k} c_-( \varphi_{k_{\nu}}\alpha) =
- \limsup_{k} c_+(\varphi_{k_{\nu}}^{-1}\alpha^{-1})
$$
$$
\geq - c _+({\overline\varpsi_\infty}^{-1}\alpha^{-1})
= c_-(\overline\varpsi_\infty \alpha)\dispdot
$$
Taking again $\alpha=\overline\varphi_\infty^{-1}$ we get 
$$\liminf_{k} c_-( \varphi_{k_{\nu}}{\overline\varpsi_\infty}^{-1}) =0$$
hence $$\lim_{k} c_-( \varphi_{k_{\nu}}{\overline\varpsi_\infty}^{-1}) =\lim_{k} c_+( \varphi_{k_{\nu}}{\overline\varpsi_\infty}^{-1}) =0$$
We thus proved that if the sequence $h_{k_{\nu}}(p)=c(\mu_{x}\otimes 1(p),\phi_{k_{\nu}})$ $C^0$-converges to $h_{\infty}$, then
$(\phi_{k_{\nu}})_{\nu \geq 1}$ $\gamma$-converges to $\overline\varpsi_\infty$. Note that conversely, if $(\phi_{k_{\nu}})_{\nu \geq 1}$ $\gamma$-converges to $\overline\varpsi_\infty$, since  $ \psi \mapsto c(\alpha, \psi)$ is $1$-Lipschitz, we have that 
$h_{k_{\nu}}(p)=c(\mu_{x}\otimes 1(p),\phi_{k_{\nu}})$ $C^0$-converges to $c(\mu_x\otimes 1(p),{\overline\varpsi_\infty})=h_{\infty}(p)$. 

Now assume there are two subsequences, $(\phi_{k_{\nu}})_{\nu \geq 1},( \phi_{l_{\nu}})_{\nu \geq 1}$ such that $c(\mu_{x}\otimes 1(p),\phi_{k_{\nu}})$ $C^0$-converges to $h^{1}_{\infty}$, while $c(\mu_{x}\otimes 1(p),\phi_{l_{\nu}})$ $C^0$-converges to $h^{2}_{\infty}$. Then we find two subsequences of $(\phi_{k})_{k\geq 1}$ $\gamma$-converging respectively to $\overline \varpsi_{1}$ and $\overline \varpsi_{2}$ (where $\overline \varpsi_{1}^{\;t}$ is the flow of $h^{1}_{\infty}$ while $\overline \phi_{2}^{\;t}$ is the flow of $h^{2}_{\infty}$). But according to Proposition \ref{Prop-5.12}, two $\gamma$-converging subsequences of $(\phi_{k})_{k\geq 1}$ must have the same limit, thus $\overline \varpsi_{1}=\overline \varpsi_{2}$. Using again the continuity of $\psi \mapsto c(\alpha,\psi)$ for $\gamma$, we have  \begin{gather*}  h^{1}_{\infty}(p)=\lim_\nu c(\mu_x\otimes 1(p), \varphi_{k_\nu})=c(\mu_x\otimes 1(p), \overline\varphi_{1})=\\ c(\mu_x\otimes 1(p), \overline\varphi_{2})=\lim_\nu c(\mu_x\otimes 1(p), \varphi_{l_\nu})=h^{2}_{\infty}(p) \end{gather*} 

As a result we proved that 
\begin{enumerate} [(a)]
  \item \label{a} the sequence $(h_k)_{k\geq 1}$ is precompact for the $C^0$ topology
\item \label{b} If a subsequence $(h_{k_\nu})_{\nu\geq 1}$ $C^0$-converges to $\overline h$ then the sequence  $(\varphi_{k_\nu})_{\nu\geq 1}$ converges to $\overline\varphi_\infty$
\item\label{c} any two converging subsequences of $(\varphi_{k})_{k\geq 1}$ have the same limit
\end{enumerate} 
We claim that this implies  that $(\varphi_{k})_{k\geq 1}$ converges, and its limit is $\overline \varphi=\overline \varphi_\infty$, the flow of $h_\infty$ defined by $h_\infty(p):=\lim_k c(\mu_x\otimes 1(p),  \varphi_k)$.

Indeed from (\ref{a}) we get the existence of a limit $\overline h$ for some subsequence, from (\ref{b}) that then the corresponding subsequence of $(\varphi_k)_{k\geq 1}$ converges to $\overline\varphi$ and since by (\ref{c}) $\overline \varphi$ is unique, so is $\overline h$ (according to Proposition \ref{Prop-A.1} from the Appendix). 
Now we claim $(h_k)_{k\geq 1}$ converges. Indeed, assume otherwise: there would be a subsequence $k_\nu$ such that $\gamma (h_{k_\nu},\overline h) \geq \varepsilon $ for all $\nu$. But the subsequence $h_{k_\nu}$ being precompact has a converging subsequence, whose limit must be different from $\overline h$ a contradiction with (\ref{c}). Now we know $(h_k)_{k\geq 1}$  converges to $\overline h$, hence so does $(\varphi_k)_{k\geq 1}$  and it must have limit $\overline \varphi_\infty$  the flow of $\overline h_\infty$ (in $\gclDHamc{T^*T^n}$).

This concludes our proofs of Proposition \ref{Prop-5.14}, modulo the proofs of Propositions \ref{Prop-5.15}
and \ref{Prop-5.16}.
 \end{proof}

We note the following 

\begin{corollary} \label{Corollary-5.17}
The $\gamma$-convergence of $\varphi_k^t$ to $\varphi_\infty^t$ is uniform in $t$ on compact sets. In other words the sequence of functions $t \mapsto \gamma(\varphi_k^t\overline\varphi^{-t})$ converges uniformly on compact sets to $0$. 
\end{corollary} 
 \begin{proof} 
 This follows immediately from the fact that if $ \Vert H \Vert_{C^0} \leq C$ then $ \Vert H_k \Vert_{C^0} \leq C$ and thus $\gamma (\varphi_k^t\varphi_k^{-s}) \leq C \vert t-s \vert $.  
 So let $f_k(t)=\gamma (\varphi_k^t\overline\varphi^{-t})$, then \begin{gather*}  \vert f_k(t)-f_k(s) \vert =\vert \gamma(\varphi_k^t\overline\varphi^{-t}) - \gamma(\varphi_k^s\overline\varphi^{-s}) \vert \leq \gamma (\varphi_k^t\overline\varphi^{-t}\circ \overline\varphi^{s}\varphi_k^{-s}) \leq \\
\gamma (\varphi_k^t\varphi_k^{-s})+ \gamma (\overline \varphi^t\overline\varphi^{-s})\leq 2 C \vert t-s \vert \end{gather*} 
 Now if a sequence $(f_k)_{k\geq 1}$ of functions defined on $[0,1]$ and uniformly Lipschitz converges simply to a function $f$, then the convergence is uniform. Indeed, the sequence $(f_k)_{k\geq 1}$ is equicontinuous. It is also uniformly bounded, since $\gamma(\varphi_k)\leq \Vert H_k\Vert_{C^0}\leq C$ so $\gamma (\varphi_k^t\overline\varphi^{-t}) \leq \gamma (\varphi_k^t)+\gamma(\varphi^{-t})\leq 2C$. So by Ascoli's theorem, one can find a subsequence converging  uniformly on compact sets,  but the limit of the subsequence is necessarily $f$. And since we proved that all converging subsequences must have the same limit, the sequence is uniformly converging to $f$ on compact sets.  
 \end{proof} 
\section{Proof of proposition \ref{Prop-5.15}} First of all, it is enough to deal with the case where $\alpha \in {\DHam}_{c}(T^*T^n)$ (i.e. not in its  completion), since $\DHam_c(T^*T^n)$ is dense in $ \gclDHamc{T^*T^n}$ and $c_+$ is continuous for $\gamma$. Now we may choose
$S(x,y;\eta)$  a  \GFQI    for  $\alpha$, with $\eta$ belonging to some vector space $V$. Let $\mathbb Z^n$ act diagonally on the $(u,x)$ variables by $\nu\star (u,x)\mapsto , (u+\nu, x+\nu)$ and extend this action by the trivial action on products of $ {\mathbb R}^n_u\times {\mathbb R}^n_x$  so that $\varphi_{k}\alpha$ has the \GFQI  defined on the $\mathbb Z^n$ quotient  
$$G_k: T^n_u\times {\mathbb R}^n_v\times {\mathbb R}^n_x\times {\mathbb R} ^n_y\times V\times E_k \longrightarrow  {\mathbb R} $$
$$
G_{k} (u,v; x,y, \eta , \xi) =
S (x+u,v; \eta) + {F}_{k} (u,y; \xi)+ \langle y-v ,x\rangle
$$
where $F_{k}$ was defined in Definition \ref{4.5} and is a \GFQI     for $\varphi_k$,   $u \in T^n_u, v \in  {\mathbb R}^n_v, x\in {\mathbb R}^{n}_x, y\in {\mathbb R}^{n}_y, \xi \in E_{k}, \eta \in V$, and $E_{k}$ is the space $ {\mathbb R}^{2n(k-1)}$ as in Definition \ref{4.5}.

It will often be more convenient to switch from cohomology to homology in order to make our argument geometrically more transparent. By relative cycle we mean a chain with boundary in  $F_k^{-\infty}$ (i.e. in $F_{k}^{-c}$ for $c$ large enough). The number $c(a,S)$ was defined  for a homology class $a$ in Remark \ref{rem-3.4} \ref{rem-homology-level}. The identification of  $c(\mu_{u}\otimes 1(y), F_k)$ with $c( [T_u^n \times \{y\}], F_k)$ follows from  Appendix \ref{Appendix-B}, Prop \ref{Prop-B3}. In the sequel we denote by $C$ a relative cycle representing $z$ in $H_*(F^b, F^a)$ which means that $C \subset F^b$ and $[C, C\cap F^a]=z$ in $H_*(F^b, F^a)$. 

Thus,  by definition of $c(\mu_{u}\otimes 1(y),\varphi_{k})=h_{k}(y)$, for each $y$, there is a relative cycle
$C^-(y)$ homologous to $T_u^{n} \times \{y\}\times E^{-}_{k}$ ($x$ lives in $T^{n}$, $\xi$ in $E_{k}$, and  $E_{k}^{-}$
is the negative eigenspace for the quadratic part $B_{k}$ of $F_{k}$)  in $H_{*}(G_{k,(x,v,\eta)}^{\infty},G_{k,(x,v,\eta)}^{-\infty})= H_{*}(F_{k}^{\infty},F_{k}^{-\infty})$ such that 
$$
{F}_{k} (y, C^-(y))\leq {h}_{k}(y) +\varepsilon\dispdot
$$
(we denote by $(y, C^-(y))$ the set of $(u,y,\xi)$ such that $(u,\xi)\in
C^-(y))$. Unfortunately we may not get such an estimate if we simultaneously require  that $C^-(y)$ is to depend continuously on $y$. However, let us first assume such a continuous dependence can be achieved and $(h_{k})_{k\geq 1}$ converges to $h_{\infty}$. 
Set
\begin{gather*} 
{\overline G}_k: T^{n}_{u}\times {\mathbb R}_{v}^{n} \times \mathbb{R}^{n}_{x} \times
\mathbb{R}^{n}_{y} \times V / \mathbb Z^n\\
{\overline G}_{k} (u,v ; x, y , \eta) =
S (x+u,v;\eta) + {h}_{k}(y) + \langle y-v , x\rangle
\end{gather*} 
Again we may find  a (relative) cycle $\Gamma$ in $T^{n}_{u}\times {\mathbb R}_{v}^{n} \times (\mathbb{R}_{x}^{n})\times
(\mathbb{R}_{y}^{n}) \times V$ in the homology
class of $T^{n}_{(u,v)}\times \Delta_{x,y} \times V^{-}$, where $V^{-}$ is the negative eigenspace of the quadratic 
part\footnote{ Notice that $\Delta_{x,y}$,  the diagonal in $\mathbb{R}^{n}_{x}
\times \mathbb{R}^{n}_{y}$,  is the negative eigenspace of
$\langle y , - x\rangle$.}  of $S$,   such that   $$\sup \overline G_{k}(\Gamma) \leq c (\mu, \overline G_k) + \varepsilon = c (\mu
, \overline\varphi\alpha ) + \varepsilon$$

\medskip

Let now $\Gamma \times_{Y} C^-$ be the (relative) cycle
$$
\Gamma \times_{Y} C^-=\left\{ (u,v,x,y,\xi,\eta) \mid  (u,v,  x, y, \eta)\in \Gamma , (u,\xi) \in C^-(y)\right\}\dispdot
$$
We claim that

\begin{enumerate} 
\item \label{1i} $$\sup 
G_{k} (\Gamma \times_{Y} C^-) \leq \overline G_{k} (\Gamma) +  \varepsilon 
$$
\item \label{2i} $\Gamma \times_{Y}C^-$ is a cycle in $H_{*}(G_{k}^{\infty},G_{k}^{-\infty})$ homologous to
$$
T^{n}_{u}\times {\mathbb R}_{v}^{n}\times \Delta_{x,y} \times E^{-}_{k} \times V^{-}$$
\end{enumerate} 
so that
$$
c(\mu, \phi_{k}\alpha)=c(\mu , G_{k} ) \leq \sup  G_k(\Gamma \times_YC^-)\leq \sup  \overline G_{k} (\Gamma)+ \varepsilon  \leq c(\mu ,\overline G_{k})+ \varepsilon  \leq c(\mu ,
\varphi_{k} \alpha)+ \varepsilon
$$

Indeed, for \ref{1i} we have $$\sup G_k (\Gamma\times_YC^-)=\sup \{S(x+u,v;\eta)+F_k(u,y,\xi)+ \langle y-v,x\rangle \mid (u,v,x,y,\eta)\in \Gamma, (u,\xi)\in C^-(y)\}$$
but $F_k(u,y,\xi)\leq h_k(y)+ \varepsilon $  for $(u,\xi)\in C^-(y)$, so

\begin{gather*} \sup G_k (\Gamma\times_YC^-) \leq \\ \sup \{S(x+u,v;\eta)+h_k(y)+ \langle y-v,x\rangle \mid (u,v,x,y,\eta)\in \Gamma\}+ \varepsilon \leq \\ \sup  \overline G_k(\Gamma)+ 
\varepsilon \end{gather*} 

For \ref{2i}, we use the fact that if $\Gamma$ is homologous to $T^{n}_{(u,v)}\times \Delta_{x,y} \times V^{-}$ and $C^-(y)$ is homologous to 
$T^{n} \times \{y\}\times E^{-}_{k}$ then $\Gamma\times_YC^-$ is homologous\footnote{The general fact that if we have two fibrations, $p_X:X \longrightarrow Z, p_Y: Y \longrightarrow Z$ and two cycles $A,B$ in $X,Y$ respectively, in general position with respect to $p_X, p_Y$, then the homology class of $A\times_ZB$ only depends on the homology classes of $A$ and $B$.  This follows from the fact that, denoting $p=(p_X,p_Y):(X\times Y) \longrightarrow Z$, then $A\times_ZB=(A\times B)\cap p^{-1}(\Delta_Z)$,  so represents the class $j_! ([A\times B])$ where $j: X\times_ZY \longrightarrow X \times Y$ is the inclusion, and $j_{!}$ the umkehr map. In general we may perturb $A,B$ so that they are in general position, and then the homology class of $A\times_ZB$ does not depend n the choice of perturbations. }
 to $(T^{n}_{(u,v)}\times \Delta_{x,y} \times V^{-})\times_Y(T^{n} \times \{y\}\times E^{-}_{k})$
that is $T^{n}_{u}\times {\mathbb R}_{v}^{n}\times \Delta_{x,y} \times E^{-}_{k} \times V^{-}$.

\medskip

Let us now try to establish the above   inequality  without the assumption that we can find $C^{-}(y)$  such that $\sup F_k(y, C^-(y)) \leq h_k(y)+ \varepsilon $ depending continuously on $y$. Let the subsequence $(h_{k_{\nu}})_{\nu\geq 1}$ converges to $h_{\infty}$. We will see in the next Lemma  that we may find a continuous family  $C^{-}(y)$ such that for $k=k_{\nu}$ and $\nu$ large enough, the  estimate $$\sup 
{F}_{k} (y, C^-(y))\leq {h}_{k}(y) +\varepsilon\dispdot
$$
holds for  $y$ outside of a subset $U_{2\delta}$ where $U_{\delta}$ is a $\delta$-neighborhood of some grid in
$(\mathbb{R}^{n})^{*}$ (see Figure \ref{Fig-1}), while inside $U_{2\delta}$, $\sup {F}_{k}(y,C^-(y))\leq a$ for some
constant $a$.

The existence of $C^{-}(y)$ is a consequence of the following general result. We remind the reader that a function $F$ satisfies the Palais-Smale condition if any sequence such that $F(x_{n})$ is bounded and $\nabla F(x_{n})$ converges to zero has a converging subsequence.

 This implies that the flow of $\frac{-\nabla F}{ \vert \nabla F (x) \vert^{2} }$
is defined for all time outside a neighborhood of the set of critical points (which is compact) provided our metric is complete. In particular,  for generic $F$ (i.e. $F$ Morse and with Morse-Smale gradient flow), classes in $H^{*}(F^{b}, F^{a})$ are represented by linear combinations of unstable manifolds of critical points (see for example \cite{Laudenbach-Bismut}). As usual we denote by $F^{\infty}$ (resp. $F^{-\infty}$) the set $F^{c}=\{ x \mid F(x) \leq c \}$ for $c$ large.

\begin{lemma}
Let $F(u,x)$ be a smooth function on the product  $V\times X$ of two oriented complete Riemannian manifolds. We moreover assume both $F$ and its restriction $F_{u}$ to a fiber $\{u\}\times X$ satisfy the Palais-Smale condition and  for $u$ outside a compact set, $F_{u}$ does not depend on $u$.  

Let  $f\in C^0(V, {\mathbb R} )$ be such that  for each $u \in V$, there exists a cycle $C(u)$ representing a class in $ H_{d}(F_{u}^{+\infty},F_{u}^{-\infty})$
with $\sup F(u, C(u))\leq f(u)$.  
Moreover we assume that $H_{p}(F_{u}^{\infty}, F_{u}^{-\infty})$ vanishes for $p\geq d+1$. 

Then for any positive $ \varepsilon $ and any subset $U$ in $V$, such that each connected component of $V\setminus U$ has sufficiently small  diameter, there exists a cycle $\widetilde C$ in $H_{d+\dim(V)}(F^{\infty}, F^{-\infty})$ and a constant $a$ such that if we denote by $\widetilde C(u)$ the slice $\widetilde C \cap \pi^{-1}(u)$ ($\pi:V\times X \longrightarrow X$ is the second projection)
$$\sup F(u,\widetilde C(u)) \leq f(u)+a \chi_{U}(u)+ \varepsilon $$ where $\chi_U$ is the characteristic function of $U$.
\end{lemma}
\begin{proof} What we are doing is a constructive version of the Leray spectral sequence $H_*(V, H_*(F_u^{+\infty}, F_u^{-\infty}))$, on the term $H_n(V, H_d(F_u^{+\infty}, F_u^{-\infty}))$, which yields a class in $H_{n+d} (F^{+\infty}, F^{-\infty})$. 

Our assumptions imply that the critical values of $F_{u}$ are contained in some bounded interval (independent from $u$), $[-a/2, a/2]$. 
We may assume  $V$ is triangulated and  $U$  contains  a neighborhood of the $(n-1)$ cells ($n=\dim (V)$). In other words $V\setminus U$ is contained in the union of the top dimensional cells. We denote by $V^{p}$ the $p$-skeleton of $V$, and by $W^{p}(\delta)$ a $\delta$-neighborhood of $V^{p}$. 
Continuity of $F$ implies that if we take $\widetilde C(u)$ to be constant (i.e. independent from $u$) in each connected component of $V^{n}\setminus W^{n-1}(\delta)$, containing a connected component of $V\setminus U$, the inequality will be satisfied there. We need to extend 
$\widetilde C(u)$ for all $u$ in $V$, so that the union of the $\widetilde C(u)$ makes a singular cycle. 

For this we proceed by induction on cells of the skeleton of $V$ of decreasing dimension, so that we are going to extend $\widetilde C (u)$ for $u\in V\setminus W^{p}(\delta_{p})$ successively for $p=n-1, n-2, ...., 1, 0$. 
For the first step, we need to glue the $\widetilde C (u)$ so that they yield a cycle over the union of $n$ and $n-1$ cells (outside a neighborhood of $(n-2)$-cells). 
For this we need to look at what happens on an $n-1$-dimensional cell, $T_{i,j}$ intersection of the $n$-cells $T_{i}$ and $T_{j}$ where we have  {\it a priori} two conflicting definitions of $\widetilde C(u)$, that we denote 
by $\widetilde C_{i}(u), \widetilde C_{j}(u)$ obtained by taking for $\widetilde C (u)$ the constant value given  on each of the two $n$-cells, $T_{i}, T_{j}$. Let us  write $\widetilde C_{i}(u)-\widetilde C_{j}(u)=\partial { \Gamma}_{i,j}(u)$ which is of course possible
 since ${\widetilde C}_{i}(u)$ and ${\widetilde C}_{j}(u)$ are homologous chains in $H_{d}(F_{u}^{+\infty}, F_{u}^{-\infty})$ on $T_{i}\cap T_{j}$. Then we may glue these together to
   $$\bigcup_{u\in \mathring{T}_{i} }{\widetilde C}_{i}(u)\cup \bigcup_{u\in \mathring{T}_{j} }{\widetilde C}_{j}(u)\cup \bigcup_{u\in \mathring{T}_{i,j} }{\Gamma}^{i,j}(u)\dispdot $$ Note that this is indeed a singular chain  of dimension $n+\dim ({\widetilde C}_{i}(u))$, since $\dim (T_{i})=\dim (T_{j})=n$, $ \dim ({\Gamma}^{i,j})=\dim ({\widetilde C}_{i}(u))+1, \dim(T_{i,j})=n-1$. Thus each of the three pieces has dimension 
  $n+\dim ({\widetilde C}_{i})=n+\dim ({\widetilde C}_{j})$. 

Now of course, the maximum of $F$ over $\bigcup_{u\in T_{i} }{\widetilde C}_{i}(u)\cup \bigcup_{u\in T_{j}} {\widetilde C}_{j}(u)\cup \bigcup_{u\in T_{i,j} }{\Gamma}_{i,j}(u)$ has increased outside $V\setminus U$,  but we can always assume that $\Gamma_{i,j}\subset F^{a}$.  If this was not the case, we could push  $\Gamma_{i,j}$ down using the  gradient of $F_u$, and since there is no critical value above $a/2$, we can push it below $a$. 
As a result we have defined a cycle over the complement of the  $(n-2)$-skeleton,  contained in $F^a$. 

Now we look at the next inductive step, that is extending to the $(n-2)$-skeleton. Consider then for $T_{\alpha}$ an $(n-2)$-simplex,  that is in the boundary of  the $T_{\beta}$ for $\beta$ in some set of $(n-1)$-simplices. For each $\beta$ we have a $\Gamma_{\beta}$ obtained in the previous step. 
We claim that at any point of the $(n-2)$-simplex $T_{\alpha}$ we have $\sum_{T_{\alpha} \in \partial T_{\beta}} \partial \Gamma_{\beta}=0$. Indeed by definition $\partial \Gamma_{\beta}= \widetilde C_{i_{\beta}}(u)-\widetilde C_{j_{\beta}}(u)$ where $T_{\beta}$ is bounded by $T_{i_{\beta}}$ and $T_{j_{\beta}}$, and in the above sum each index will appear twice (of course we need oriented simplices, and the above is correct provided they are oriented in a compatible way: this is made possible by the  orientability of $V$).

 Now, we claim that $\sum_{T_{\alpha} \in \partial T_{\beta}} \Gamma_{\beta}$, which is closed by the above argument  is in fact a boundary $\partial \Gamma_{\alpha}$. Indeed, over each fiber, $u$, this is a $d+2$-cycle,  and 
 the homology of the fiber $H_{k}(F_{u}^{+\infty}, F_u^{-\infty})$ vanishes for $k\geq d+2$. Therefore in this dimension, any cycle is a boundary, and we may find $\Gamma_{\alpha}$ such that $$\partial \Gamma_{\alpha}=\sum_{T_{\alpha} \in \partial T_{\beta}} \Gamma_{\beta}$$

Now $\Gamma_{\alpha}$ is contained in $F^{c}$ and we get  $$\bigcup_{u\in \mathring{T}_{i} }{\widetilde C}_{i}(u)\cup \bigcup_{u\in \mathring{T}_{j} }{\widetilde C}_{j}(u)\cup \bigcup_{u\in \mathring{T}_{i,j} }{\Gamma}_{i,j}(u) \bigcup_{u\in \mathring{T}_\alpha} \Gamma_{\alpha}(u)$$ is a cycle over the $(n-2)$-skeleton and contained in $H_{k}(F^{a}, F^{-\infty})$. 
We must then iterate this procedure on the $n-3, ...,0$ skeleton. 
\end{proof} 

We then apply the Lemma to $F_{k}$, which clearly satisfies the assumptions with $U=U^\delta$. We then get that $\sup {F}_{k}(y , C^-(y)) \leq {h}_{k} (y) + a_k \chi^{\delta}
(y) + \varepsilon$ where $\chi^{\delta}$ is a smooth function with values in $[0,1]$,  equal to one
on $U_{\delta}$ and to zero outside $U_{2\delta}$, $a_k$ is a constant 
and $ \varepsilon$  is arbitrarily small.

Now for $\ell \geq 1$ we want an expression\footnote{Note that $G_{\ell, k}$ actually depends on the choice of $\ell$ and $k$ and not just of the product $\ell k$.} for $G_{\ell,k}$ , a \GFQI     of $\varphi_{\ell k}$, using
 the explicit formula for ${F}_{\ell k}$ obtained as in Lemma \ref{Lemma-5.7} (or rather (\ref{red-normal-4}) in Remarks \ref{red-normal}). We thus get  $$
 G_{\ell,k} (u ,v; \overline x , \overline y, \overline\xi , \eta) =
{ S}(x_{1}+u,v, \eta) +
\frac{1}{\ell} \sum_{j=1}^{\ell} {F}_{k} (\ell x_{j} , y_{j} ,
\xi_{j} ) + B_{\ell} (\overline x ,\overline y)
+ \langle y_{\ell}-v , x_{1} \rangle
$$
and $u\in T^{n}, v \in {\mathbb R}^{n}, \overline x = (x_{1},..., x_{\ell})\in ( {\mathbb R}^{n})^{\ell},   \overline y = (y_{1},...,y_{\ell})\in ( {\mathbb R}^{n})^{\ell}, \overline \xi = (\xi_{1},..., \xi_{\ell})\in (E_k)^{\ell}, \eta \in V$. 

We may now ``spread our  error terms'' $a_k\chi^{\delta}(y)$ by translating them. More precisely, for $j$ from $1$ to $\ell$,   let us choose  domains $U_{j}^{\delta} \subset {\mathbb R}^n$ yielding $\chi_{j}^\delta$ (with $\supp (\chi_{j}^{\delta})\subset U_{j}^{\delta}$),  such that\footnote{Here we measure distance with the norm $\Vert x \Vert_\infty=\sup_{1\leq j \leq n} \vert x_j\vert$, otherwise we get an unpleasant $\sqrt n$ factor.}the following holds : 
\begin{enumerate} \label{grid}
\item the connected components of ${\mathbb R}^n \setminus  U_{j}^{\delta}$ have diameter less than $\delta$
\item any $(n+1)$ distinct sets $U_{j}^{\delta}$ have empty intersection. 
\end{enumerate} 
Such $U_{j}^{\delta}$ may be constructed by taking for $U_{j}^{\delta}$ the $\delta/2$-neighborhood of $S_{j}=\{(x_{1},..., x_{n}) \in {\mathbb R}^{n}\mid \exists k \; ,x_{k}\in 2\delta({\mathbb Z}+r_{j}) \}=2\delta (S + r_j(e_1+\ldots +e_n))$, where 
\begin{enumerate} [(a)]
\item $S=\{ (x_1,....,x_n) \mid \exists k, \;  \;  x_k \in {\mathbb Z} \}$
\item $r_{1}, ..., r_{\ell}$ are distinct elements of $ {\mathbb R} / {\mathbb Z} $ such that $ \vert r_i-r_j \vert \geq \delta$ for $i\neq j$ (so it is understood that $\delta < \frac{1}{\ell}$) 
\end{enumerate} 
We refer to Figure \ref{Fig-1} for a representation of the $U_j^\delta$ when $n=2, \ell =3$. 
Clearly we have $S_{j_1}\cap S_{j_2}\cap .... \cap S_{j_n}\cap S_{j_{n+1}}=\emptyset$ provided $j_i\neq j_k$ for $i\neq k$ and the same holds for their 
 neighborhood:  $U_{j_1}^{\delta}\cap U_{j_2}^{\delta}\cap ...\cap U_{j_n}^{\delta}\cap U_{j_{n+1}}^{\delta}=\emptyset$

We now have 
$$
\sup {F}_{k}(y,C_{j}(y))\leq {h}_{k} (y) + a_{k}
\chi^{\delta}_{j} (y) + \varepsilon \dispdot
$$

Consider
\begin{gather*} 
\overline G_{\ell,k} (u, v; \overline x , \overline y, \eta)=\\
S (x_1+u,v ; \eta )+ \frac{1}{\ell} \sum_{j=1}^{\ell}
\left({h}_{k}(y_{j}) + a_{k} \chi_{j}^{\delta} (y_{j})\right)
+ B_{\ell} (\overline x ,\overline y) + \langle y_{\ell}-v , x_{1} \rangle
\end{gather*} 

 \begin{figure}[H]
 \begin{center}  \begin{overpic}[width=8cm]
 {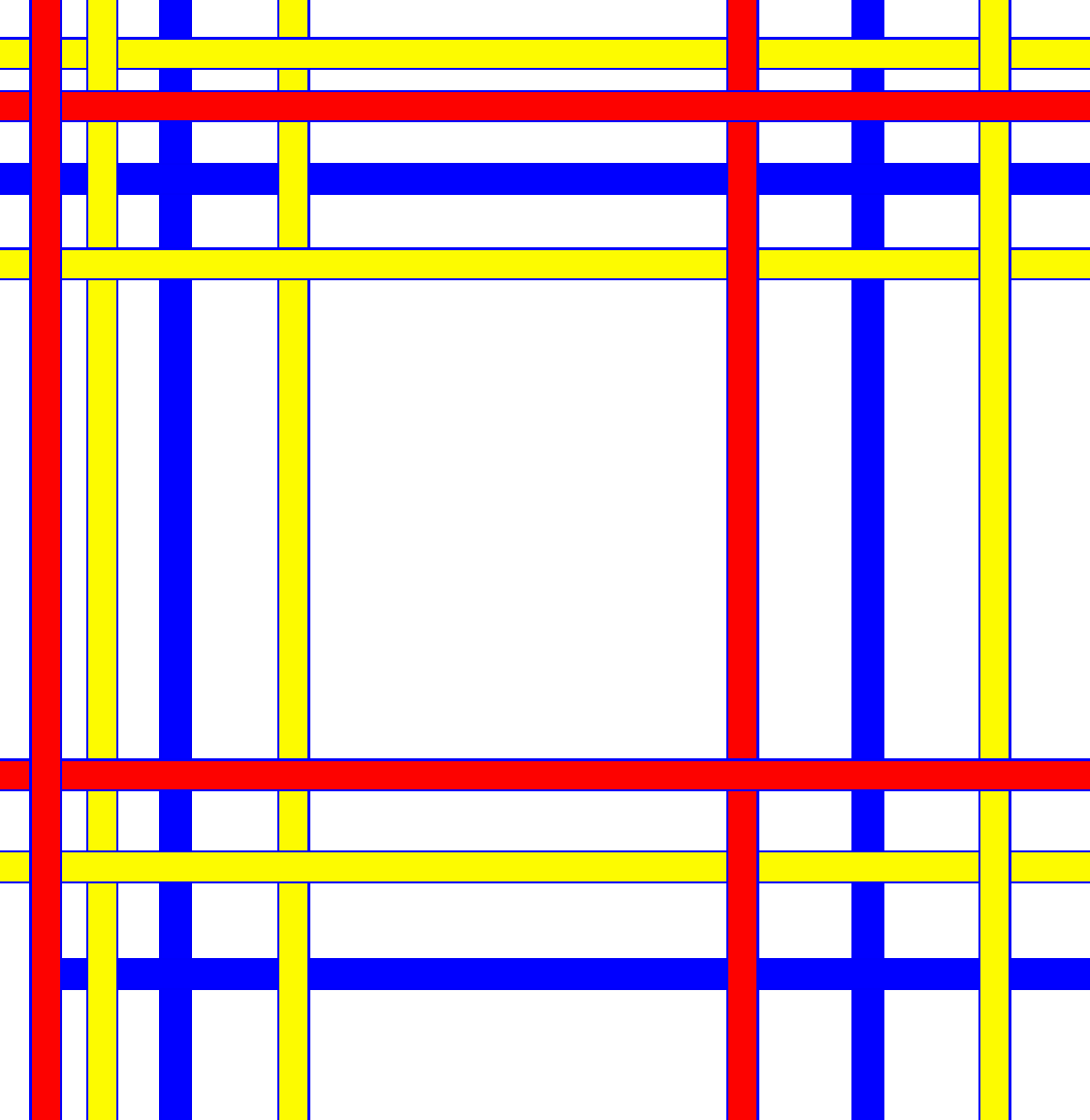}
\end{overpic}
\end{center}
\caption{An example of the sets $U_{1}^{\delta}$ (yellow), $U_{2}^{\delta}$ (red), $U_{3}^{\delta}$ (blue) for $n=2, \ell=3$ ( inspired by P. Mondrian,  \cite{Mondrian}). }
\label{U_s}\label{Fig-1}\end{figure}

Let $\Gamma$ be a cycle  in the homology class of $T^n_{(u,v)}\times \Delta_{x,y}\times V^-$, such that
$$\sup 
\overline G_{\ell,k} (\Gamma) \leq c(\mu , \overline G_{\ell,k}) + \varepsilon
$$
and let
$
(\Gamma \times_{Y} C^-[\ell])$ be defined as 
\begin{gather*} (\Gamma \times_{Y} C^-  [\ell])=\Gamma\times_Y (C_1^-\times ...\times C_\ell^-)= \\ \left\{ (u , v;\overline x, \overline y ,\overline\xi , \eta)
\mid (u , v, \overline x , \overline y, \eta) \in \Gamma , (\ell x_{j}
,\xi_{j}) \in C_{j}^- (y_{j})\right\}\dispdot
\end{gather*} 
Now $(\Gamma \times_{Y} C^-[\ell])$ is in the  homology class of $T^n_{u} \times {\mathbb R}^n_v \times \Delta_{x,y}\times (E_{k}^-)^{\ell}\times V^-$ and we
may thus infer that
$$
c(\mu, \phi_{k}\alpha)= c (\mu ,  G_{\ell,k} ) \leq  \sup G_{\ell,k} (\Gamma \times_{Y}C^-[\ell])
$$
and $ G_{\ell,k}((\Gamma \times_{Y}C^-[\ell])) \leq \overline G_{\ell,k} (\Gamma)$. We may thus conclude that
 $$ c (\mu ,  G_{\ell,k} ) \leq c(\mu, \overline  G_{\ell , k }) + 2 \varepsilon \dispdot $$

Our last step will be to prove

\begin{lemma}\label{5.4}
For each $k$ there is a constant $A_{k}$ such that the following inequality holds for all $\ell$:
$$
c (\mu , {\overline G}_{\ell ,k}) \leq c (\mu , \overline G_{k}) + \frac{A_{k}}{\ell} \dispdot
$$
\end{lemma}

\begin{proof} Indeed $\overline G_{\ell,k}$ is the generating function of
$\psi_{k,\delta,\ell}\; \alpha $ where
$$
\psi_{k,\delta,\ell}= \rho^{-1}_{\ell} \left( \psi^{1}_{k,\delta} \circ \cdots \circ
  \psi^{\ell}_{k,\delta} \right) \rho_{\ell}
$$
and $\psi^{j}_{k,\delta}$ is the time-one flow of ${h}_{k} (y) + a_{k}
\chi_{j}^{\delta} (y)$.

But since the Hamiltonians ${h}_{k} (y) + a_{k}
\chi_{j}^{\delta} (y)$ depend only on $y$, their flows commute, hence $\psi_{k,\delta,\ell}$
is the time-one flow of

$$K_{k,\delta,\ell}(y) =
\frac{1}{\ell} \left( \sum_{j=1}^{\ell} {h}_{k}(y) + a_{k}
  \chi^{j}_{\delta} (y)\right) \dispdot
$$
Now since $(n+1)$ sets $U_{j}^{\delta}$ have empty intersection, we
have that $$\forall y \in {\mathbb R}^n, \;\;  \left \vert \sum_{j=1}^{\ell} \chi_{j}^{\delta}(y) \right\vert
\leq n $$
hence
$$
\left\vert K_{k,\delta,\ell} (y) - {h}_{k}(y) \right\vert \leq \frac{n a_{k}}{\ell} \dispdot
$$
As a result, using the inequality between $\gamma$ and the $C^0$-norm of the Hamiltonian from Proposition \ref{Prop-3.12}, (1), we get
$$
\gamma ( \psi_{k,\delta,\ell} , \psi_{k}) \leq \frac{na_{k}}{\ell}\overset{def}=\frac{A_{k}}{\ell}$$
where $\psi_{k}$ is the time-one flow of ${h}_{k}(y)$ hence
$$
\gamma \left(\psi_{k,\delta,\ell} \alpha ,  \psi_{k} \alpha \right)
\leq \frac{A_{k}}{\ell}
$$
and
\begin{eqnarray*}
c(\mu , \overline G_{\ell, k}) & \leq & c (\mu ,\psi_{k,\delta,\ell} \alpha)\\
& \leq & c (\mu , \psi_{k} \alpha ) +\frac{A_{k}}{\ell}\\
& \leq & c (\mu , \overline\varphi_{k}\alpha) +\frac{A_{k}}{\ell} \dispdot \end{eqnarray*}
\end{proof}

Since by assumption, as $\nu$ goes to infinity  $(\psi_{k_{\nu}})_{\nu \geq 1}$ $\gamma$-converges to
$\overline\varphi$  we get for $k=k_{\nu}$
$$
c(\mu , \varphi_{k\ell} \alpha) = c (\mu ,  G_{\ell,k}) \leq
$$
$$
c (\mu ,\overline G_{\ell,k}) + \varepsilon \leq c (\mu ,\overline G_{k}) +
\varepsilon +\frac{A_{k}}{\ell}
$$
and since this holds for any positive $ \varepsilon $ we have 
$$
\leq c (\mu , \overline\varphi \alpha ) + \frac{A_{k}}{\ell}\dispdot
$$

Taking $\ell$ large enough, we see that there is a sequence\footnote{In fact this holds for any sequence.}  $\ell_{\nu}$ 
$$
\lim_{\nu} c(\mu , \varphi_{\ell_{\nu}\cdot k_{\nu}}\alpha)
$$
$$
\leq c (\mu , \overline\varphi \alpha )
$$
as announced.
This concludes the proof of Proposition \ref{Prop-5.15}.

  \begin{remark}

  It is important to notice that $$\overline G_{\ell , k}=S (x_1+u,v ; \eta )+ \frac{1}{\ell} \sum_{j=1}^{\ell}
\left({h}_{k}(y_{j}) + a_k \chi_{j}^{\delta} (y_{j})\right)
+ B_{\ell} (\overline x ,\overline y) + \langle y_{\ell}-v ,  x_{1} \rangle $$
   cannot be bounded from above by

  $$S (x_1+u,v ; \eta )+ \frac{1}{\ell} \sum_{j=1}^{\ell}
{h}_{k}(y_{j})
+ B_{\ell} (\overline x ,\overline y) + \langle y_{\ell}-v , x_{1} \rangle + \frac{a_k}{\ell }$$
as it is obvious by choosing  $(y_{1},..., y_{\ell})$ such that each $y_{j}$ is in $U_{j}^{\delta}$.
Our proof makes crucial use of the commutation property of the $h_{k}(y)+a\chi_{j}(y)$ and  would not hold if we replaced $\chi_{j}^{\delta}(y)$ by an analogous function $ \chi_{j}^{\delta}(x,y)$.
\end{remark}

\section{Proof of proposition \ref{Prop-5.16}}\label{new-section}

Let $\Gamma_{\varphi}$ be the graph of $\varphi$, where $\varphi$ is written  in coordinates $\varphi(x,y)=(X(x,y),Y(x,y))$ and $S_\varphi$ a \GFQI     for $\Gamma(\varphi)$
$$
\Gamma_{\varphi} = \left\{ (X(x,y) , y , y - Y(x,y) , X (x,y) - x) \mid  (x,y)\in T^{*}T^{n} \right\} \dispdot
$$
Then the reduction of $\Gamma_{\varphi}$ at $y = y_{0}$ is
$$(\Gamma_{\varphi})_{y_{0}} = \left\{ (X (x,y_{0}) , y_{0} - Y
  (x,y_{0})) \mid x \in T^n \right\}$$ that is $L_{y_{0}} - \varphi(L_{y_{0}})$ where
$$
L_{y_{0}} = \{ (x,y_{0}) \mid x \in T^{n} \} \dispdot
$$
Note that if $S_\varphi$ is a \GFQI     for $\varphi$, it is normalized, and this yields a normalization of the \GFQI     for $L_y-\varphi(L_y)$. This could very well NOT be the normalization  expected by the reader (see for example Remark \ref{rem-3.4} \ref{rem-3.4-1}).

For example for $H$ integrable, of the form $H(p)$, we have $L_y=\varphi(L_y)$, so $L_y-\varphi_H(L_y)=0_N$. However the normalization of the generating functions yields $c(\mu, L_y-\varphi_H(L_y))=H(y)$ and not $0$.

Now $c(\alpha_{x} \otimes 1(y),\varphi)=c(\alpha_{x} \otimes 1(y), S_{\varphi})
= c (\alpha , (S_{\varphi})_{y}) = c (\alpha , L_{y} -\varphi(L_{y}))$.

\begin{lemma} 
We have $$c(\mu_x\otimes1(y),\varphi_k)=c(\mu,L_y-\varphi_k(L_y))=-c(1,L_y-\varphi_k^{-1}(L_y))=-c(1_x\otimes1(y),\varphi_k^{-1}) \dispdot $$
\end{lemma} 

\begin{proof} 
Indeed, $c(\mu,L_y-\varphi_k(L_y)))=c(\mu,\varphi_k^{-1}(L_y)-L_y)$ by Hamiltonian invariance\footnote{that is $c(\alpha, L_1-L_2)=c(\alpha, \varphi(L_1)-\varphi(L_2))$ for $\varphi$ a Hamiltonian map. The proof is the same as in \cite{Viterbo-STAGGF} page 695, proof of Proposition 3.5.} of $c(\mu,L_1-L_2)$. Moreover  $c(\mu,\varphi_k^{-1}(L_y)-L_y)= -c(1,L_y -\varphi_k^{-1}(L_y))$ by corollary 2.8 page 693 of \cite{Viterbo-STAGGF}.\end{proof} 

 Denoting by $\widehat h_{k}(y)$ the number  $c(\mu_{x}\otimes 1(y), \phi_{k}^{-1})$,  to prove  Proposition \ref{Prop-5.16}  it is enough to show that $\widehat h_k(y)=c(\mu_x\otimes1(y),\varphi_k^{-1})=-c(1_x\otimes1(y),\varphi_k)$ differs from $-c(\mu_x\otimes1(y),\varphi_k)=-h_k(y)$ by a term of the size $O(1/k)$. This is the content of  
 \begin{proposition}\label{Prop-7.2}
 We have, uniformly in $y$,
$c(\mu_x\otimes 1(y),\varphi_k)-c(1_x\otimes 1(y),\varphi_k)=O(1/k)$. 
\end{proposition} 
\begin{remarks} 
\begin{enumerate} 
\item Using the proof of Conjecture \ref{conjecture} by Shelukhin (see \cite{Shelukhin}) we can prove the above Proposition as follows:   
\begin{gather*} c(\mu_x\otimes1(y),\varphi_k)- c(1_x\otimes1(y),\varphi_k)=c(\mu,L_y-\varphi_k(L_y))-c(1,L_y-\varphi_k^{-1}(L_y))=\\ \gamma (L_y-\varphi_k(L_{y}))= \frac{1}{k}\gamma (L_y-\varphi^k(L_y)) 
\end{gather*} 
But the  $\varphi^k$ have a  fixed compact support, so the same holds for $L_y-\varphi^k(L_y)$ hence according to the main theorem in \cite{Shelukhin}, $\gamma (L_y-\varphi^k(L_y))$ is bounded independently from $k$, and we conclude that $\lim_k\frac{1}{k}\gamma (L_y-\varphi^k(L_y)) 
=0$. 
\item Indeed the  proposition implies  \begin{gather*} \widehat h_{k}(y)=c(\mu_{x}\otimes 1(y), \phi_{k}^{-1})=-c(1_x\otimes1(y),\varphi_k)=\\ -c(\mu_x\otimes 1(y),\varphi_k)+O(1/k)=-h_k(y)+O(1/k) \end{gather*}

Since by definition $c(\alpha,\varphi_k)= c(\alpha,F_k)$ we have to prove that $c(\mu_x\otimes 1(y),F_k)-c(1_x\otimes1(y),F_k)=O(1/k)$. 
According to Proposition \ref{Prop-B3},  if $f$ is defined on an $n$-dimensional orientable manifold, and $\alpha \in H^q(f^b,f^a)$ 
then $$c(\alpha, f) = \sup\{ c(u,f) \mid u\in H_q(f^b,f^a),  \langle \alpha,u\rangle\neq 0\} \dispdot $$
As a result, we may in the sequel replace cohomology classes by homology classes using this property. 
\end{enumerate} 
 \end{remarks} 
\begin{proof}[Proof of Proposition \ref{Prop-7.2}]
First of all, let $u$ be a homology class in $H_d(T^n)$ represented by a map $A: X \longrightarrow T^n$ and $v \in H_1(T^n)$ be  represented by a loop $B : S^1 \longrightarrow T^n$. Denote by $u\cdot v$ the  class represented by $C: (\theta,x) \longrightarrow B(\theta)\cdot A(x)$. Here $x\cdot y$ is the product in the group $T^n$ and  $u \cdot b \in H_{k+1}(T^n)$ is then well defined, i.e. does not depend on the choice of the representatives of $u$ and $v$ : this is the Pontryagin product of $u$ and $v$. Moreover it is easy to see that if $v_1,...,v_n$ form a basis of $H_1(T^n)$, then $v_1\cdot v_2\cdot .... \cdot v_n$ is a nonzero multiple of the fundamental class\footnote{by abuse of language, we denote by $\mu$ both the fundamental homology class and the fundamental cohomology class. This should cause no confusion.}, $\mu$. Since we must compare $c(1_x\otimes 1(y),F_k)$ and $c(\mu_x\otimes 1(y), F_k)$, our proposition will follow from 
\begin{proposition} \label{Prop-7.4}
Let $A:V \longrightarrow T^n_x\times {\mathbb R}^n \times E_k$ representing the class $u\times [\{y\} \times E_k^-] \in H_{d+n(k-1)}(F_{k,y}^c,F_{k,y}^{-\infty})$, and $v\in H_1(T^n)$. Then there exists $C':S^1\times V \longrightarrow T^n\times {\mathbb R}^n \times E_k$ representing $u\cdot v \times [\{y\} \times E_k^-] \in H_{d+1+n(k-1)}(F_{k,y}^{c'},F_{k,y}^{-\infty})$ where 
$$ c' \leq c +O(1/k)\dispdot $$
\end{proposition}  
\begin{proof} In the proof we shall assume $v$ is represented by the loop $s \mapsto s\nu $ where $\nu \in {\mathbb Z}^n$. All degree one classes on the torus can be represented in such a way. 
The first step will be to modify $C$ to a cycle $\widetilde C$ in the same homology class. We denote by  $K$ a number such that $$K> \max\left\{\sup_{(q,p)\in T^*T^n} \vert S(q,p) \vert , \sup_ {(q,p)\in T^*T^n}\vert dS(q,p) \vert \right\}$$ where $S$ is the generating function for $\varphi$ as in Subsection \ref{subsec:reform}.  
We first show that relative cycles in sublevel sets can be deformed to  ``standard form''.
 \begin{lemma} \label{Lem-7.5}
 There exists a cycle $\widetilde C$ homologous to $C$ in $H_d(F_{k,y}^c, F_{k,y}^{-\infty})$  and a constant $M$ such that  
we have
\begin{enumerate}[(i)] \item \label{U}   $\widetilde C \subset \left ( \left\{(q,p) \mid  \max_j \vert p_j \vert  \leq M \right\} \cup F_{k,y}^{-4K}\right ) \bigcap F_{k,y}^c$
\item \label{V}  $\widetilde C \cap F_{k,y}^{-3K} \subset T^n\times \{y\}\times E_k^-$
\end{enumerate} 
 \end{lemma} 
\begin{proof} 
Let us choose $M$ so that  $S$ has support in the set $ \vert p \vert \leq M/2$. We are going to deform $C$ by  first using  the vector field $Z$, associated to the differential equation 
$$\dot q_j = -\chi(\vert p_j \vert ) (p_j-p_{j-1})=X_j(q,p)\; , \; \dot p_j=0\; \text{for}\; 1\leq j \leq k$$
where $\chi (\vert p_j \vert )$ vanishes for $ \vert p_j \vert \leq M/2$.  Note that $y=p_k$ is preserved. As a result if $\psi^t$ is the flow of $Z$, we have  \begin{gather*} \frac{d}{dt}F_{k,y}(\psi^t(q,p))_{\mid t=0}=  \langle dF_{k,y}(q,p) , Z(q,p) \rangle =\left\langle \frac{\partial}{\partial q}F_{k,y}(q,p), X(q,p)\right\rangle =\\ -\sum_{j=1}^{k} \chi( \vert p_j \vert ) \left\langle \frac{\partial}{\partial q_j}F_{k,y}(q,p), (p_j-p_{j-1})\right\rangle =\\ \sum_j \chi( \vert p_j \vert ) \left ( -\vert p_j-p_{j-1} \vert^2+ \left\langle \frac{\partial S}{\partial q_j}(k\cdot q_j,p_j), (p_j-p_{j-1})\right\rangle \right )=\\ - \sum_j \chi( \vert p_j \vert )  \vert p_j-p_{j-1} \vert^2 \end{gather*} 
since $S$ vanishes on the support of $\chi ( \vert p \vert )$. 
Note that the $p_j$ are integrals of the flow of $Z$, that $F_{k,y}$ is decreasing along the flow of $Z$ and that since $p_k=y$ is fixed,  if $\max_j \vert  p_j \vert  \geq M$, for $M$ large enough with respect\footnote{This is where the proof would fail if we did not fix $y$, and wanted to prove the incorrect  statement $\frac{1}{k}\gamma(\varphi^k) \longrightarrow 0$.} to $y$, the quantity $\sum_j \chi( \vert p_j \vert )  \vert p_j-p_{j-1} \vert^2$ is bounded from below by some constant $m_k>0$ (note that $m_k=O(1/k)$, but we don't care). Thus outside the region $\left\{ \max_j \vert p_j \vert  \leq M \right\}$,   $Z$ is a pseudo-gradient for $F_{k,y}$, and since the flow is complete (it is bounded by a linear quantity), it satisfies the following properties
\begin{enumerate} 
\item Setting  $\psi^t(q,p)=(Q,P)$ we have $p=P$. 
\item For $(q,p) \notin \left\{ (q,p) \mid \max_j \vert p_j \vert  \leq M \right\}$ we have  $F_{k,y}(\psi^t(q,p))\leq F_{k,y}(q,p)-m_kt$
\end{enumerate} 
As a result, if $(q,p)\in F_{k,y}^c$ we have $\psi^{ \frac{c+4K}{m_k} }(q,p)\in  \left (\left\{ \max_j \vert p_j \vert  \leq M \right\} \cup F_{k,y}^{-4K}\right ) \bigcap F_{k,y}^c$ (see Figure \ref{fig-7.1} and \ref{fig-7.2}).

We thus obtained  a cycle $\widetilde C_1=\psi^{ \frac{c+4K}{m_k} }(C)$ satisfying (\ref{U}). We now deform $\widetilde C_1$ to $\widetilde C$ as follows. 
Since $\Vert F_{k,y}-B_k \Vert \leq K$ we have $-K\leq c \leq K$, and  the inclusions $$ F_{k,y}^{-4K}\subset B_k^{-3K}\subset F_{k,y}^{-2K}\subset B_k^{-K} \dispdot $$  
Thus $\widetilde C_1 \in H_d(F_{k,y}^{+\infty},F_{k,y}^{-\infty})
= H_d(B_k^{+\infty},B_k^{-\infty})$ and $B_k^{-\infty}\simeq B_k^{-K}$, and $[\widetilde C_1] =[\A \times  \{y\}\times E_k^- ]$ in  $H_d(B_k^{+\infty},B_k^{-\infty})$ 
for some cycle $\A$ representing the class $\a$. 

This means there is a cycle $D$ such that $\partial D= \widetilde C_1 - (\A \times  \{y\}\times E_k^- )$, so $ \partial (D\cap B_k^{-3K})= \widetilde C_1 \cap B_k^{-3K} - \left ((\A \times  \{y\}\times E_k^-)\cap B_k^{-3K}\right ) + D\cap \{B_k=-3K\}$. 

Thus in $B_k^{-3K}$ we may replace $\widetilde C_1$ by the homologous cycle  $\A\times  \{y\}\times E_k^-\cap B_k^{-3K} -  D'$ where $D'\subset  \{B_k=-3K\}$. 
  But then $D'\cap F_{k,y}^{-4K}=\emptyset$ so $$((\A\times  \{y\}\times E_k^-)\cap B_k^{-3K}) \cup D')\cap F_{k,y}^{-4K}=(A\times  \{y\}\times E_k^-)\cap B_k^{-3K} \cap F_{k,y}^{-4K} \dispdot $$
Then the cycle $$\widetilde C=( \widetilde C_1\cap \{B_k\geq -3K\} )\cup D' \cup ( (\A\times  \{y\}\times E_k^-)\cap B_k^{-3K})$$ is homologous to $\widetilde C_1$.

\begin{figure}[H]
\center\begin{overpic}[width=6cm]{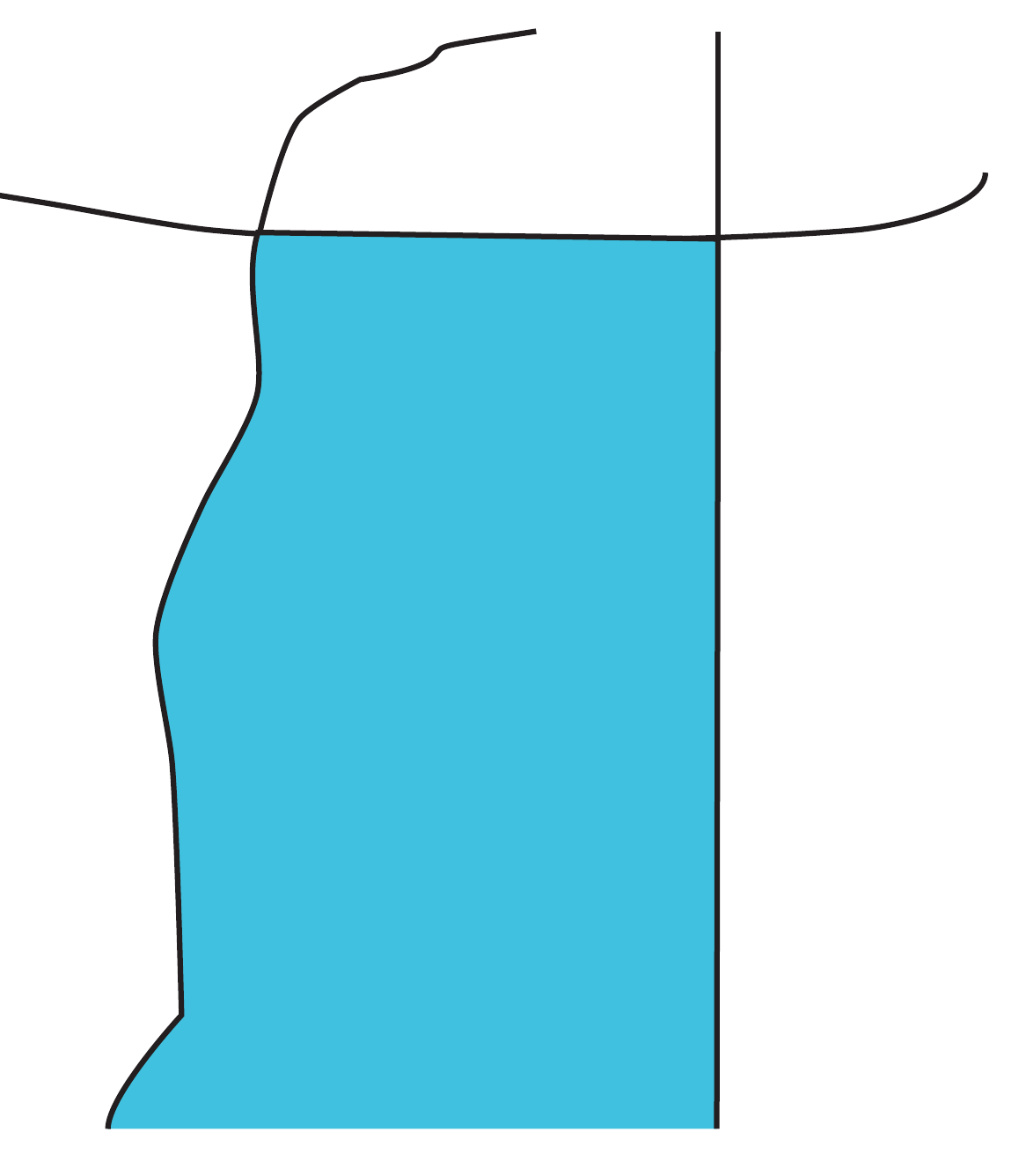}
 \put (35,100) {$\widetilde C_1$}  \put (62,98){$A\times \{y\}\times E_k^-$}
  \put (75,72) {$B_k^{-3K}$}  \put (35,47) {$D$}   
\end{overpic}
\caption{On the left: The cycles $\widetilde C_1$, $\A\times  \{y\}\times E_k^-$ and bounding cycle $D$
}\label{fig-7.1}
\end{figure}

\begin{figure}[H]
\center\begin{overpic}[width=6cm]{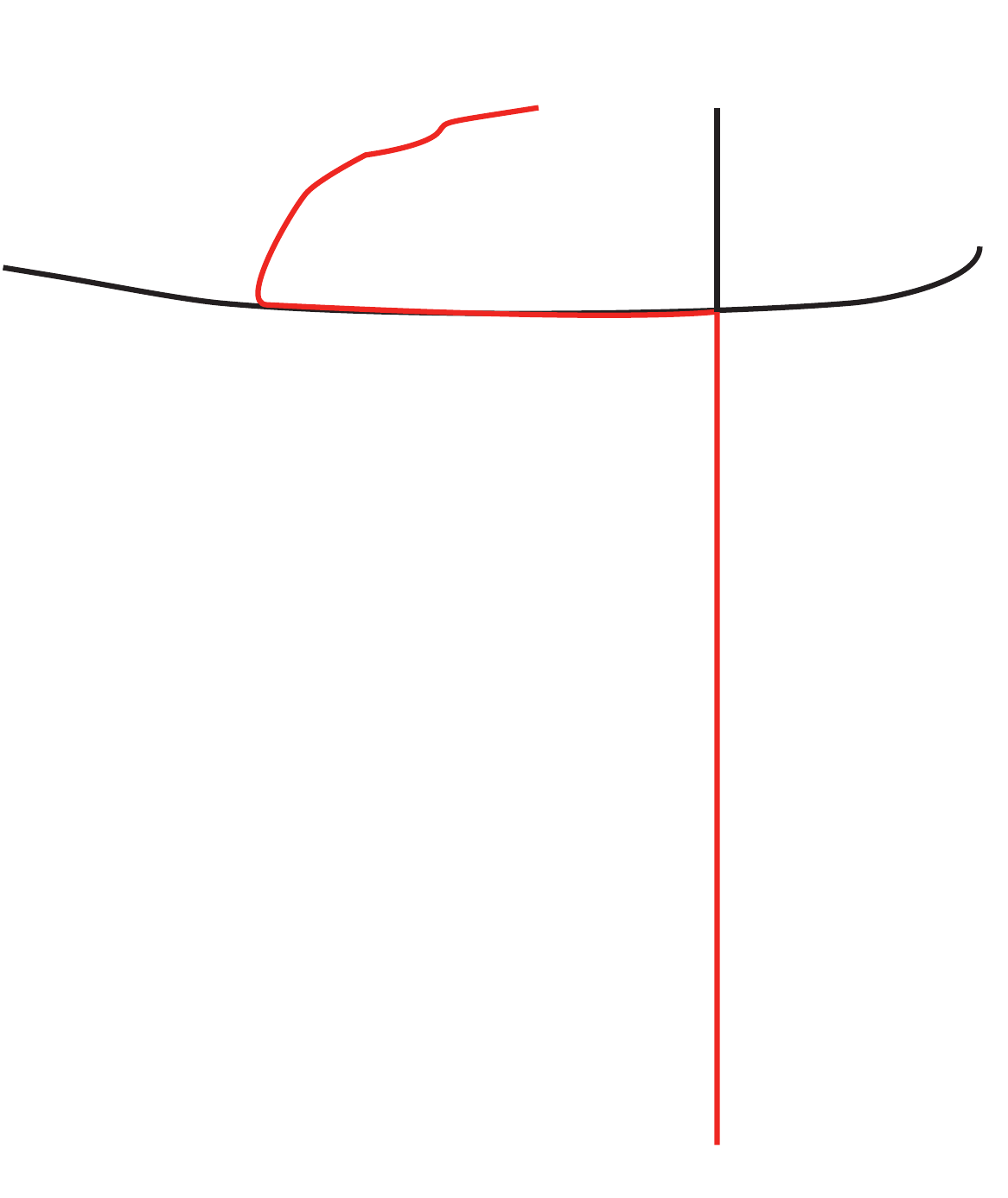}
 \put (35,85) {$(\widetilde C_1)$}  \put (60,85){$A\times \{y\}\times E_k^-$}
  \put (75,68) {$B_k^{-3K}$}  \put (35,67) {$D'\subset \partial D$}   
\end{overpic}
\caption{The cycle $\widetilde C$}
\label{fig-7.2}
\end{figure}

Now, since
\begin{gather*} \sup \{ F_{k,y}(q,p) \mid (q,p)\in D'\cup ((\A\times  \{y\}\times E_k^-)\cap B_k^{-3K})\} \leq \\ \sup \{F_{k,y}(q,p)\mid  B_k(q,p) \leq -3K\} \leq -2K\end{gather*} 
we may  conclude that 
\begin{gather*} \sup\{F_{k,y}(q,p)\mid (q,p)\in  \widetilde C\} \leq \\ \max \{\sup\{F_{k,y}(q,p)\mid (q,p)\in  \widetilde C_1\}, \sup\{F_{k,y}(q,p)\mid (q,p)\in D'\cup (\A\times  \{y\}\times E_k^-)\cap B_k^{-3K}\}\}\leq \\ \max \{c, -2K\}=c \end{gather*} 
Finally $\widetilde C$ satisfies both (\ref{U}), (\ref{V}).
\end{proof} 

Now that we have a cycle in ``standard position'' we are going to construct a representative of $[C]\cdot v$. Remember that $v$ is represented by $s\mapsto s\cdot \nu$ for $\nu \in \mathbb Z^n$. 
We set $$\tau_j^s(q_1,...q_k,p_1,...p_k)=(q_1,...,q_{j-1},q_j+s\nu,q_{j+1},...,q_k,p_1,...,p_k)$$
and $\tau^{(s_1,..,s_k)}=\tau_k^{s_k}\circ...\circ \tau_1^{s-1}$ and $\tau^s=\tau^{(s,s,..,s)}$. 

Then $C_1'= \bigcup_{s\in [0,1]}\tau^s(\widetilde C)$ is a cycle representing $u\cdot v \times \{y\}\times E_k^-$. More generally for any continuous  path $\sigma: [0,1] \longrightarrow [0,1]^k$, such that $\sigma(0)=(0,..,0)$ and $\sigma(1)=(1,..,1)$ we  have, since all such paths are homotopic, that $C'= \bigcup_{s\in [0,1]}\tau^{\sigma(s)}(\widetilde C)$ is homologous to $C'_1$. 

Moreover  define the paths $\sigma_k (s)=(1/k,...,1/k,s-j/k,0,0,...,0)$ for $s\in [j/k, {(j+1)}/k]$  where there are $j$ components equal to $1/k$. Thus $\sigma_k$ joins $(0,...,0)$ to $(1/k,...,1/k)$. Now set $$\sigma(s)=(l/k,...,l/k)+\sigma_k(k\cdot s-l)$$ for $s\in [l/k, {(l+1)}/k]$ for $0\leq l \leq k-1$. Then $\sigma$ connects $(0,...,0)$ to $(1,...,1)$. 
\begin{figure}[H]
\center\begin{overpic}[width=6cm]{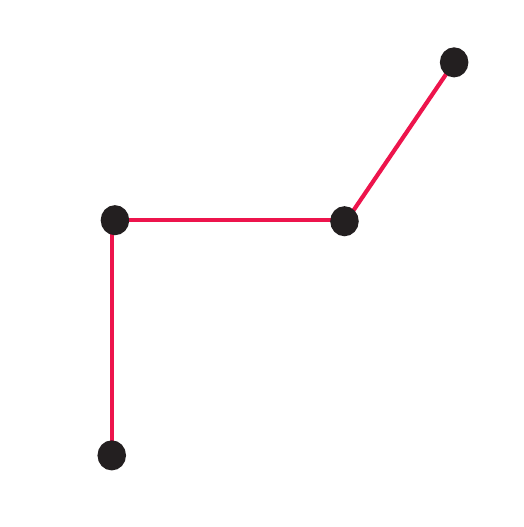}
 \put (15,5) {$(0,0,0)$}  \put (69,52){$(1/3,1/3,0)$}
  \put (75,92) {$(1/3,1/3,1/3)$}  \put (15,65) {$(1/3,0,0)$}   
\end{overpic}
\caption{The path $\sigma_3: [0,1] \longrightarrow [0,1]^3$}
\label{fig-7.3}
\end{figure}

Now we claim  $F_{k,y}(\tau^{1/k}(q,p))=F_{k,y}(q,p)$. This follows clearly from the fact that $q_j \mapsto S(q_j,p_j)$ is $ {\mathbb Z}^n$-periodic and the formula

$$F_{k,y}(q,p)=\frac{1}{k} \sum_{j=1}^k S(k\cdot q_j,p_j) +B_k(q,p)$$  where $$B_k(q,p)=\sum_{j=1}^{k-1}\langle p_{j}, q_{j}-q_{j+1}\rangle + \langle p_{k},q_{k}-q_{1}\rangle  \dispdot $$

Since

\begin{enumerate} 
\item $\sigma$ sends $[0,1]$ in $[0,1]^k$ with $\sigma(0)=(0,...,0)$ and $\sigma(1)=(1,..,1)$. 
\item $\sigma (l/k)=(l/k,...,l/k)$
\end{enumerate} 
and $\tau^{l/k}(\widetilde C)$ still satisfies properties (\ref{U}) and (\ref{V}) from Lemma \ref{Lem-7.5},
it is enough to prove, using  $F_{k,y} \circ \tau^{1/k}=F_{k,y}$, that for any such  $\widetilde C$, we have $\bigcup_{s\in [0,1/k]} \tau^{\sigma (s)}(\widetilde C)=\bigcup_{s\in [0,1]} \tau^{\sigma_k (s)}(\widetilde C) \subset F_{k,y}^{c'}$ with $c'\leq c+ O(1/k)$. 

\begin{lemma}\label{Lem-7.6}
We have 
\begin{enumerate} 
\item \label{aa}$F_{k,y}(\tau_j^s(q,p))= F_{k,y}(q,p)+ \frac{1}{k}\left ( S(kq_j+ks\nu,p_j)-S(kq_j,p_j)\right ) + s \langle p_j-p_{j-1}, \nu \rangle$
\item \label{bb} $F_{k,y}( \tau_{j-1}^{1/k}\circ...\circ \tau_1^{1/k}(q,p))=F_{k,y}(q,p)+ \frac{1}{k} \langle p_{j-1}-p_1,\nu\rangle$
\item \label{cc} $F_{k,y}(\tau_j^{s}\circ \tau_{j-1}^{1/k}\circ...\circ \tau_1^{1/k}(q,p))\leq F_{k,y}(q,p)+ \frac{K}{k} + \frac{ \vert \nu \vert }{k} \left ( \vert p_j-p_{j-1} \vert +\vert  p_{j-1}-p_1 \vert \right )$
\end{enumerate} 
\end{lemma} 

\begin{proof} 
Assertion \ref{aa} is checked immediately. Then \ref{bb} follows at once from the fact that  $S(q_j+\nu,p_j)=S(q_j,p_j)$. Finally \ref{cc} follows immediately from \ref{bb} and \ref{aa}. 

\end{proof} 

Now for $(q,p) \in \widetilde C$ satisfying properties (\ref{U}) and  (\ref{V}) of Lemma \ref{Lem-7.5}, since $\tau_k(s)$ is of the form $\tau_j^t\circ\tau_{j-1}^{1/k}\circ ...\circ\tau_1^{1/k}$ for $0\leq t \leq 1/k$, we have 

\begin{enumerate} 

\item If $(q,p) \in \left\{ \max_j \vert p_j \vert  \leq M \right\}$ we have, using \ref{cc} of Lemma \ref{Lem-7.6}
 $$F_{k,y}(\tau^{\sigma_k(s)}(q,p))\leq F_{k,y}(q,p)+ \frac{K+4M \vert \nu \vert }{k}\leq c+ O(1/k)$$
\item If $(q,p) \in F_{k,y}^{-4K}$ and $(q,p) \in T^n\times \{y\}\times E_k^-$ we have $p_j=-(q_j-q_{j-1})$ hence $F_{k,y}(q,p)=-\sum_j p_j^2+\frac{1}{k}\sum_j S(kq_j,p_j) $ thus using \ref{cc} of Lemma \ref{Lem-7.6} we have 
\begin{gather*} F_{k,y}(\tau_j^{\sigma_k(s)}(q,p))\leq -\sum_j p_j^2+\frac{1}{k}\sum_j S(kq_j,p_j) + \frac{\vert \nu \vert }{k} ( \vert  p_j-p_{j-1} \vert + \vert p_{j-1}-p_1 \vert ) \leq \\ -\sum_j p_j^2+ \frac{\vert \nu \vert }{k} ( \vert p_j \vert +2 \vert p_{j-1} \vert + \vert p_1 \vert ) + K\end{gather*} 

In the path $\sigma_k(s)$  only the $k$-th variable varies from $0$ to $1/k$.   
But since $-3K\geq F_{k,y}(q,p)\geq B_k(q,p)-K $ we get $B_k(q,p)\leq -2K $ so $\sum_j p_j^2 \geq 2K$ and using standard inequalities
$$\sum_j p_j^2- \frac{\vert \nu \vert }{k} ( \vert p_j \vert +2 \vert p_{j-1} \vert + \vert p_1 \vert )-K\geq M- \frac{\sqrt {6M}\vert \nu \vert  }{k}- K\geq K$$ for $M$ large enough. 
And this implies $$F_{k,y}(\tau^{\sigma_k(s)}(q,p))\leq -K+ O(1/k)\leq c + O(1/k) \dispdot $$
\end{enumerate} 

We thus proved that for any $(q,p) \in \widetilde C$ and $s \in [0,1]$, we have 
$$F_{k,y}(\tau^{\sigma_k(s)}(q,p))\leq c'=c +O(1/k) \dispdot $$
Since the map from $S^1 \times \widetilde C \longrightarrow F_{k,y}^{c'}$ given by $(s,q,p) \mapsto \tau^{\sigma_k(s)}$ represents $[\widetilde C]\cdot v$, this concludes our proof and the proof of proposition \ref{Prop-5.16}.
\end{proof} 
\end{proof}

\begin{remarks}
 \begin{enumerate} 
\item The above method of proof, in the case of variational problems for closed geodesics, is related to the ``passing the obstacles one at the time'' that can be found in the paper by \cite{Bangert} page 87, as well as to Gromov's book \cite{Gromov} sections 2.26 and 2.27. I wish to thank V. Bangert for the reference. 
\item Since $\varphi_k=\rho_k^{-1}\varphi^k\rho_k$, and we proved  $\gamma (\varphi_k(L_y)-L_y)$ converges to zero, so for $y=0$, we get $\frac{1}{k} \gamma (\phi^k(0_{T^n}))$ converges to zero. This is much weaker than Conjecture \ref{conjecture}.
\end{enumerate} 
\end{remarks}

\section{Proof of Theorem \ref{Main-theorem}}
 We are now going to prove Theorem \ref{Main-theorem} in the autonomous case. The general case could be proved along the same lines, but we shall show in Subsection \ref{naut}, how to deduce it formally from the time-independent case.

Thus far we showed that some subsequence of $(\phi_{k}^{t})_{k\geq 1}$, $\gamma$-converges to $\overline \phi_\infty^{t}$ and in Proposition \ref{Prop-5.12} that this implies the convergence of the sequence itself to $\overline \phi_\infty^{t}$.  Remember that according to Corollary \ref{Corollary-5.17}  the $\gamma$-convergence of $\varphi_k^t$ to $\varphi_\infty^t$ is uniform in $t$ on compact sets. In other words the sequence of functions $t \mapsto \gamma(\varphi_k^t\overline\varphi^{-t})$ converges uniformly on compact sets to $0$ and thus the sequence $H(kq,p)$ converges to $\overline{H}=\mathcal A (H)$ for the $\gamma$-metric (remember that $\overline H$ is continuous, so belongs to $\gclHamc{T^*T^n}$, the $\gamma$-completion of $C_c^\infty([0,1]\times T^*T^n)$). Clearly $\overline H (q,p)=h_{\infty }(p)$ defined in the previous subsection ( Proposition \ref{Prop-5.14}),  satisfies the first statement  of the Main Theorem.

\begin{proof} [End of the proof of theorem \ref{Main-theorem} (for time independent Hamiltonians)]
Assertion \ref{Main-theorem-2} follows from the fact that $\overline\varphi_\infty^{1}$ determines $\overline H$ (see Appendix \ref{Appendix-A}, Corollary \ref{Cor-app-A}), and that $$\overline\varphi_\infty^{1}=  \lim_{k\to \infty} \rho_{k}^{-1}\phi^{k}\rho_{k}$$ which only depends on $\varphi=\varphi ^{1}$.

 We finally prove assertion \ref{Main-theorem-3}.  Given two compactly supported  Hamiltonians $H_1, H_2$, with $\varphi_1,\varphi_2$ the time-one maps of their  flows setting  $h_{k,i}(p)=c(\mu_{x}\otimes 1 (y), \rho_{k}^{-1}\phi_{i}^{k}\rho_{k})$
\begin{gather*}  \vert { h}_{k,1}(y) - { h}_{k,2}(y) \vert \leq \vert  c(\mu_{x}\otimes 1 (y), \rho_{k}^{-1}\phi_{1}^{k}\rho_{k})-c(\mu_{x}\otimes 1(y), \rho_{k}^{-1}\phi_{2}^{k}\rho_{k}) \vert \leq \\ \gamma \left ( (\rho_{k}^{-1}\phi_{1}^{k}\rho_{k})^{-1} \circ \rho_{k}^{-1}\phi_{2}^{k}\rho_{k} \right ) \leq \gamma (\rho_{k}^{-1}\phi_{1}^{-k}\phi_{2}^{k}\rho_{k}) \leq \frac{1}{k} \gamma (\phi_{1}^{-k}\phi_{2}^{k}) \leq \gamma (\phi_{1}^{-1}\phi_{2})
\end{gather*}
the last inequality follows from Equation (\ref{1of4.7}) in  Lemma \ref{Lem-5.13}. 
Therefore the map  $\mathcal A : \Ham(T^*T^n) \longrightarrow C^0( {\mathbb R}^n, {\mathbb R} )$ is $1$-Lipschitz for the norms $\gamma$ and $C^0$ respectively. It  thus extends to a Lipschitz map from $\widehat \Ham(T^*T^n)$ to $C^{0}( {\mathbb R}^{n}, {\mathbb R} )$.

Since if $H$ only depends on $p$, we have  $\overline H=H$, we get that $\mathcal A$ is a projector.
Finally, if $C$ is the supremum of $ \vert \frac{\partial H }{\partial p}(q,p)\vert$ on $T^*T^n$, the functions $h_k$ defined in Lemma \ref{Lemma-5.10} are $C$-Lipschitz, hence their uniform limit is also $C$-Lipschitz. 
This settles the last claim of the theorem and  concludes our proof of theorem \ref{Main-theorem} for the time-independent case.
\end{proof}
\section{Proof of theorem \ref{Main-properties}}
We assume again $H$ to be  time-independent. 
In order to prove \ref{Main-properties-1} of theorem \ref{Main-properties}, we need to prove that if $H_1 \leq H_2$ then ${ h}_{\infty,1} \leq { h}_{\infty,2}$. This would follow immediately from the fact that we may choose  $S_{1}, S_{2}$ such that\footnote{This is automatically the case if $S_{j}$ is a finite dimensional reduction of the action functional, as described in \cite{Chaperon2} or \cite{Laudenbach-Sikorav}.} $S_{1}(q,p) \leq S_{2}(q,p)$ hence, ${F}_{k,1}\leq {F}_{k,2}$ and therefore  $${ h}_{k,1}(y)=c(\mu\otimes 1(y), {F}_{k,1}) \leq  c(\mu\otimes 1(y), {F}_{k,2})={ h}_{k,2}(y) \dispdot  $$
 As a result, ${ h}_{\infty,1}(y)\leq { h}_{\infty,2}(y)$. 
 
 However there is the following more general and simpler proof, assuming  only\footnote{Remember that $\varphi_1\preceq \varphi_2$ means $c_-(\varphi_2\varphi_1^{-1})=0$, see Remark \ref{rem-4.3}.}  that $H_{1} \preceq H_{2}$ so that $\phi^1_{1}\preceq \phi_{2}^{1}$ and
$$\rho_{k}^{-1}\varphi_{1}^{1}\rho_{k}\preceq \rho_{k}^{-1}\varphi_{2}^{1}\rho_{k}$$ then by going to the limit, $\overline\varphi_{1}\preceq \overline\varphi_{2}$. Now $\overline\varphi_{1}$ and $\overline\varphi_{2}$ are the flows of $\overline H_{1}$ and $\overline H_{2}$ which depend only on $p$.
Therefore, according to Appendix \ref{Appendix-A}, Corollary \ref{Cor-app-A}, they commute, and our assertion follows from the

 \begin{lemma}
 If $\phi^1$, the time-one  flow of the compactly supported integrable Hamiltonian $H(p)$, satisfies $\Id \preceq \phi^{1}$ then $H$ is non-negative.

 \end{lemma}
 \begin{proof}
Remember that $\Id \preceq \varphi$ means $c_-(\varphi)=0$. In  Appendix \ref{Appendix-A}, Proposition \ref{Prop-A.1}, we prove that $c_{-}(\phi^1)= \inf_{p\in {\mathbb R} ^{n}}H(p)$. Therefore if $c_{-}(\phi^{1})$ vanishes, $H$ must be non-negative.
 \end{proof}

To prove \ref{Main-properties-2},  we have to compare ${\mathcal A}(H\circ \psi)$ to ${\mathcal A}(H)$. Note that the flow associated to $H\circ \psi$ is $\psi^{-1}\circ \phi^{t}\circ \psi$. Thus ${\mathcal A}(H\circ \psi)$ is associated to the $\gamma$-limit of
$$\rho_{k}^{-1}\psi^{-1}\phi^k\psi\rho_{k}=(\rho_{k}^{-1}\psi^{-1}\rho_{k}) (\rho_{k}^{-1}\phi^{k}\rho_{k}) ( \rho_{k}^{-1} \psi \rho_{k}) \dispdot $$

 But $\lim_{k\to \infty}\gamma ( \rho_{k}^{-1}\psi^{-1}\rho_{k} )=0$, that is $\rho_k^{-1}\circ \psi^{-1}\circ \rho_k$ $\gamma$-converges to 
 $\Id$.  Hence

 $$\lim_{k\to\infty} \rho_{k}^{-1}\psi^{-1}\phi^{k} \psi \rho_{k}=\lim_{k\to \infty } \rho_{k}^{-1}\phi^{k}\rho_{k}^{-1}\dispdot $$

 For property \ref{Main-properties-3}, we start with the case $c>0$. First if $c$ is a positive integer, $\lim_{k\to \infty} \rho_{k}^{-1}\phi^{ck}\rho_{k}=\lim_{k\to \infty} (\rho_{k}^{-1}\phi^{k}\rho_{k})^c$ and it thus $\gamma$-converges to $\overline \varphi)^c$, that is the flow of $c\overline H$. If $c$ is positive and of the form $1/q$
 $$\lim_{k\to \infty} \rho_{k}^{-1}\phi^{k/q}\rho_{k}= \lim_{\ell\to \infty} \rho_{q\ell}^{-1}\phi^{\ell}\rho_{q\ell}= \rho_q^{-1} \left (\lim_{\ell\to \infty} \rho_{\ell}^{-1}\phi^{\ell}\rho_{\ell}\right ) \rho_q = \rho_q^{-1}\overline \varphi\rho_q$$
 and this is the flow of $ \frac{1}{q}\overline H$. Finally,  
 we have to compare $\lim_{k\to \infty} \rho_{k}^{-1}\phi^k\rho_{k}$ and \newline
 $\lim_{k\to \infty} \rho_{k}^{-1}\phi^{-k}\rho_{k}$. Clearly, since the limits exist, they must be inverses of each other, that is they are given by $\overline\varphi_\infty$ and $(\overline\varphi_\infty)^{-1}$. Now it follows from \cite{Humiliere} or (since we are in the situation of Hamiltonians depending on $p$) from Appendix \ref{Appendix-A} corollary \ref{Cor-app-A} that two integrable autonomous continuous compactly supported Hamiltonians $H,K$  in $\widehat\Ham_{c}$, such that their flows satisfy $\varphi\psi=Id$ in $\widehat {\DHam_c}(T^{*}T^{n})$, must satisfy  $H=-K$.

 We now prove property \ref{Main-properties-4}. We limit ourselves to the case where $U$ is closed and bounded.  Let us consider a decreasing sequence of non-negative smooth functions $(H_{\nu})_{\nu \geq 1}$ such that  $\bigcap_\nu \supp (H_\nu)= U$, and  $\lim_{\nu}H_{\nu}=\chi_{U}$,  where $\chi_{U}$ is the characteristic function of $U$, the limit being here a pointwise limit. Then $\overline H_{\nu}$ is also a decreasing sequence of non-negative continuous functions, and therefore has a limit $\overline H_{\infty}$, and we denote by ${\mathcal A} (U)$ the support of $\overline H_{\infty}$. Since for any 
 other  sequence $(K_{\nu})_{\nu \geq 1}$ decreasing to $\chi_{U}$, we may find, for each $\nu$, a $\mu$ such that $K_{\mu }\leq H_{\nu}$, we have $\overline K_{\infty} \leq \overline H_{\infty}$. By symmetry, we get  $\overline K_{\infty} = \overline H_{\infty}$ hence  the support of ${\overline K}_{\infty}$ coincides with the support of $\overline H_{\infty}$.  This support defines ${\mathcal A}(U)$.
For $U$ bounded but not closed we may set  ${\mathcal A} (U)={\mathcal A} (\overline U)$. In the general case, we set ${\mathcal A} (U)={\mathcal A} (\overline U)= \lim_{V \;\text{bounded in}\; \overline U} \mathcal A (V)$
Assume now $L$ is a Lagrangian submanifold Hamiltonian isotopic to $L_{y_{0}}$. By the Hamiltonian invariance we just proved, ${\mathcal A}(L)={\mathcal A}(L_{y_{0}})$. Now it is easy to show that $${\mathcal A}(L_{y_{0}})=\{y_{0}\}\dispdot $$

  Since $\mathop{shape} (U)$ contains $p$ if and only if $U$ contains a Lagrangian $L$,  Hamiltonian isotopic to $L_{p}$, we get that for $p \in \mathop{shape} (U)$, we must have $ p\in \mathcal A (U)$. This concludes the proof of \ref{Main-properties-4}.

  As for property \ref{Main-properties-5}, it is an easy consequence of the above. Indeed, assume first
  $H(L)\geq h$ where $L$ is Hamiltonian isotopic to $L_{p_{0}}$.  Let  $\kappa_{p_{0}}$ be  a function on $ ({\mathbb R})^n$ equal to $1$ near $p_{0}$, very negative in a neighborhood of the $p$-projection of the support of $H$, and  compactly supported. Then if $\psi(L_{p_{0}})=L$, we have $$H\geq h\cdot (\kappa_{p_{0}}\circ \psi)$$  hence $$\overline H \geq h\cdot (\overline {\kappa_{p_{0}} \circ \psi})=h\cdot  \overline \kappa_{p_{0}}=h\cdot  \kappa_{p_{0}}\dispdot $$

  As a result, $$\overline H(p_{0})\geq h \kappa_{p_{0}}(p_{0})=h\dispdot $$
    Changing $H$ to $-H$, and using \ref{Main-properties-3} we get the second statement.

Property \ref{Main-properties-6} follows from the fact that $c_\pm(\rho_k^{-1}\varphi^k\rho_k)=\frac{1}{k} c_\pm (\varphi^k)$, the fact that $c_\pm$ are continuous for the $\gamma$-topology and Proposition \ref{Prop-A.1}. 

For \ref{Main-properties-7}, since $\mathcal A$ is only defined for compactly supported functions, this means
 that for any sequence $(H_n)_{n\geq 1}$ of compactly supported functions such that $\lim_n H_n=1$, where the limit is uniform on any compact set, we have $\lim_n \mathcal A (H_n)=1$
 
Let us define $K$ to be a smooth non-negative compact supported function equal to $1$ on $[-1,1]$ and set $K_R(p)=K( \frac{ \vert p \vert }{R})$.  

Let $(H_n)_{n\geq 1}$ be a sequence of compactly supported functions converging uniformly on compact sets to $1$, with the extra requirement that there exist sequences   $(R_{n})_{n\geq 1}, (R'_{n})_{n\geq 1}, (\varepsilon_{n})_{n\geq 1}$ such that $\lim_{n}R_{n}=\lim_{n}R'_{n}=+\infty$, $\lim_{n} \varepsilon _{n}=0$ and  $$K_{R_{n}}- \varepsilon_n K_{R'_{n}} \leq H_n \leq K_{R_n}+ \varepsilon_{n}K_{R'_{n}}\dispdot $$ Since $K_R$ only depends on $p$ we have that $K_R,K_{R'}$ commute and ${\mathcal A}(K_R)=K_R$. This implies
 that   $$ \lim_n(K_{R_n}- \varepsilon_nK_{R'_n})\leq \lim_n{\mathcal A}(H_n) \leq \lim_n(K_{R_n}- \varepsilon_nK_{R'_n})\dispdot $$
 Since $\lim_n K_{R_n}=1$ this implies that $\lim_n{\mathcal A}(H_n)=1$.
 
 Finally, to show that $\zeta: H \mapsto \int \mathcal A (H) d\mu(p)$ is a quasi-state, it is enough to deal with the case where $\mu$ is a Dirac mass at $p$. We must then prove

 \begin{enumerate}
 \item (Monotonicity) $H_{1}\leq H_{2}$ implies $\overline H_{1} \leq \overline H_{2}$. This follows from \ref{Main-properties-1} of theorem \ref{Main-properties}
 
  \item (Quasi-linearity) If $H,K$ Poisson commute, then $\overline {(H+K)}(p)=\overline H (p) + \overline K(p)$. this follows from the fact that if $H,K$ commute, with respective flows $\phi^{t}, \psi^{t}$, then $H+K$ has flow $\psi^{t}\phi^{t}$ and then
 ${\mathcal A}(H+K)$ corresponds to
 $$ \lim_{k\to \infty}\rho_{k}^{-1}\phi^{kt}\psi^{kt}\rho_{k}^{-1}= \lim_{k\to \infty}(\rho_{k}^{-1}\phi^{kt}\rho_{k})\lim_{k\to \infty}(\rho_{k}^{-1}\psi^{kt}\rho_{k})= \overline\varphi^t\circ \overline\psi^t$$ and this corresponds to ${\mathcal A}(H)(p)+ {\mathcal A}(K)(p)$ according to Corollary \ref{Cor-app-A} of Appendix \ref{Appendix-A}
 \end{enumerate}
This concludes the proof of theorem \ref{Main-properties} in the time independent case.

\section{Proof of Theorem \ref{Thm-4.5}, the partial homogenization case}

We here consider the case of the sequence defined by  $H_k(x,y,q,p)=H(kx,y,q, p)$ and prove that  it $\gamma$-converges to
${\overline H}(y,q, p)$ obtained by performing the above homogenization, on the variables $(x,y)$ and freezing the $(q,p)$ variables.

The  flow $\Psi_k^t$ of $H(kx,y,q,p)$ is given by

 \begin{gather*}\left \{ \begin{array}{l}\dot x= \frac{\partial}{\partial y}H(k\cdot x,y,q, p)\\ \\
\dot  y=- k \frac{\partial}{\partial x}H(k\cdot x,y,q, p)\\ \\
\dot q=  \frac{\partial}{\partial p}H(k\cdot x,y,q, p)\\ \\
\dot  p=- \frac{\partial}{\partial q}H(k\cdot x,y,q, p)\\ \\
\end{array}\right . \end{gather*}

Set $$x_{k}(t)=k\cdot x(\frac{t}{k}), y_{k}(t)=y(\frac{t}{k}), q_{k}(t)=q(\frac{t}{k}), p_{k}(t)=p(\frac{t}{k})\dispdot $$

We shall consider the flow $\varphi_{k}^t$ associated to the Hamiltonian equations:

 \begin{gather*}\left \{ \begin{array}{l}\dot x_{k}= \frac{\partial}{\partial y}H(x_{k},y_{k},q_{k},p_{k})\\ \\
\dot  y_{k}=- \frac{\partial}{\partial x}H(x_{k},y_{k},q_{k},p_{k})\\ \\
\dot q_{k}= \frac{1}{k} \frac{\partial}{\partial p}H(x_{k},y_{k},q_{k},p_{k})\\ \\
\dot  p_{k}=-\frac{1}{k} \frac{\partial}{\partial q}H(x_{k},y_{k},q_{k},p_{k})\\ \\
\end{array}\right . \end{gather*}

Then the flow $\Psi_k^t$ is given by  $\rho_{k}^{-1}\varphi_{k}^{kt}\rho_{k}$ where $$\rho_{k}(x ,y,q,p)=(k\cdot x,y,q,p)\dispdot $$

Let  $S_{k}(x,y,q,p, \xi)$ be a generating function for the flow above. The candidate for the homogenization is again given by $ \lim_{k \to \infty}\overline H_{k}$ where  $$\overline{H}_{k}(y,q,p) = c(\mu_{x}\otimes 1(y)\otimes 1(q,p),  S_{k})\dispdot $$

is obtained by freezing the $(q,p)$ variables and performing homogenization on the $(x,y)$ variables as in the previous section.
The precompactness of the sequence is proved as in proposition \ref{4.6}, and the same holds for the uniqueness of the limit as in Proposition \ref{Prop-5.12}.
Let us reformulate the problem by considering the symplectic form $\sigma_{k}$ on $T^*T^{m+n}$ given by $\sigma_k=dy\wedge dx +k dp\wedge dq$. For a Hamiltonian $H(x,y,q,p)$ its flow for $\sigma_k$ is defined by the equations

\begin{gather*} \left \{ \begin{array}{ll}  \dot x= \frac{\partial H}{\partial y}(x,y,q,p) ,\quad \dot y= -\frac{\partial H}{\partial x}(x,y,q,p) \\ \\ \dot q=\frac{1}{k} \frac{\partial H}{\partial P}(x,y,q,p) ,  \quad \dot p= \frac{-1}{k}\frac{\partial H}{\partial q}(x,y,q,p)\end{array} \right . \end{gather*}

Note that $\varphi_k^t$ is the flow associated to $H$ for the symplectic form $\sigma_k$. We thus have
\begin{lemma} \label{Lemma-10.1}
The flow $\Psi_k^t$ of $H_k(x,y,q,p)=H(kx,y,q,p)$ is given by $$ \Psi_k^t= \rho_k^{-1}\varphi_k^t\rho_k$$
where $\rho_k(x,y,q,p)=(kx, y ,q,p)$ and $\varphi_k^t$ is the flow of $H(x,y,q,p)$ for the symplectic form $$\sigma_k=dy\wedge dx +k dp\wedge dq \dispdot $$
\end{lemma} 
Now to the above Hamiltonian map $C^{1}$-close to the identity, we may associate the function $S(x,Y,q,P)$ on $T^*(T^{n+m})$ given by
 \begin{gather*}\left \{  \begin{array}{ll}
 X-x=\frac{\partial S}{\partial Y}(x,Y,q,P) , \quad y-Y= \frac{\partial S}{\partial x}(x,Y,q,P) \\ \\ Q-q=\frac{1}{k} \frac{\partial S}{\partial P}(x,Y,q,P) ,  \quad  p-P= \frac{1}{k}\frac{\partial S}{\partial q}(x,Y,q,P)\end{array} \right . \end{gather*}

Indeed this amounts to the identification of $T^*(T^{m+n})\times T^*(T^{m+n})$ endowed with the symplectic form $\sigma_{k} \ominus \sigma_{k}$ (i.e. $\pi_{1},\pi_{2}$ denoting the projections on the first and second $T^*T^n$ factor, $\sigma_{k} \ominus \sigma_{k}$ is defined as $(\pi_{1}^*\sigma_{k}-\pi_{2}^*\sigma_{k}$),  with $T^*(T^{n+m}\times {\mathbb R} ^{n+m})$ endowed with the standard form by

 $$ (x,y,q,p,X,Y,Q,P) \longrightarrow (x,Y,q,P, y-Y,X-x,k(p-P),k(Q-q))\dispdot $$
Note that $S$ depends on $k$, even though it is not apparent in the notation. 
Two such transformations are composed by the following formula: If $S_{1}(x_{1},Y_{1},q_{1},P_{1}), S_{2}(x_{2},Y_{2},q_{2},P_{2})$ are the generating functions for $\phi_1, \phi_2$, we will have the generating function of $\varphi_1\circ \varphi_2$ given by the next formula similar to that of Lemma \ref{Lem-5.2}

 \begin{gather*} S(x,Y,q,P; x_{2},Y_{1},q_{2},P_{1})=S_{1}(x,Y_{1},q,P_{1})+ S_{2}(x_{2},Y,q_{2},P)-
 \\ \langle x-x_{2},Y_{1}-Y \rangle  - k \langle P_{1}-P,q-q_{2} \rangle
 \end{gather*}

Indeed, the constraining equations are

  \begin{gather*}\left\{ \begin{array}{ll}
  \frac{\partial S}{\partial x_{2}}=0 \Longleftrightarrow  \frac{\partial S_{2}}{\partial x}(x_{2},Y,q_{2},P) -Y+Y_{1}=0 \\ \\
 \frac{\partial S}{\partial Y_{1}}=0 \Longleftrightarrow  \frac{\partial S_{1}}{\partial Y}(x,Y_{1},q,P_{1})+ -x+x_{2}=0 \\ \\
\frac{\partial S}{\partial q_{2}}=0 \Longleftrightarrow  \frac{\partial S_{2}}{\partial q}(x_{2},Y,q_{2},P)+k(P-P_{1})=0 \\ \\
   \frac{\partial S}{\partial P_{1}}=0 \Longleftrightarrow  \frac{\partial S_{1}}{\partial P}(x,Y_{1},q,P_{1})+k (q-q_{2})=0
\end{array} \right . \end{gather*}

and the map  $\varphi$ is given by
 \begin{gather*}
 \left (x, Y+   \frac{\partial S}{\partial x}(x,Y,q,P;x_{2},Y_{1},q_{2},P_{1}), q, P+\frac{1}{k} \frac{\partial S}{\partial q}(x,Y,q;P,x_{2},Y_{1},q_{2},P_{1})\right ) \longrightarrow \\ \left  (x+\frac{\partial S}{\partial Y}(x,Y,q,P;x_{2},Y_{1},q_{2},P_{1}), Y, q+\frac{1}{k}\frac{\partial S}{\partial P}(x,Y,q,P;x_{2},Y_{1},q_{2},P_{1}),P\right ) \end{gather*}

that is

 \begin{gather*}
\left (x, Y+  \frac{\partial S_{1}}{\partial x}(x,Y_{1},q,P_{1}), q, P+\frac{1}{k} \frac{\partial S_{1}}{\partial q}(x,Y_{1},q,P_{1})\right ) \longrightarrow \\  \left (x+\frac{\partial S_{2}}{\partial Y}(x_{2},Y,q_{2},P), Y, q+\frac{1}{k}\frac{\partial S_{2}}{\partial P}(x_{2},Y,q_{2},P \right ) \end{gather*}

Now the map $\varphi_{1}$ sends

 \begin{gather*}
\left  (x, Y_{1}+   \frac{\partial S_{1}}{\partial x}(x,Y_{1},q,P_{1}), q, P_{1}+\frac{1}{k} \frac{\partial S_{1}}{\partial q}(x,Y_{1},q,P_{1})\right ) \longrightarrow \\ \left ( x+ \frac{\partial S_{1}}{\partial Y_{}}(x,Y_{1},q,P_{1}),Y_{1}, q+ \frac{1}{k}\frac{\partial S_{1}}{\partial P_{}}(x,Y_{1},q,P_{1}), P_{1} \right )
 \end{gather*}

and the map  $\varphi_{2}$ sends

\begin{gather*}
\left (  x_{2}, Y+  \frac{\partial S_{2}}{\partial x_{}}(x_{2},Y,q_{2},P), q, P+\frac{1}{k} \frac{\partial S_{1}}{\partial q_{}}(x_{2},Y,q_{2},P)\right ) \longrightarrow \\ \left ( x_{2}+ \frac{\partial S_{2}}{\partial Y}(x_{2},Y,q_{2},P),Y, q_{2}+ \frac{1}{k}\frac{\partial S_{2}}{\partial P}(x_{2},Y,q_{2},P), P \right )
 \end{gather*}

Since \begin{gather*} \left \{ \begin{array}{ll}Y=Y_{1}+  \frac{\partial S_{2}}{\partial x}(x_{2},Y,q_{2},P)  \\ \\
 x=x_2 + \frac{\partial S_{2}}{\partial Y}(x_{2},Y,q_{2},P) \\ \\
   P=P_{1}+ \frac{1}{k} \frac{\partial S_{2}}{\partial q}(x_{2},Y,q_{2},P)   \\ \\
   q=q_{2} + \frac{1}{k} \frac{\partial S_{2}}{\partial P}(x_{2},Y,q_{2},P) \end{array} \right . \end{gather*}

we may  infer $$ \varphi= \varphi_1 \circ \varphi_2 \dispdot $$

\subsection{Resolution in the  \texorpdfstring{$(q_{2},P_{1})$}{(q2,p1)} variables}

For $j=1,2$, let the functions $$\frac{\partial S_{j}}{\partial P_{}}(x_{2},Y,q_{2},P) , \; \frac{\partial S_{j}}{\partial q_{}}(x_{2},Y,q_{2},P)$$ be  $C^1$ bounded, and assume $k$ is large. We may then solve

\begin{gather*} \left \{ \begin{array}{ll}
  \frac{\partial S_{2}}{\partial q}(x_{2},Y,q_{2},P)+k(P-P_{1})=0 \\ \\
  \frac{\partial S_{1}}{\partial P}(x,Y_{1},q,P_{1})+k (q-q_{2})=0
\end{array} \right . \end{gather*}

in $$ (q_2,P_1)= \left ( q_2(x,Y,q,P; x_2,Y_1), P_1(x,Y,q,P; x_2,Y_1) \right ) \dispdot $$

This requires the following matrix to be invertible:

$$ \begin{pmatrix}
 I- \frac{1}{k} \frac{\partial^2 S_{2}}{\partial q\partial P}(x_{2},Y,q_{2},P) & \frac{1}{k} \frac{\partial^2 S_{2}}{\partial q^2}(x_{2},Y,q_{2},P)    \\ & \\
 \frac{1}{k} \frac{\partial^2 S_{1}}{\partial P^{2}}(x_{2},Y,q_{2},P)   & I- \frac{1}{k}  \frac{\partial^{2}S_{1}}{\partial P\partial q}(x,Y_{1},q,P_{1})
 \end{pmatrix}
 $$

If we moreover assume that the norm of the inverse matrix is bounded, we get that the map is globally invertible : this is a theorem of Hadamard (see \cite{Hadamard} and \cite{DM-G-Z} for a modern exposition in English).

We thus get, under this assumption,  a new generating function

 $${\widehat S}(x,Y,q,P; x_2,Y_1)= S(x,Y,q,P; x_2,Y_1, q_2(x,Y,q,P; x_2,Y_1), P_1(x,Y,q,P; x_2,Y_1)) \dispdot $$

 Note that $$ \Vert q_2(x,Y,q,P; x_2,Y_1)-q \Vert_{C^1} =O(1/k),\quad \Vert P_1(x,Y,q,P; x_2,Y_1)-P \Vert_{C^1}  = O(1/k)$$

hence $$ \Vert  \widehat S(x,Y,q,P; x_2,Y_1) -S (x,Y,q,P; x_2,Y_1,q,P) \Vert_{C^1} = O(1/k)$$

where $$S (x,Y,q,P; x_2,Y_1,q,P)=S_{1}(x,Y_{1},q,P)+ S_{2}(x_{2},Y,q,P)-
 \\ \langle x-x_{2},Y_{1}-Y \rangle  \dispdot $$
Note that the left-hand side of all these equations depend on $k$, since $S$ itself depends on $k$.  
  \subsection{Generating functions for compositions}

  Suppose now that $S_{1}$ is a function of $(x_{1},Y_{1},q_{1},P_{1}, \xi_{1})$ and $S_{2}$ of
  $(x_{2},Y_{2},q_{2},P_{2}, \xi_{2})$ which are both \GFQI     .

Then   \begin{gather*} S(x,Y,q,P; x_{2},Y_{1},q_{2},P_{1}, \xi_{1}, \xi_{2})=S_{1}(x,Y_{1},q,P_{1},\xi_{1})+ S_{2}(x_{2},Y,q_{2},P,\xi_{2})-
 \\ \langle x-x_{2},Y_{1}-Y \rangle  - k \langle P_{1}-P,q-q_{2} \rangle
 \end{gather*}

The conditions defining $\varphi=\varphi_1\circ \varphi_2$ are then given by

 \begin{gather*} \left \{ \begin{array}{ll}
  \frac{\partial S_{2}}{\partial q_{}}(x_{2},Y,q_{2},P,\xi_{2})+k(P-P_{1})=0 \\ \\
  \frac{\partial S_{1}}{\partial P_{}}(x,Y_{1},q,P_{1}, \xi_{1})+k (q-q_{2})=0
\end{array} \right . \end{gather*}

and for  $k$ large enough we may write, as in the previous section

$$ (q_2,P_1)= \left ( q_2(x,Y,q,P; x_2,Y_1,\xi_{1},\xi_{2}), P_1(x,Y,q,P; x_2,Y_1,\xi_{1},\xi_{2}) \right ) \dispdot $$

We then  set

 \begin{gather*}\widehat S(x,Y,q,P; x_2,Y_1,\xi_{1},\xi_{2})=S_{1}(x,Y_{1},q,P_{1}(x,Y,q,P; x_2,Y_1,\xi_{1},\xi_{2}),\xi_{1})+ \\ S_{2}(x_{2},Y,q_{2}(x,Y,q,P; x_2,Y_1,\xi_{1},\xi_{2}),P,\xi_{2})-
  \langle x-x_{2},Y_{1}-Y \rangle  - k \langle P_{1}-P,q-q_{2} \rangle \end{gather*}

Again, we have, as above

\begin{gather*} \Vert \widehat S(x,Y,q,P; x_2,Y_1,\xi_{1},\xi_{2})-S(x,Y,q,P;x_{2},Y_{1},q,P,\xi_{1},\xi_{2}) \Vert _{C^1}=O(1/k)  \end{gather*}

\subsection{From \texorpdfstring{$\varphi_k^t$}{phikt} to \texorpdfstring{$\varphi_{k\ell}^t$}{phikellt}}

Let  $\varphi_{k}^{t}$ be the flow associated to  $H(x,y,q,p)$ for the symplectic form $\sigma_{k}$. According to Lemma \ref{Lemma-10.1}, the flow $\Psi_{k}^{t}$ associated to $H(kx,y,q,p)$ for $\sigma_{1}$ is given by
$$\Psi_{k}^{t}=\rho_k^{-1}\phi_k^{kt}\rho_k\dispdot $$
Let $F_{k}(x,Y,q,P,\xi)$ be a generating function associated to the time-one flow of $H(x,y,q,p)$, for the symplectic form $\sigma_{k}$, i.e. $\varphi_k^1$. We then have a generating function for $\varphi_{k\ell}$ given by
 \begin{gather*}
 \F_{k,\ell}(x,Y,q,P;\overline x, \overline Y, \overline q, \overline P, \overline \xi)=\\ \sum_{j=1}^{\ell}F_{k}( x_{j},Y_{j},q_{j},P_{j},\xi_{j}) -\sum_{j=1}^{\ell} \langle x_{j}-x_{j+1}, Y_{j}-Y_{j+1}\rangle - k \langle q_{j}-q_{j+1}, P_{j}-P_{j+1}\rangle
 \end{gather*}

Here   \begin{gather*} x_{1}=x, q_{1}=q, P_{\ell}=P, Y_{\ell}=Y, \\ \overline x= (x_{2} ,...,x_{\ell}), \overline q= (q_{2} ,...,q_{\ell}), \overline P= (P_{1} ,...,P_{\ell -1}) , \overline Y= (Y_{1},...,Y_{\ell -1}),  \overline \xi =(\xi_1,..., \xi_\ell) \end{gather*}

The condition for solving the constraints $\frac{\partial F}{\partial {\overline q}} = \frac{\partial F}{\partial {\overline P}}=0$  in  $(\overline q, \overline P)$ is that for $k$ large enough the inverse of the following matrix is bounded

$$ \begin{pmatrix}
 I- \frac{1}{k} \frac{\partial^2 \F_{k}}{\partial \overline q\partial \overline P}(\overline x,Y,\overline q,P) & \frac{1}{k} \frac{\partial^2 \F_{k}}{\partial \overline q^2}(\overline x,Y,\overline q,P)    \\ & \\
 \frac{1}{k} \frac{\partial^2 \F_{k}}{\partial \overline P^{2}}(\overline x,Y,\overline q,P) & I- \frac{1}{k}  \frac{\partial^{2}\F_{k}}{\partial \overline P\partial \overline q}(x,\overline Y,q,\overline P)
 \end{pmatrix}
 $$

This amounts to the inequality \begin{gather*} \tag{$\star$} \frac{1}{k} \left\Vert \begin{pmatrix}
 \frac{\partial^2 F_k}{\partial q\partial  P}(x,Y,q,P) &  \frac{\partial^2 F_k}{\partial q^2}(x,Y,q,P)    \\ & \\
 \frac{\partial^2 F_k}{\partial P^{2}} (x,Y,q,P) &   \frac{\partial^{2}F_k}{\partial P\partial q}(x,Y,q,P)
 \end{pmatrix}\right \Vert \leq \varepsilon \end{gather*}
  since a matrix of the type 

 $$ \begin{pmatrix} I +A&C & 0 &\ldots & \ldots & 0 \\
                                   B& I+A & C &0& \ldots & 0  \\
                                   0 & B & I+A& C&\ldots &0\\
                                   \vdots &\ddots  &  \ddots & \ddots \\
                                   0 & \ldots &0& B& I+A & C \\
                                   0 &\ldots & 0&0& B& I+A

 \end{pmatrix}$$

 is invertible with bounded inverse,  provided  $ \Vert A \Vert ,  \Vert B \Vert ,  \Vert C \Vert $ are small enough\footnote{It holds independently from the number of blocks: this follows from Gershgorin's theorem (see \cite{Gershgorin, Varga}), stating that if $R$ bounds the sum on any line of the  absolute value of the off diagonal terms, then the  eigenvalues of the matrix are at distance less than $R$ from the diagonal terms. The bound on the inverse follows immediately since in our case $\ell$ is fixed and $k$ can be taken arbitrarily large. }.
Indeed, the $ \Vert  \Vert_\infty$ norm of the matrix
$$ \begin{pmatrix} A&C & 0 &\ldots & \ldots & 0 \\
                                   B& A & C &0& \ldots & 0  \\
                                   0 & B & A& C&\ldots &0\\
                                   \vdots &\ddots  &  \ddots & \ddots \\
                                   0 & \ldots &0& B& A & C \\
                                   0 &\ldots & 0&0& B& A

 \end{pmatrix}$$
is bounded by a constant times $\max ( \Vert A \Vert, \Vert B \Vert , \Vert C \Vert )$,  with a constant independent from the number of blocks.  
 Under the above assumption \thetag{$\star$}, we have that

 $$\Vert \F_{k,\ell} (x,Y,q,P;\overline x, \overline Y, \overline q, \overline P, \overline \xi) - \widehat \F_{k,\ell} (x,Y,q,P;\overline x, \overline Y, \overline \xi) \Vert_{C^1}  \leq C\ell  \frac{1}{k}
$$

where

 \begin{gather*} \widehat \F_{k,\ell} (x,Y,q,P;\overline x, \overline Y, \overline \xi)=
\sum_{j=1}^{\ell}F_k( x_{j},Y_{j},q,P,\xi_{j}) -\sum_{j=1}^{\ell} \langle x_{j}-x_{j+1}, Y_{j}-Y_{j+1}\rangle \dispdot \end{gather*}

Now let  us for typographical convenience revert to $(x,y,q,p)$ notation instead of $(x,Y,q,P)$. Let
the generating function associated to $\Psi_k^1$  be given by
$F_k(x,y,q,p; \overline \xi)$.

We thus have according to proposition \ref{Prop-5.15}  for each fixed value of $(q,p)$, a function $h_{k}(y,q,p)$ and a cycle $\Gamma(y,q,p)$ with the proper homology class such that

$$F_{k}(y,q,p; \Gamma (y,q,p)) \leq h_{k}(q,y,p)+ \varepsilon_k$$

where $\lim_{k\to \infty }h_{k}(y,q,p)=h_{\infty}(y,q,p)$.

Moreover $\Gamma(y,q,p)$ may be allowed to depend continuously on $(y,q,p)$ provided we allow the weaker inequality

$$F_{k}(y,q,p; \Gamma_{j}(y,q,p)) \leq h_{k}(q,y,p) +a \chi_{j}^{\delta}(q,p)$$

where now $\chi_{j}^{\delta}$ is supported in $W_{j}^{\delta}$, a $\delta$-neighborhood of a grid in the $(q,p)$ variables. The procedure here is the same as in the proof of proposition \ref{Prop-5.15}. Note that we used proposition \ref{Prop-5.15}, in order to get rid of the $y$ dependence of the $\chi_{j}^{\delta}$. 

Then $$\widehat F_{k,\ell}(x,y,q,p; \overline x, \overline y, \overline \xi ) = \frac{1}{\ell} \sum_{j=1}^{\ell} F_{k}(x_{j},y_{j},q,p,\xi_{j}) - \langle y_{j}-y_{j+1}, x_{j}-x_{j+1}\rangle $$

will satisfy on $$\widehat\Gamma_{k,\ell} = \left\{ (x_{j},y_{j},q,p,\xi_{j}) \mid  (x_{j},\xi_{j}) \in \Gamma_{j}(y_{j},q,p) \right \}$$

   the inequality

  \begin{gather*} \widehat F_{k,\ell}(x,y,q,p; \overline x, \overline y, \overline \xi )  \leq \\  \frac{1}{\ell} \sum_{j=1}^{\ell}h_k(y,q,p)+\frac{a}{\ell} \sum_{j=1}^{\ell}\chi_j(q,p) - \langle x_{j}-x_{j+1}, y_{j}-y_{j+1}\rangle \dispdot \end{gather*}

  As before we choose the $W^{\delta}_{j}$ so that the intersection of $2m+1$ distinct $W^{\delta}_{j} =\supp (\chi_{j})$ is empty.

  Thus   $\widehat F_{k,\ell}$ is bounded by the generating function of $h_{k}(y,q,p)$ up to  $\frac{2ma}{\ell}$.
  We therefore get for all $\alpha$, that $$c(\mu, \Psi_{k\ell}^1\alpha) \leq c(\mu, \overline\Psi_{k}^{1}\alpha) +  \frac{2ma}{\ell} $$
hence $\lim_k c(\mu, \Psi_{k,\ell}^1\alpha) \leq c(\mu, \overline\Psi_{\infty}^{1}\alpha)$. We thus proved the analogue of Proposition \ref{Prop-5.15} in the partial homogenization case. 

Finally, we may conclude the proof of theorem \ref{Thm-4.5} as we did in the standard case for Theorem \ref{Main-theorem}, noting that the proof of Proposition \ref{Prop-7.4} extends to the situation where we have parameters $(q,p)$ without any notable modification.

\section{Proof of proposition  \ref{cor-3.5}}

We shall limit ourselves to the case where homogenization is done on all variables.

According to proposition \ref{B1} from Appendix \ref{Appendix-B}, if $u_{1},u_{2}$ are given by $c(1(q), L_{1})$ and $c(1(q),L_{2})$ respectively, we have $$ \vert c(1(q), L_{1})-c(1(q),L_{2}) \vert \leq  \gamma (L_{1},L_{2})\dispdot $$
Now if $L_{1}=\phi_{1}(\Lambda)$ and $L_{2}=\phi_{2}(\Lambda)$, we have
 $$\gamma (L_{1},L_{2})\leq \gamma (\varphi_{1}\varphi_{2}^{-1}) \dispdot $$

 In our case, $L_k=(\Id\times \phi_{k}^{t})\Delta$ , ${\overline L}=(\Id \times \overline{\phi}^t) \Delta $, and therefore we get $$ \vert u_{k}(t,q)-{\overline u}(t,q) \vert \leq \gamma (\phi_{k}^{t}{\overline \phi}^{-t})\dispdot $$

 Now using the estimate (\ref{1of4.7})  in the proof of Proposition \ref{Prop-5.12}, we see that $$\gamma (\phi_{k}^{mt}{\overline \phi}^{-mt}) \leq m \gamma (\phi_{k}^{t}{\overline \phi}^{-t})$$
 and taking the supremum over $t$ in $[0,1]$, we get, since the flow is autonomous, 
 $$\sup_{t\in [0,m]} \gamma(\phi_{k}^{t}{\overline \phi}^{-t}) \leq m \sup_{t\in [0,1]}\gamma(\phi_{k}^{t}{\overline \phi}^{-t})$$
thus implying
$$ \sup_{t\in [0,m]} \vert u_{k}-\overline u \vert \leq \gamma (L_{k},\overline L)\leq m \varepsilon_{k}
$$
where $ \varepsilon_{k}= \sup_{t\in [0,1]}\gamma (\varphi_{k}^{t}{\overline\varphi}^{-t})$. This concludes the proof of proposition   \ref{cor-3.5}.

\section{Non compact-supported Hamiltonians and the time dependent case}\label {ncs}

\subsection{The coercive case}\label{coercive-case}
Assume first that the autonomous Hamiltonian $H(q,p)$, defined on $T^*T^n$, is not compactly supported, but that $H$ is coercive, that is $$\lim_{ \vert p \vert \to \infty} H(q,p)= +\infty \dispdot $$

Then let $\chi_{A} : {\mathbb R} \to {\mathbb R} $ be a truncation function, that is a smooth function such that

\begin{enumerate}
\item $0\leq \chi_{A}\leq 1$
\item $\chi_{A}$ is supported in $[-2A,2A]$
\item $\chi_{A}=1$ on $[-A,A]$
 \end{enumerate}

 We then consider $\chi_{A}(  \vert p \vert ) H(q,p) = K(q,p)$, and denote by $\varphi^{t}$ the flow of $H$, and by $\psi^t$ the flow of $K$. Since $\varphi^{t}$ preserves $H$, we have that if  
 $a(\lambda), b(\lambda)$ are defined by $a(\lambda)= \inf \{ \vert p \vert \mid \forall q, \; H(q,p) \geq \lambda \}$ and $b(\lambda)= \sup \{ \vert p \vert \mid \forall q,\;H(q,p) \leq \lambda \}$, so that  
 $$W^{a( \lambda )}=\{(q,p) \mid \vert p \vert \leq a( \lambda ) \} \subset  \{ (q,p) \mid H(q,p) \leq  \lambda \} \subset \{ (q,p) \mid \vert p \vert \leq b( \lambda )\}=W^{b( \lambda )}$$
 then $\varphi^{t}$ sends $W^{a( \lambda )}$ into $W^{b( \lambda )}$ thus, for $A \geq b( \lambda )$,  we have $\psi^{t}=\varphi^{t}$ on $W^{a( \lambda )}$. Since $\rho_{k}$ preserves $ W^{a(\lambda)}$ and $W^{b(\lambda)}$, the flow $\varphi_{k}^t =\rho_{k}^{-1}\varphi^{kt} \rho_{k}$ sends also  $W^{a( \lambda )}$  into $W^{b( \lambda )}$ and moreover coincides with $\psi_{k}^{t}$ on $W^{a( \lambda )}$.

We want to conclude that the homogenizations $\overline\phi^t= \lim_{k} \varphi_{k}^t$ and $\overline\psi^t= \lim_{k}\psi_{k}^t$ coincide on $W^{a(\lambda)}$. This  is given by the following result based on the ideas of \cite{Humiliere}, section 4.4. First we shall say that ${\overline \psi}^{\; t}={\overline \varphi}^{\; t}$ on $U$ if and only if  $ ({\overline \psi}^{\; t})^{-1}{\overline \varphi}^{\; t}$ has support in the complement of $U$ (in the sense of \cite{Humiliere} section 4.4 or \cite{Humiliere 2} definition 2.24, page 51).

\begin{definition} [\cite{Humiliere} section 4.4 or \cite{Humiliere 2} definition 2.24, page 51] Let $H \in \gclHamc{T^{*}T^{n}}$. We define 
$\supp(H)$ to be the intersection of closed sets $F$, such that there exists a sequence $H_{n}$ of Hamiltonians supported in $F$, such that $\gamma-\lim_{n}H_{n}=H$. 
 One similarly defines $\supp (\varphi)$ for $\varphi \in \gclDHamc{T^{*}M}$ as the intersection of the closed sets, $F$, such that there exists a sequence $\varphi_{n}$ converging to $\varphi$ such that $\supp (\varphi_{n})\subset F$. 
 \end{definition} 
 
 \begin{lemma}\label{lemma-12.2} Let $\varphi_{k}^{t}, \psi_{k}^{t}$ be two sequences of  smooth Hamiltonian flows, with support contained in a fixed compact set for all $k$. 
 Let $U\subset V$ such  that for any $t$, $\varphi_{k}^{t}(U) \subset V$,   $\psi_{k}^{t}(U) \subset V$ and $\varphi_{k}^{t}= \psi_{k}^{t}$ on $V$. Then if $\gamma-\lim_{k\to \infty }\varphi_{k}^t={ \varphi}_{\infty}^{\; t}$ and
 $\gamma-\lim_{k\to \infty }\psi_{k}^t={\psi}_{\infty}^{\; t}$ and $H_{\infty}, K_{\infty}$ generate $\varphi_{\infty}^t$, $\psi_{\infty}^t$, then we may conclude that  $H_{\infty}-K_{\infty}$ is constant on $U$.
 \end{lemma}

 \begin{proof}Let $H_{k}(t,z), K_{k}(t,z)$ be the compactly supported  Hamiltonians generating $\varphi_{k}^t, \psi_{k}^t$. 
 For $x\in V$, we have $\psi_{k}^t(x)=\varphi_{k}^t(x)$,  hence  $\psi_{k}^{t}$ and $\varphi_{k}^t$ coincide. Then  $\left ((\psi_{k}^{t})^{-1}\circ \varphi_{k}^{t}\right ) _{\mid V}= \Id_{V}$, so the Hamiltonian generating this flow, $L_{k}(t,z)=(H_{k}-K_{k})(t, \psi_{k}^t(z))$ vanishes on $V$. Now by assumption $L_{k}$ $\gamma$-converges to a Hamiltonian generating  $(\psi_\infty^{t})^{-1}\circ  \varphi_\infty^{t}$, thus $\supp ((\psi_\infty^{t})^{-1}\circ \varphi_\infty^{t}) \cap U= \emptyset$. This flow  being generated by $H_\infty- K_\infty$ which is continuous, we have according to \cite{Humiliere 2} (proposition 2.25 page 52), that $ H_\infty- K_\infty$ is constant on $U$. 
 \end{proof}

 Now  if $a(\lambda)\leq A\leq B$, setting $H_{A}=H\cdot \chi_{A}$ (resp. $H_{B}=H \chi_{B}$),  we have  $\overline H_{A}=\overline H_{B}$ (up to constant that can be adjusted) on $W^{a(\lambda)}$ hence we may set

\begin{prop-def} 
Let $H$ be an autonomous coercive Hamiltonian on $T^*T^n$. Then the limit $\overline H = \lim_{A \longrightarrow +\infty}\overline H_{A}$ is well defined. %
\end{prop-def}

Thus any {\bf autonomous} coercive  Hamiltonian can be homogenized:

 \begin{proposition}\label{Prop-12.4}
 The map $\mathcal{A}$ from $\gclHamc{T^*T^n}$ to $C_c^{0}( {\mathbb R} ^n, {\mathbb R} )$ extends to a map defined on the set of autonomous coercive Hamiltonians, i.e. such that $\lim_{ \vert p \vert \longrightarrow +\infty }H(q,p)=+\infty$ with values in $C^0( {\mathbb R}^n, {\mathbb R} )$. Moreover if $H$ is convex in $p$, then so is $\overline H$. 
  \end{proposition}
\begin{proof} 
This follows from lemma \ref{lemma-12.2} applied to the sequence $(H_{N})_{N\geq 1}$. 
According to our truncation argument, the convexity statement needs only to be checked for $H$ of class $C^\infty$ and Tonelli. Then according to Section \ref{Subsection-Mather} Proposition \ref{Prop-13.4}, $\overline H$ coincides with Mather's $\alpha $ function and according to  \cite{CIPP}, theorem A and corollary 1, the $\alpha$ function is given by 
$$\overline H(p)= \inf_{u\in C^1(N, {\mathbb R})} \sup_{q\in N} H(q,p+du(q))$$

Now we have 
\begin{gather*} 
\overline H(tp_1+(1-t)p_2)= \inf_{u\in C^1(N, {\mathbb R})} \sup_{q\in N} H(q,tp_1+(1-t)p_2+du(q)) \leq \\  \inf_{u_1,u_2\in C^1(N, {\mathbb R})} \sup_{q\in N} H(q,tp_1+(1-t)p_2+tdu_1(q)+(1-t)du_2(q)) \leq \\
t \inf_{u_1\in C^1(N, {\mathbb R})} \sup_{q\in N} H(q,p_1+du_1(q))+ (1-t) \inf_{u_2\in C^1(N, {\mathbb R})} \sup_{q\in N}H(q,p_2+du_2(q)) \leq \\ t \overline H(p_1)+ (1-t) \overline H(p_2)
\end{gather*} 
The first inequality is obtained by just setting $u=(1-t)u_1+tu_2$, the second one by convexity of $H$. 
\end{proof}

  \subsection{Non-autonomous Hamiltonians}\label{naut}
Consider now a compactly supported $1$-periodic Hamiltonian $H(t,q,p)$ on $T^*T^n$ and consider the  Hamiltonian $K(t,\tau, q,p)=\tau+H(t,q,p)$. This new Hamiltonian, defined on $T^*(T^{n+1})$ is not compactly supported, but, considering the function $\chi_{A}$ as defined above, the Hamiltonian $$K_A(t,\tau,q,p)=\chi_{A}( \tau+H(t,q,p)) (\tau+H(t,q,p))$$ is {\bf compactly supported and autonomous}.
Its flow preserves the subsets $W^\lambda=\{(t,\tau,q,p) \mid - \lambda \leq  \tau + H(t,q,p) \leq \lambda \}$

Then $K_A$ can be homogenized, and  the same argument as above shows that that ${\overline K}_A = \overline K_{B}$ for $\lambda \leq A \leq B$ on $ \vert \tau \vert \leq \lambda- \Vert H\Vert$. Moreover, we now prove that  $\overline K = \lim_{A\to +\infty} \overline K_{A}$ is of the form $\tau + \overline{H}(p)$ on $ \vert \tau \vert \leq A$: in other words, ${\overline K}(\tau,p)-\tau$ is independent from $\tau$ (and denoted by ${\overline H}(p)$).

\begin{proposition} \label{Prop-12.6}
Let $H(t,q,p)$ be a Hamiltonian on $T^*T^n$, $1$-periodic in time and compactly supported in $(q,p)$. Then $ \lim_{A\to +\infty }\overline K_{A}=\overline K(\tau,p)$ is well defined and there exists ${\overline H}(p)$ such that $\overline K(\tau ,p)=  \tau +\overline H(p)$. The function $\overline K$   satisfies the properties of theorems \ref{Main-theorem} and \ref{Main-properties}. 
\end{proposition} 
                                                                                   
\begin{proof} 
Being of the form $L(t,\tau,q,p)= a \tau + L_{0}(t,q,p)$ is equivalent to the commutation of  $\varphi_{L}^{s}$ and $\psi_{}^s$, where $\psi_{}$ is the flow of $\frac {\partial}{\partial \tau}$, since $\{L, t\}= \frac{\partial L}{\partial \tau}$. Now by assumption on $\vert \tau \vert \leq A- \Vert H \Vert $, we have $\{K_A,t\}=1$, so setting $\rho_{k}(t,\tau,q,p)=(kt, \tau, kq,p)$,  we may write in this region (note that the region is preserved by $\rho_k$)

  \begin{gather*}  \varphi_{\overline K}^s \psi^s= \lim_{k \to +\infty }\rho_{k}^{-1}\varphi_{K}^{ks}\rho_{k} \psi^s =\lim_{k \to +\infty } \rho_k^{-1}\varphi_{K}^{ks}\psi^{ks}\rho_{k}=\lim_{k \to +\infty }
\rho_{k}^{-1}\psi^{ks} \varphi_{K}^{ks} \rho_{k} = \\ \lim_{k\to \infty} \psi^s \rho_{k}^{-1}\varphi_{K}^{ks}\rho_{k} =
\psi^s \varphi_{\overline K}^s
\end{gather*} 
Thus $\varphi_{\overline K}^s$ and $\psi^s$ commute, therefore $\overline K(t,\tau,q,p)=a\tau + \overline H (p)$ for some constant $a$.  The following lemma implies $a=1$.

 \begin{lemma} 
 Let $K(t,\tau,q,p)=\tau+H_{}(t,q,p)$, with flow $\varphi_{}^s$, and let $\psi^s(t,\tau,q,p)=(t,\tau+s,q,p)$. Then $\varphi^s\psi^{\sigma}\varphi^{-s}\psi^{-\sigma}$ is generated by $L_{\sigma}(s,t,\tau,q,p)=K(t,\tau,q,p)-K(\psi^{-\sigma}\varphi^{-s}(t,\tau,q,p))=\sigma$.
 \end{lemma} 
 
 \begin{proof} 
This is just the translation, using \cite{Cardin-Viterbo}, of the fact that $\{K,t\}=1$.
 \end{proof} 
 As a result, if we have a sequence $K_{n}=\tau+H_{n}(t,q,p)$ $\gamma$-converging to $K_{\infty}$, we shall have   $\gamma (\varphi_{n}^s\psi^{\sigma}\varphi_{n}^{-s}\psi^{-\sigma})=\sigma$, hence $$\gamma (\varphi_{\infty}^s\psi^{\sigma}\varphi_{\infty}^{-s}\psi^{-\sigma})=\sigma$$ and this implies $K_{\infty}(t,\tau,q,p)=\tau+H_{\infty}(t,q,p)$.  Note that even after truncation, we have $L_{\sigma}=\sigma$ over a large compact set, and $ \vert L_{\sigma} \vert \leq \sigma$, so that $\gamma (\varphi_{\infty}^s\psi^{\sigma}\varphi_{\infty}^{-s}\psi^{-\sigma})=\sigma$.
\end{proof} 

\begin{remark} 
We could have taken a direct approach to the non-autonomous problem. For this one can replace in subsection \ref{subsec:reform} the generating function $S(q,P)$ for $\varphi^{1/r}$ by a generating function $S(q,P,\xi)$ for $\varphi$. All formulas for the generating function of $\varphi_{k}^{1}$ translate immediately, as well as the proofs of section \ref{Main}.

\end{remark} 
Finally, we may consider the case of a non-autonomous, coercive Hamiltonian. Note that if $H$ is a Hamiltonian equal to a constant $c$ outside a compact set, we may still define $\overline H$ as $c+\overline {(H-c)}$.

\begin{proof} [Proof of Theorems \ref{Main-theorem} and \ref{Main-properties} for non-autonomous Hamiltonians]
The proof of theorems \ref{Main-theorem} and \ref{Main-properties} for general Hamiltonians now follows easily from the autonomous case and Proposition \ref{Prop-12.6}.
We only need to check that homogenization for $K_A$, that is  $\gamma$-convergence of $(K_A)_k$ for all $A$ implies the $\gamma$ convergence of $H_k$. 
\begin{lemma} 
Let $H_n(t,q,p)$ be a sequence of compact supported Hamiltonians (with fixed support) in $T^*T^n$. Assume  for all $A$ we have
$(K_n)_A=\chi_{A}( \tau+H_n(t,q,p)) (\tau+H_n(t,q,p))$  is $\gamma$-converging to $(K_\infty)_A$ such that 
$\lim_{A\to +\infty }\overline K_{A}=\overline K(\tau,p)=\tau+\overline H(p)$. Then $H_n$  $\gamma$-converges to $\overline H$.
 \end{lemma} 
\begin{proof} 
The flow of $K=\tau+H(t,q,p)$ is given by $$\Phi^s(t,\tau, q,p)=(t+s, \tau+H(t,q,p)-H(t+s,\varphi_t^{t+s}(q,p), \varphi_t^{t+s}(q,p))$$ where $\varphi_t^{t+s}$ is the flow of $H$.  If $\chi_A$ is a truncation function  replacing $K$ by $\chi_A(K)$ replaces $\Phi^s(t,\tau,q,p)$ by 
$\Phi_A^s(t,\tau,q,p)$. Notice that $\Phi^s(t,\tau,q,p)=\Phi_A^s(t,\tau,q,p)$ provided $ \vert \tau+H(t,q,p) \vert \leq A$. Now we write the coordinates as $(t,\tau,q,p,t',\tau',Q,P)$ and  the graph $\Gamma(\Phi^1)$ of 
$\Phi^1$ is given by
$$\{ (t,\tau,q,p, t+1,\tau + \Delta\tau(t,q,p),\varphi_t^{t+1}(q,p))\}$$ where $\Delta\tau(q,p)=H(t,q,p)-H(t+1, \varphi_t^{t+1}(q,p))$.  Taking the reduction by $t=0, t'=1$, we get 
$$\{ (q,p,\varphi_0^{1}(q,p))\}= \Gamma (\varphi_0^1) \dispdot $$

As  a result if $ A > \Vert H \Vert _{C^0}$, we have $\Phi^s(t,\tau,q,p)=\Phi_A^s(t,\tau,q,p)$ for $t=0, t'=1$, so the reduction of $\Gamma( \Phi_A^1)$ by $\{t=0,t'=1\}$ coincides with the reduction of $\Gamma (\Phi^1)$, that is $\Gamma(\varphi_0^1)$. 

Now the reduction inequality (\cite{Viterbo-STAGGF}, prop. 5.1 p. 705 ) implies continuity of the reduction for $\gamma$-topology, hence if
$K_n= \tau+H_n(t,q,p)$ is a $\gamma$-converging sequence with limit $K_\infty=\tau+H_\infty((t,q,p)$, then the flow $(\varphi_n)_t^{t+s}$ $\gamma$ converges to $(\varphi_\infty)_t^{t+s}$, in other words $H_n$ is a $\gamma$-converging sequence with limit $H_\infty$. Applying this with $H_n(t,q,p)=H(nt,nq,p)$ and $H_\infty(t,q,p)=\overline H(p)$ and using Proposition \ref{Prop-12.6} to verify that $K_\infty$ is of the form $\tau+\overline H(p)$, we get our theorem in the non-autonomous case. 
\end{proof} 
  Note that in Theorem \ref{Main-properties} only \ref{Main-properties-1} \ref{Main-properties-2} \ref{Main-properties-3} \ref{Main-properties-6} involve the non-autonomous case. 

\end{proof} 
Now we may even extend homogenization to the coercive situation. Let $f_{r}$ such that $f_{r}(x)=x$ for $ x \leq r$ and $f_{r}(x)=r$ for $x\geq r$. 
We then set

\begin{definition} 
Let $H(t,q,p)$ be a coercive non-autonomous  Hamiltonian. Then we set $$\overline H(p)= \limsup_{K\leq H, K \in C_{c}^0([0,1]\times T^*T^n, {\mathbb R} )} \overline K(p)=\lim_{r \to \infty} \overline {f_{r}(H)} \dispdot $$

\end{definition}
 
 \begin{proposition} \label{nonaut-coercive-homogenization}
 The function $\overline H$ is well defined and is lower semi-continuous. If $H$ is convex, so is $\overline H$. 
 \end{proposition} 
 \begin{proof} 
Indeed, as a converging increasing sequence of continuous functions, $\overline H$ is lower semi-continuous. Similarly, if $H$ is convex, given the convex domain $\Omega_R=\{ p \in {\mathbb R}^n)^* \mid \vert p \vert \leq R\} $ we have that the $\overline{f_r(H)}$  are convex in $\Omega_R$  for $r\geq R$ hence according to  Proposition \ref{Prop-12.4}, $\overline H$ is a limit of  functions convex in $\Omega_R$. Since $R$ is arbitrary, we get that $\overline H$ is convex. 
 \end{proof} 
  \begin{remarks}
  \begin{enumerate} 
  \item Because $H(t,q,p)\leq h(p)$ for some function $h$, we have $\overline H \leq h$, so $\overline H$ is locally bounded. We do not know  whether
  $\overline H$ should be continuous. 
  
  However if $ \frac{\partial H}{\partial p}(q,p)$ is bounded, then $\overline H$ is Lipschitz, hence continuous (but most interesting coercive Hamiltonians are superlinear so this does not hold). We do not know an example of coercive  Hamiltonian $H$ such that $\overline H$ is not continuous.  Note that $H$ would necessarily be non-autonomous, non-convex and  superlinear. 
  \item 
  Note that we may also use the distance $\widehat\gamma$ defined by $$\widehat\gamma (\varphi, id)= \sup\{ \gamma (\varphi (L),L) \mid L\in {\mathcal L} \} $$
  and we may also  define  the weak limit as  $\varphi=\lim_k \varphi_k$ if and only if for any $L$ in $\mathcal L$ we have
  $$\lim_k \gamma (\varphi_k(L), \varphi (L))=0 \dispdot $$
Note that this is different from convergence for $\widehat \gamma$, which would require that the above convergence is uniform in $L$.   \end{enumerate} 
  \end{remarks}

 We may now consider applications of the non-compact situation to homogenization of Hamilton-Jacobi equations.
 Indeed, let us consider a Hamiltonian $H(t,q,p)$ on $T^*T^n$, $\varphi^t$ its flow,  and  $f$ a $C^1$ function defined on $T^n$. Since the graph of $df$ is bounded, for any positive time $T$, we may replace $H$ by $K_{A}$ for $A$ large enough, in such a way that $\varphi^t (\Gamma_{df})$ is unchanged for $0\leq t \leq T$. Since $H$ is now compactly supported,  we get
 a function $u_f(t,x)$, and  the variational  solutions  $u_{k,f}(t,x)$ of
\begin{gather*} \tag{$HJ_{k}$}
\left\lbrace
\begin{array}{ll}
& \frac{\partial}{\partial t} u_{k} (t,q) + H (kt, kq,
\frac{\partial}{\partial
x} u_{k}(t,q)) = 0\\
& u_{k} (0,q) = f(q)
\end{array}
\right.
\end{gather*}

 converge to  the variational  solution ${\overline u}_f$ of

  \begin{gather*} \tag{$\overline{HJ}$}
\left\lbrace
\begin{array}{ll}
& \frac{\partial}{\partial t} u(t,q) + {\overline H} (
\frac{\partial}{\partial
x} u(t,q)) = 0\\
& u (0,q) = f(q)
\end{array}
\right.
 \end{gather*}

We may now extend the convergence to the case where $f$ is only $C^{0}$ :
\begin{corollary}   Assume $f \in C^0 ( T^n)$, and we have a sequence $f_\nu \in C^\infty (T^n) $ converging uniformly to $f$. Then $u_{k,f}$ converges uniformly on bounded time intervals  to 
$\overline u_{f} =\lim_\nu \overline u_{f_{\nu}}$.
\end{corollary} 
\begin{proof} Indeed,  we have the estimate:

  \begin{lemma} \label{Lem-12.4} Let $u_f, u_g$ be the variational solutions of the Hamilton-Jacobi equation \thetag{$HJ$}  with initial conditions $f$ and $g$ respectively. Then
  
  $$ \Vert {u}_f - { u}_g \Vert_{C^{0}} \leq  \Vert f-g \Vert_{C^{0}} \dispdot $$
  \end{lemma}

   \begin{proof}
   Indeed, let $\Psi$ be a Hamiltonian diffeomorphism of $T^{*}N$ such that $\Psi (\Lambda_{f})=\Lambda_{g}$, where $\Lambda_{f}=\{(x,df(x)) \mid x \in N\}$, and such that $\gamma (\Psi) \leq \vert f-g \vert _{C^{0}}$. We may take for $\Psi^{t}$ a truncation to a compact domain of the isotopy $\Psi^{t}(x,p)=(x, p+td(g-f)(x))$.

   Then the function $u_{f}$ is obtained as $c(1(x),\phi^{t}(\Lambda_{f}))$, and we have

   \begin{gather*}  \vert c(1(x),\phi^{t}(\Lambda_{f})) - c(1(x),\phi^{t}(\Lambda_{g})) \vert =
   \vert c(1(x),\phi^{t}(\Lambda_{f})) - c(1(x),\phi^{t} \Psi (\Lambda_{f})) \vert \leq \\ \vert c(1(x), \phi^{t}(\Lambda_{f})- \phi^{t} \Psi (\Lambda_{f})) \vert \leq  \vert  c(1(x), \Lambda_{f}- \Psi ( \Lambda_{f})) \vert
   \leq \gamma (\Lambda_{f}, \Psi (\Lambda_{f})) \leq \\ \gamma (\Psi)\leq  \Vert f-g \Vert _{C^{0}}
   \end{gather*}
   \end{proof}
   
   Now this implies, denoting by $u_{k,f_{\nu}}$ the sequence of variational solutions of the equation obtained by replacing $f$ by $f_{\nu}$ in  \thetag{$HJ_{k}$}

   $$\Vert u_{k,f_{\nu}}-u_{k, f_{\mu}}\Vert_{C^0}  \leq \Vert f_{\nu}- f_{\mu}\Vert _{C^0}$$ and $$\Vert  \overline u_{f_{\nu}}- \overline u_{f_{\mu}} \Vert \leq \Vert f_{\nu}- f_{\mu}\Vert _{C^0}$$ hence the sequences $(u_{k,f_{\mu}})_{\mu \geq 1}$ and $(\overline u_{f_{\mu}})_{\mu\geq 1}$ are Cauchy, hence have limits denoted $u_{k,f}$ and $\overline u_{f}$. 
   
   Given a positive $ \varepsilon $, choose $\nu$ large enough, so that $ \vert f - f_{\mu} \vert \leq \varepsilon $ for all $\mu \geq \nu$, 
   and $l$ large enough so that for $k\geq l$
  we have $ \vert u_{k,f_{\nu}}- \overline u_{f_{\nu}} \vert \leq \varepsilon$ we get  $$ \vert u_{k, f_{\mu}}- \overline u_{f_{\mu}} \vert \leq \vert u_{k, f_{\mu}}-u_{f_{\mu}} \vert +  \vert u_{k,f_{\nu}}- \overline u_{f_{\nu}} \vert +
  \vert  \overline u_{f_{\nu}}- \overline u_{f_{\mu}} \vert \leq 3 \varepsilon \dispdot $$

As a result, letting $\mu$ go to infinity, we get $ \vert u_{k,f}- \overline u_{f} \vert \leq 3 \varepsilon $ hence $\lim_{k\to \infty}u_{k,f}=\overline u_{f}$. 
\end{proof} 
 \subsection{A non-coercive example}

Assume for example that $$H(x_{1},x_{2},p_{1},p_{2}) = h(p_{1},p_{2})$$ outside a compact set. Notice that the Poisson brackets, $\{H, p_{1}\}=\{H,p_{2}\}=0$ outside a compact set, therefore $ \{H, \vert p_{1} \vert ^2+ \vert p_{2} \vert ^2\}=0$ outside a compact set. The flow $\varphi^{t}$ of $H$ will then remain inside a bounded domain $W^\lambda$ for $\lambda$ large enough. We may then use the same truncation method as above, and infer that we may homogenize $H$:

\begin{proposition}\label{Prop-12.15}
Let $H(x_{1},x_{2},p_{1},p_{2}) = h(p)$ outside a compact set. Then we have a homogenization operator $\mathcal A$ with the same properties as in the compactly supported case.
\end{proposition}

\begin{corollary} Assume $u_{k}$ is a variational solution of Hamilton-Jacobi equation \thetag{$HJ_{k}$} where $H$satisfies the conditions of Proposition \ref{Prop-12.15}.
Then the sequence $u_{k}$ converges uniformly to a solution $\bar u$ of \thetag{$\overline{HJ}$}.
\end{corollary}

\begin{remark} By an approximation method, this will work for any Hamiltonian such that  $$\lim_{ \vert p \vert \to \infty } \vert H(q,p) - h(p) \vert =0\dispdot $$
\end{remark}

\section{Homogenization in the \texorpdfstring{$p$}{p} variable and connection with Mather's \texorpdfstring{$\alpha$}{alpha} function }
\subsection{Homogenization in the \texorpdfstring{$p$}{p} variable}

This problem of homogenization in the $p$ variable corresponds to  the singular perturbation or penalization problem, studied in recent years by several authors (see for example \cite{Alv-Bar1, Alv-Bar2}).

 Let us consider a Hamiltonian $H(q,p)$ which is either compact supported or coercive. The sequence defined by $H_{k}(q,p)=H(q, k\cdot p)$. Its flow is given by $$\psi_{k}^{t}=\zeta_{k}^{-1}\phi^{kt}\zeta_{k}$$ where $\zeta_{k} (q,p)=(q,k\cdot p)$.

 Note that here $\zeta_{k}$ is a {\it bona fide} map on $T^*T^n$, so that we do not have to invoke covering arguments. Since $\zeta_{k}$ satisfies $\zeta_{k}^*\omega= k \omega$, we get, that

 $$\gamma (\zeta_{k}^{-1}\phi^{kt}\zeta_{k})=\frac{1}{k} \gamma (\phi^{kt})$$

 There is {\it a priori} no limit for the sequence $\psi_k^t=\zeta_{k}^{-1}\phi^{kt}\zeta_{k}$: indeed if $\phi^{t}$ is the flow of $H(p)$, $\zeta_{k}^{-1}\phi^{kt}\zeta_{k}$ will be the flow of $H(kp)$. However let us write $\tau_{k}(q,p)=\rho_k\circ \zeta_k^{-1}(q,p)=(k\cdot q, \frac{p}{k})$, then $\zeta_k=\rho_k\circ \tau_k^{-1}$ and  $$\psi_{k}^{t}=\zeta_{k}^{-1}\phi^{kt}\zeta_{k}= \tau_{k}^{}\rho_{k}^{-1}\phi^{kt}\rho_{k}\tau_{k}^{-1}= \tau_{k}^{}   \phi_{k}^{t} \tau_{k}^{-1} $$

 Now $$\gamma (\varphi_{k}^{t} {\overline \varphi}^{\; -t}) \leq \varepsilon _{k}t   $$
 thus\footnote{We use here that $\gamma (\tau_k\varphi \tau_k^{-1})=\gamma (\varphi)$. Note that $\tau_k$ is not Hamiltonian, as it is not even isotopic to the identity. However if we lift $\varphi$ to $\widetilde \varphi : T^*{\mathbb R}^n \longrightarrow T^* {\mathbb R}^n$, and $\tau_k$ also obviously lifts to a Hamiltonian map of $T^* {\mathbb R}^n$ given by $\widetilde\tau_t: (q,p) \mapsto (tq,\frac{p}{t})$, and $c_\pm (\tau_k\varphi\tau_k^{-1})$   belong to the set of actions of the fixed points of  $(\tau_k\varphi \tau_k^{-1})$ contained in the set of actions of  $\widetilde\tau_k \widetilde\varphi \widetilde \tau_k^{-1}$. But since the set of actions of fixed points of $\widetilde\tau_t \widetilde\varphi \widetilde \tau_t^{-1}$ is constant, the result follows by taking $t=k$. }
 
 $$ \gamma \left  ( (\tau_{k}^{} \varphi_{k}^{t} \tau_{k}^{-1})( \tau_{k}^{} {\overline\varphi}^{\;-t} \tau_{k}^{-1})\right ) = \gamma (\tau_{k}^{} ( \varphi_{k}^{\; t} {\overline\varphi}^{-t}) \tau_{k}^{-1}) = \gamma (  \varphi_{k}^{t}  {\overline\varphi}^{\;-t} )  \leq \varepsilon_{k} t \dispdot $$

 Now since $ \tau_{k}^{-1}\overline\phi^{\; t}\tau_{k}$ is generated by ${\overline H} (k\cdot p )$, we do not get a limit for $H(q,k\cdot p)$ but we get:

\begin{proposition}\label{10.1}
Let $H$ be an autonomous Hamiltonian which is either compact supported or coercive. 
 $$\lim_{k\to \infty} \gamma (H(q,k\cdot p) , {\overline H}(k\cdot p) ) =\lim_{k \to \infty} \gamma (\psi_{k}^{t}(\overline \phi)^{-kt})=0$$
\end{proposition}

In spite of the fact that $\overline H (k\cdot p)$ has no limit as $k$ goes to infinity, the above proposition has a number of applications.
 First, let us consider
 the standard parabolic Hamilton-Jacobi equations
  \begin{gather*}
\left\lbrace
\begin{array}{ll}
 \frac{\partial}{\partial t} u(t,q) + H (q,
\frac{\partial}{\partial
q} u(t,q)) = 0 \\ u(0,q)=f(q).
\end{array}
\right.
\end{gather*}

Set $v_{k}(t,q)= \frac{1}{k}u(k\cdot t,q)$. This is now a solution of 

\begin{gather*}  \left\lbrace \begin{array}{ll} \frac{\partial}{\partial t} v_{k}(t,q) + H (q,
k\frac{\partial}{\partial
q} v_{k}(t,q)) = 0 \\  v_{k}(0,q)=\frac{1}{k}f(kq). \end{array}\right .
\end{gather*}

and since $\lim_{k\to \infty}\gamma (H(q,k\cdot p), \overline H(k\cdot p))=0$ and $\lim_{k\to \infty}\frac{1}{k}f(kq)=0$ we get that 
$v_{k}$ is approximated by $w_{k}$, variational solution of 

\begin{gather*}  \left\lbrace \begin{array}{ll} \frac{\partial}{\partial t} w_{k}(t,q) + \overline H (
k\frac{\partial}{\partial
q} w_{k}(t,q)) = 0 \\  w_{k}(0,q)=0. \end{array}\right .
\end{gather*}

that is $\lim_{k\to \infty} \vert w_{k}-v_{k} \vert =0$.  Now, it is clear that $w_{k}(t,q)=-t\overline H(0)$, so that 
we get
the following result, which had been proved in \cite{L-P-V} in the $p$-convex one-dimensional case. 
\begin{proposition} \label{prop-9.2}
Let $u$ be a variational solution of \begin{gather*} \tag{$HJ$}\left\{
\begin{array}{ll}
 \frac{\partial}{\partial t} u (t,q) + H (q,
\frac{\partial}{\partial
q} u(t,q)) = 0 \\ u(0,q)=f(q) &
\end{array}
\right.
\end{gather*}

then $$\lim_{t\to \infty} \frac{1}{t} u(t,q)= -\overline H(0)$$
\end{proposition}
In the general convex case this result is due to the fact that the solutions of 
\thetag{$HJ$} are defined by the
Lax-Oleinik semi-group $T^t$ and that for $u_0$ the viscosity solution of $H(x,du_0(x))=-\overline H(0)$ we have $T^tu_0+\overline H(0)t=u_0$ and for all $u$
$ \vert T^t u - T^t u_0 \vert \leq \vert u- u_0 \vert $ so that $ \vert T^tu+\overline H(0)t \vert$ is bounded (see \cite{Fathi1, Fathi2}). 

\subsection{Connection with Mather \texorpdfstring{$\alpha$}{alpha} function}\label{Subsection-Mather}
We start with 
\begin{definition} \label{Tonelli-Lagrangian}
The $C^2$ Lagrangian $L(t,q,\xi)$ is said to be a Tonelli Lagrangian if it is strictly convex and coercive, that is $ \frac{\partial^{2}}{\partial \xi^{2}}L(t,q,\xi) > \varepsilon \Id$ for some $ \varepsilon >0$ and the Euler-Lagrange flow defined on $TN$ by $$ \frac{d}{dt} \frac{\partial }{\partial \xi }L(t,q,\xi) - \frac{\partial }{\partial q}L(t,q,\xi)=0$$ 
is complete. 
\end{definition} 
Of course if $H$ is the Legendre dual of $L$ that is 
$$H(t,q,\xi) = \inf_\xi \left \{ \langle p,\xi\rangle - L(t,q,\xi)\right \}$$
 this definition is equivalent to requiring that  the flow $\varphi^{t}$ of the corresponding Hamiltonian is complete. 

For a Tonelli Lagrangian $L(t,x,\xi)$,  $1$-periodic in $t$ and strictly convex in $p$, the $\alpha$ function has been defined by Mather in \cite{Ma}  as
$$\alpha (p)= \lim_{T\to \infty}\frac {1}{T}\inf_{q\in C^1([0,T],N)}\left\{  \int_0^T L(t,q(t), \dot q (t))dt - \langle p, x_1-x_0 \rangle \mid q(0)=x_0, q(T)=x_1 \right\}$$
Note that in the above, the infimum is for $x_0, x_1$ free to vary. 

 As a special case, we may show
 \begin{proposition}\label{Prop-13.4}
Let $H$ be the Legendre dual of the Lagrangian $L$, i.e. $H$ is strictly convex in $p$ and $$L(t,q,\xi)=\sup_{p\in T_q^*N} \left\{ \langle p, \xi \rangle -H(t,q,p)\right \} \dispdot $$ Then $${\overline H}(p)=\lim_{T \to \infty} \inf_{q\in C^1([0,T],N)} \left\{ \frac{1}{T} \int_0^T L(t,q(t), \dot q (t))dt - \langle p, x_1-x_0 \rangle \mid q(0)=x_0, q(T)=x_1 \right\} \dispdot $$
As a result $\overline H$ coincides with Mather's $\alpha$ function. 
 \end{proposition}

 \begin{proof}
 Replacing $L(t,q,\xi)$ by $L(t,q,\xi)+ \langle p, \xi \rangle$, it is enough to consider the case $p=0$.
 Then let $${\mathcal P}_{t}=\{q:[0,t] \longrightarrow M \}$$ and $\pi:{\mathcal P}_{t} \longrightarrow M$ the map $q \mapsto  q(t)$. Let $$E_{t}(q)=\int_{0}^t L(s, q (s), \dot q (s)) ds$$ be defined on ${\mathcal P}_{t}$, and consider\footnote{As Frol Zapolsky pointed out, one must first do a finite dimensional reduction of $E_{t}$, for example using a broken geodesic method see \cite{Chaperon2}. This is done in Appendix \ref{Appendix-Lag}.} $E_{t}$ as a \GFQI    . We shall write $(x_{1},q)$ instead of $q$ to remind the reader that $\pi(q)=q(t)=x_{1}$.

 Now  \begin{gather*} DE_{t}(x_{1},q)= \int_{0}^t\left [\frac{\partial L}{\partial x}(s,q(s),\dot q (s))- \frac{d}{dt}\frac{\partial L}{\partial \xi}(s,q(s),\dot q (s))\right ] \delta q(s) ds +\\  \frac{\partial L}{\partial \xi}(t,q(t),\dot q (t)) \delta q(t)-
 \frac{\partial L}{\partial \xi}(0,q(0),\dot q (0)) \delta q(0) \end{gather*}

 Setting $$p(t)= \frac{\partial L}{\partial \xi}(t, q(t),\dot q (t))$$  we get
 $(x_{1}, \frac{\partial E_{t}}{\partial x_{1}})= (x_{1}, p_{1}) = \phi^{t}(x_{0},0)$. Therefore $E_{t}$ is a \GFQI     of $\phi^{t}(0_{N})$,
 and since $$u_L(t,x)=\inf\{ E_{t}(x,q)\mid   q \in {\mathcal P}, q(1)=x \}=c(1(x),E_{t})$$  is a variational solution of 
  \begin{gather*} \tag{$HJ$}\left\{
\begin{array}{ll}
 \frac{\partial}{\partial t} u (t,q) = H (t,q,
\frac{\partial}{\partial
q} u(t,q)) = 0 \\ u(0,q)=0 &
\end{array}
\right.
\end{gather*}

and  we  proved in Proposition \ref{prop-9.2} (note that here the equation has a different sign in front of $H$)

 $$\lim_{t\to \infty} \frac{1}{t}u_{L}(t,x) =\overline H(0).$$
this concludes our proof. 
 \end{proof}
As a consequence we get
\begin{corollary}[P. Bernard, \cite{Bernard2}])
The Mather $\alpha$ function is symplectically invariant. 
\end{corollary} 

\section{More examples and applications}
 \subsection{Homogenization of \texorpdfstring{$H(t,q,p)$}{H(t,q,p)} in the variable \texorpdfstring{$t$}{t}}

  Applying partial homogenization's Theorem \ref{Thm-4.5}  to the $t$ variable for a $T$-periodic Hamiltonian $H(t,q,p)$ defined on $T^*T^n$, we obtain an autonomous Hamiltonian $\overline H(q,p)$. However, this is nothing else than $$\overline H(q,p)= \frac{1}{T}\int_{0}^{T}H(t,q,p) dt$$

  Indeed, if $H(kt, q,p)$ has flow $\phi_{k}^{t}$, we have,  by the fundamental theorem of classical averaging (see \cite{B-M} page 429, \cite{S-V} theorem 3.2.10 page 39): $$\lim_{k\to \infty} \phi_{k}^{t} = \overline \phi^{t}$$ in the $C^0$ topology, where $\overline \phi^{t}$ is the flow of $$\frac{1}{T} \int_{0}^{T} X_{H}(t,q,p) dt =X_{\overline H}(q,p).$$
  Since according to \cite{Viterbo-STAGGF}(proposition 4.15, page 703), $C^0$-convergence implies $\gamma$-convergence\footnote{This is proved in $ {\mathbb R}^{2n}$ in the quoted reference, but is easily extended to the torus case. Indeed the proof is based on the lemma, page 703, which states the existence of a compactly supported Hamiltonian diffeomorphism such that each point of the unit disc bundle is displaced by at least $ \varepsilon $. This can be constructed by taking the Hamiltonian flow associated to the Hamiltonian $H(q,p)=\chi ( \vert p \vert ) \langle p,\xi(q) \rangle$, where $\chi$ is a cutoff function, $\xi(q)$ a non-vanishing vector field on $T^n$. }, our claim follows.

  \subsection{The one dimensional case}

 In \cite{L-P-V}, the computation of $\overline H$ in the case $H(q,p)= \vert p \vert ^2 - V(q)$ and for  $V$  bounded from below is explicitly dealt with. 
 Indeed,  assuming $V\geq 0$ is one-periodic and vanishes at least at one point, we have

 \begin{equation} \tag{$\star$}\begin{aligned} \left \{ \begin{array}{ll}
 \overline H (p)=0 & \;\text{if}\quad  \vert p \vert  \leq \int_{0}^{1}( V(q) )^{1/2}dq \\ \\  \overline H(p)=\lambda &\;\text{where}\; \lambda \;\text{solves}\; \vert p \vert = \int_{0}^{1}\left ({V(q)+\lambda }\right ) ^{1/2} dq  \;\text{otherwise} 
   \end{array}\right .
 \end{aligned} \end{equation} 

 Indeed, according to theorem \ref{Main-properties} property \ref{Main-properties-5}, if we can find a curve $L$ in $T^*S^1$ such that $ \vert p \vert ^{2}-V(q) \leq h$ on $L$ and $\int_{L}pdq=v$ then $\overline H (v) \leq h$. Since $H(q,p) \leq h$ contains $L_{v}=\{(q,p(q)) \mid p= (V(q)+h)^{1/2} \}$ and this is the graph of a Lagrangian submanifold with integral of the Liouville class $$v=\int_{0}^{1} (V(q)+h)^{1/2} dq$$ we get that $\overline H(v) \leq h$. But the  Lagrangian, $\{(q,p) \mid p= (V(q)+h)^{1/2}+ \varepsilon \} $  is contained in  $H(q,p) \geq h$, hence
 $\overline H(v)\geq h$. By the monotonicity property ( \ref{Main-properties-1} of theorem \ref{Main-properties}) $\overline H(v) \geq 0$, and this proves \thetag{$\star$}.

In higher dimension,  since $\{(q,p+du(q)) \mid q \in T^{n}\}$ is Lagrangian and Hamiltonian isotopic to $L_{p}$, the fact  that if $u(q)$ is a smooth function such that $ H(q, p+ du(q))\leq h$ then $\overline H(p) \leq h$ is a useful piece of information in estimating $\overline H$. 

\subsubsection{A special ``geometric'' example}

We now give an example of a  Hamiltonian that is the characteristic function of a domain in $T^{*}T^{1}$. The Homogenized Hamiltonian is well-defined according to Theorem \ref{Main-properties} \ref{Main-properties-4} and Remark \ref{rem-4.3} (\ref{rem-4.3-3}). 
\begin{figure}[H]
\begin{center}  \begin{overpic}[width=6cm]%
 {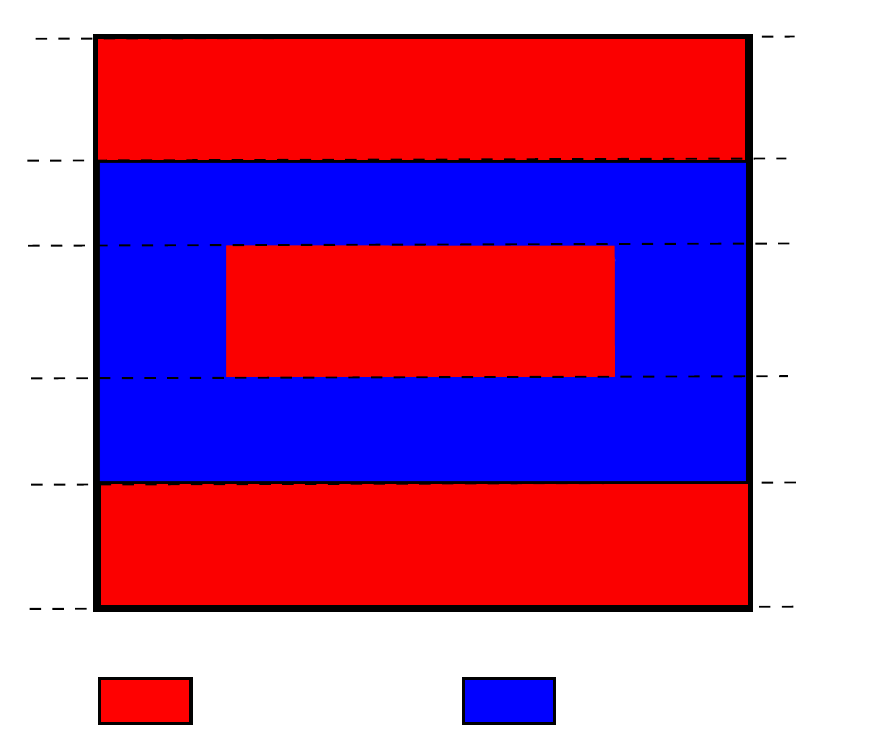}
 \put (-20,80){$p=2$}
 \put (-20,66){$p=1$}
   \put (-20,29){$p=-1$}
  \put (-20,15){$p=-2$}
  \put (25,3){$H=0$}
 \put (70,3){$H=1$}
\end{overpic}
\end{center}
\caption{The Hamiltonian $H(q,p)$.}\label{fig2}\end{figure}

\begin{figure}[H]\begin{center}  \begin{overpic}[width=6cm]%
 {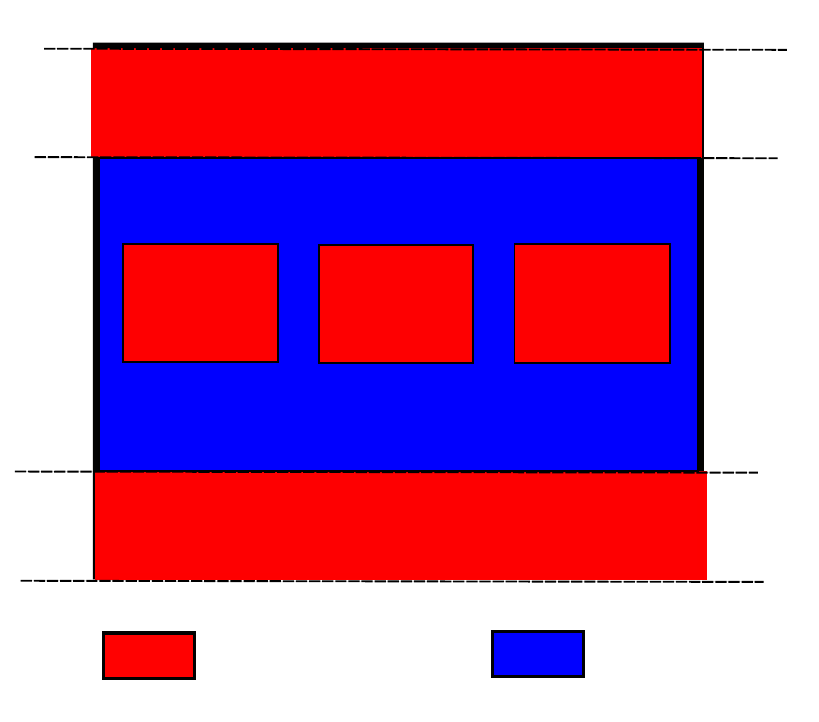}
 \put (-20,82){$p=2$}
 \put (-20,68){$p=1$}
   \put (-20,30){$p=-1$}
  \put (-20,15){$p=-2$}
  \put (30,5){$H=0$}
 \put (80,5){$H=1$}
\end{overpic}
\end{center}
\caption{The Hamiltonian $H(k\cdot q,p)$ (for $k=3$). }
\label{fig3}\end{figure}

We let $H$ be a Hamiltonian  on $T^{*}T^{1}$ represented on Figure \ref{fig2}, vanishing on the red set, assumed to be open, and equal to one in the blue set. We want to compute $\overline H (p)$  for $\vert p \vert \leq 2$. 
Denote by $A$ the area of the red island, and by $B$ the area of the blue sea, so $A+B=2$.  Note that for $A$ large compared to $B$ (in fact as soon as $A>B$) it may be difficult to construct a curve with  Liouville class $0$ contained in $H=1$ since we have to either go above the red island, thus adding an area of $A/2$ and we cannot substract more than $B/2$, or we go below and then add $-A/2$ to which we cannot add more than $B/2$. However replacing $H$ by $H_k$ the red island is replaced by $k$ smaller islands, and the difficulty vanishes as we have the choice to go above or below each island. More generally for $k$ large enough, we may find an embedded curve isotopic to the zero section, contained in the blue region of Figure \ref{fig3} (where $H(k\cdot q,p)=1)$, with  any given Liouville class  in $[-1,1]$. Thus, the curve yields an $L\in {\mathcal L}_{p}$ contained in $H_{k}=1$ for any $p$ in $[-1,1]$. As a result, since obviously $\overline H_k=\overline H$, we have  $\overline H (p) \geq 1$ for any $p$ in $[-1,1]$.
Since obviously, $\overline H(p)=0$ for $ \vert p \vert >1 $, we get
\begin{enumerate} 
\item $\overline H(p)=1$ on $[-1, -1]$ 
\item $\overline H(p)=0$ for $ \vert p \vert > 1$ \dispdot
\end{enumerate} 
\begin{remarks}
\begin{enumerate} 
\item Note that here $H$ is not continuous, so it is not surprising that $\overline H$ isn't either. 
\item The above example can be easily adapted to get homogenized Hamiltonians  taking more than two values.
\item With some more work, one can compute the homogenization of any autonomous Hamiltonian on $T^*S^1$.
\end{enumerate} 
\end{remarks}

  \subsection{Homogenized metric and the Thurston-Gromov norm}
First consider the case where $H$ generates the geodesic flow of $g$:
even though $H(q,p) = \langle g(q) p, p \rangle ^{1/2}=\vert   p\vert   ^{2}_{g}$ is not compactly supported,
it is coercive, so we may define its homogenization according to Proposition \ref{Prop-12.4}. Then $T^n$
$H_{k}$ generates the geodesic flow of the rescaled metric by the
covering map
\begin{eqnarray*}
T^{n} &  \longrightarrow  & T^{n}\\
q & \longrightarrow  kq
\end{eqnarray*}
of degree $k^{n}$.

It is well-known that if $D$ is the distance defined by $g$ on $ {\mathbb R}^n$ the universal cover of $T^n$
(i.e. $D(x,y)$ is the length of the shortest geodesic for $g$ 
connecting $x$ to $y$ in $ {\mathbb R}^n$) and $D_{k}$ the one defined by $g_{k}(q)=g(kq)$
(corresponding to $H_{k}$), again on $ {\mathbb R}^n$,  
we have
$$
D_{k} (x,y) = \frac{1}{k} D (kx , ky)
$$
and
$$
\lim_{k\to +\infty} D_{k} (x,y) = {\overline D} (x,y)
$$
where $\overline D$ is the distance associated to some flat Finsler metric
$\overline g$ (see \cite{Acerbi-Buttazzo}, Theorem III.1).
Since $g_k$ is invariant by the $ \mathbb Z^n$-action, $g_k$ and hence $D_k$ descend to a metric $d_k$ on $T^n$, and
similarly $\overline D$ descends to a Finsler metric $\overline d$. 

 It is also well known that $g_{k}$ does not converge to $\overline g$
in any reasonable sense, except for the convergence of minimizers of
the associated energy functional
$$
E(\gamma) = \int_{0}^{1} \vert  \dot\gamma(t)\vert ^{2}_{g} dt  \dispdot
$$
This phenomenon is related to the notion of $\Gamma$-convergence introduced by De Giorgi and his school in the 70's (cf.  \cite{de Giorgi, de Giorgi-Franzoni}, \cite{Dal
Maso}, \cite{Braides}). We shall denote by 
 $L(\gamma), L_k(\gamma), \overline L (\gamma)$ the length of a curve for the respective metric $d, d_k, \overline d$.  
 In particular we easily see that $\ell_{k}(\alpha)$, the length for
$d_{k}$ of the shortest closed geodesic in the homotopy class $\alpha$
(in $\pi_1(T^n)\simeq H_1(T^n)\simeq \mathbb{Z}^{n}$), $\ell_{k}(\alpha)$, converges to  ${\overline\ell} (\alpha)$, the length for
$\overline d$ of the shortest closed geodesic in the homotopy class $\alpha$, $\overline\ell(\alpha)$.
In other words, set $$\mathcal P_\alpha=\{ u\in C^\infty([0,1], T^n) \mid [u]\in \alpha \}$$ and notice that this is the image by the projection of
$$\widetilde {\mathcal P}_ \alpha =\{ u\in C^\infty([0,1],  {\mathbb R}^n) \mid u(t+1)=u(t)+ \alpha \}$$
then
$$
\ell_{k} (\alpha) = \inf_{x\in \mathcal P_\alpha} L_k(x)=\inf \{ D_{k}(x,x+\alpha)\mid x \in \mathbb{R}^{n}\}
$$
and since $D_{k}(x,x+\alpha)$ converges uniformly to $\overline D(x,x+\alpha )$, and
$x$ needs only to vary in a fundamental domain $[0,1]^{n}$ in $\mathbb{R}^{n}$, we get that
$$
\lim_{k\to+\infty} \ell_{k} (\alpha) = \overline\ell (\alpha)\dispdot
$$
But the class $\alpha$ contains at least a second
closed geodesic, obtained by a minimax procedure (see \cite{Birkhoff}, page 133). Indeed let $\beta \in H_1(T^n)$ be independent form $\alpha$ and $u: [0,1]^2 \longrightarrow {\mathbb R}^n$ be  a smooth map. Set 
\begin{displaymath} 
\widetilde{\mathcal P}_{\alpha, \beta}=\{ u\in C^\infty( [0,1]^2,  {\mathbb R}^n) \mid u(s, t+1)=u(s,t)+\alpha, u(s+1,t)=u(s,t)+\beta\}\dispdot
\end{displaymath} 
Then there is a closed geodesic of  length
$$
\ell_{k} (\alpha ,\beta) =\inf_{u\in \mathcal P_{\alpha, \beta}}  \sup_{s\in [0,1]}  L_{k}(u_s)
$$ where $u_s(t)=u(s,t)$,  
and similarly for $\overline\ell (\alpha ,\beta)$. It is thus reasonable to ask whether
$$
\lim_{k\to +\infty} \ell_{k} (\alpha,\beta) = \overline\ell (\alpha ,
\beta) ?
$$
The methods of our theorem imply a positive answer, since $$\ell_{k}(\alpha, \beta)^2= c(\beta, E)$$
that is $\ell_{k}(\alpha, \beta)$ is the homological minimax level associated to the $1$-dimensional  homology class of the free loop space that is the image of 
$S^{1}$ by $\theta \mapsto \beta(\theta)\cdot \alpha $ where ``$\cdot$'' denotes the addition law on the torus. We more generally can look at 
$\ell_k(\alpha , \beta)$ where $\beta \in H_k(T^n)$. 
Our results imply
\begin{proposition} 
We have
\begin{displaymath} 
\lim_{k\to \infty} \ell_k(\alpha, \beta) = \overline  \ell(\alpha, \beta) = \overline \ell(\alpha) \dispdot
 \end{displaymath}  
\end{proposition} 
\begin{proof} 
The first statement is just Proposition \ref{Prop-12.4} (or Proposition \ref{Prop-13.4}). The fact that $\overline  \ell(\alpha, \beta) = \overline \ell(\alpha)$ follows from the fact that the Finsler metric $\overline d$ is flat (i.e. invariant by translation on the torus), hence if $c(t)$ is a geodesic such that $c(t+1)=c(t)+ \alpha$ and $\beta \in H_1(T^n)=  \mathbb Z^n$ then $c_s(t)=c(t)+s\beta$ has  length $\overline L (c_s)$ , independent from $s$, so that taking $x(s,t)=c(t)+s\beta$ we see $\overline  \ell(\alpha, \beta) = \overline \ell(\alpha)$ holds for $\beta \in H_1(T^n)$. The general case follows by using the Pontryagin product (see Proof of Proposition \ref{Prop-7.2}) to prove by induction that 
$\overline  \ell(\alpha, \beta_1\cdot ...\cdot \beta_k) = \overline \ell(\alpha) $. Since any class in $H_k(T^n)$ is a linear combination of Pontryagin products, and 
because of the general fact that $c(a,f)=c(b,f)$ implies $c(a+b,f)=c(a,f)=c(b,f)$ we may conclude our proof. 

\end{proof} 
Note that the analogous statement cannot hold for the whole length
spectrum of $g_{k}$ (i.e. the set of lengths of closed geodesics), as
it is easy to construct examples for which the length spectrum of
$g_{k}$ becomes dense as $k$ goes to infinity i.e. for any $\lambda
\in \mathbb{R}_{+}$ and $\delta > 0$ there is $k_{0}$ in $\mathbb{N}$
such that for all $k\geq k_{0}$, Spec $(g_{k})\cap [\lambda -\delta ,
\lambda + \delta] \not= \emptyset$ (just add lots of small ``blisters'' to a flat metric), while the spectrum of $\overline g$ is discrete.

Remember also that the Thurston-Gromov norm associated to $g$ is defined as follows : for each homology class $c$ in $H_{1}(T^n, {\mathbb R} )\simeq {\mathbb R} ^n$, let us define 
$ \Vert c \Vert_{TG}$ as follows. For $c$ rational, that is $m\cdot c \in H_{1}(T^n, {\mathbb Z})$ for some integer $m$, $ \Vert c \Vert _{TG} = \frac{1}{m} \ell (m\cdot c )$.  The norm is then extended by density of the rationals. 

The proof of the following result is then left to the reader:

\begin{proposition} 
The Thurston-Gromov norm coincides with the symplectic homogenization of the metric. 
\end{proposition}

 \section{Further questions}

 Sergei Kuksin pointed out that the homogenization or averaging described here is a ``dequantized averaging'', in the sense that the traditional homogenization is concerned with the limit of the ``quantized Hamiltonian'', $H(\frac{x}{ \varepsilon }, D_{x})$ as $ \varepsilon $ goes to zero. By this we mean a Partial differential operator with principal symbol $H (\frac{x}{ \varepsilon }, p)$ operating on the set of smooth functions on $T^{n}$ or $ {\mathbb R}^{n}$.  Here, on the other hand,  we deal directly with  the ``classical Hamiltonian'' $H(\frac{x}{ \varepsilon }, p)$. It is natural to ask whether there is a connection between quantized and dequantized averaging, or to use a simpler language, between the homogenization of an operator, and the symplectic homogenization of its symbol. We may already consider the simple case of the previous section:  according to the classical theory of  $\Gamma$-convergence, the limiting operator of the Laplacians associated to the metric $g_{k}$ converges to some elliptic operator, denoted $\Delta_{\infty}$. But  $\Delta_{\infty}$ is not in general equal to the Laplacian of the  metric $g_{\infty}$. First of all $g_{\infty}$ is not Riemannian, but only Finslerian. Moreover, it seems that $g_{\infty}$ detects changes in the metric on small sets: typically a three-dimensional torus with a metric made small along three geodesic circles in three orthogonal directions will have a much smaller $g_{\infty}$ than one without such ``short directions``. But the Laplacian will not detect this, since the Brownian motion will not see such lines. So the only reasonable question is whether the metric $g_{\infty}$ determines the Laplacian $\Delta_{\infty}$.

One may ask a more general question, that is

\begin{question}
Assume $H_{\nu}$ converges to $H$ for the $\gamma$-topology. Does the spectrum- or some quantity defined using the spectrum- of the operators $H_{\nu}(x, D_{x})$ converge to the spectrum - or some quantity defined using the spectrum - of $H(x,D_{x})$?
\end{question}

\section{Appendix}
\appendix
\section{Capacity of completely integrable systems}\label{Appendix-A}

Our goal is to prove the following
 \begin{proposition}\label{Prop-A.1}Let $\phi^{1}$ be the time-one flow associated to the continuous, compactly supported Hamiltonian, $h(p)$, defined on ${\rm T}^*{\mathbb T}^n$. Then $$c_+(\varphi ^1)\; =\; \sup_{p}h(p) \quad ,\quad
c_-(\varphi^1)\; =\; \inf_{p}h(p)$$
$$\gamma (\varphi^1)=c_{+}(\varphi^{1})-c_{-}(\varphi^{1})=\operatornamewithlimits{osc}_{p}h $$
As a result, a continuous compactly supported and integrable Hamiltonian has generalized flow equal to $\Id$ if and only if it is identically zero.
 \end{proposition}
\begin{proof}
We shall only consider smooth Hamiltonians, the general case follows by density of compactly supported smooth Hamiltonians for the $\gamma$-topology in the set of continuous 
compactly supported ones (since density already holds for the $C^{0}$ metric). 
Set $\varphi^{t}(q,p)=(Q _t(q,p),P_t(q,p))$, then the graph of $\varphi _t$
defines a Lagrangian submanifold $\Gamma _t$ in
$T^*(T^n\times {\mathbb R}^n)$ as the image of $(q,P)\mapsto
\bigl (\frac{q+Q_t}{2}\, ,\,
\frac{p+P_t}{2}\, ,\, p-P_t\, ,\, Q _t-q \bigr
)$. Note that even though $Q _t$ is in $T^n$,
$Q _t-q $ has a unique lift to ${\mathbb R} ^n$
which is continuous in $t$ and equals 0 for $t=0$. The
same argument allows us to define $\frac{q +Q
_t}{2}=q +\frac{Q _t-q }{2}$.

Moreover, if we set $x=\frac{q +Q _t}{2}\;$,
$y=\frac{p+P_t}{2}\;$, $\xi =p-P_t\;$, $\eta  =Q
_t-q$, the symplectic form is given by $d\xi \wedge
dx+d\eta  \wedge dy$. In our situation, we have
$$x_t\; =\; q +\frac{t}{2}h'(p)\qquad y_t=p\qquad
\xi _t=0\qquad \eta _t= h'(p)$$
Thus if we set $f_t(x,y)=t\, h(y)$, we have
$$\xi _t\; =\; \frac{\partial }{\partial x}\,
f_t(x,y)\quad ,\quad \eta _t\; =\; \frac{\partial
}{\partial y}f_t(x,y)$$
that is $f_t$ is a generating function of $\Gamma _t$
with no ``fiber variables". As we mentioned in Remark \ref{rem-3.4} \ref{rem-trivial-c}
$$c_+(\varphi^t)\; =\; \sup_{} f_t\quad ,\quad
c_-(\varphi ^t)\; =\; \inf_{}f_t$$
$$\gamma (\varphi ^t)\; =\; \sup_{(x,y}f_t-\inf_{(x,y}f_t$$
Since $f_1(x,y)=h(y)$ this proves our proposition.

\end{proof}

The following applies the ideas of \cite{Humiliere} section 4.4 page 390 and \cite{Humiliere 2} section 3.3. For a continuous Hamiltonian, $H$, we define its ``generalized flow'' to be 
the image of $H$ by the extension of $H \mapsto \varphi_{H}^{1}$ as a map from $\widehat {\mathcal{H}_{c}}(T^{*}M)$ to
$\gclDHamc{T^{*}M}$.
 
\begin{corollary} \label{Cor-app-A}
The following assertions hold
\begin{enumerate} 
\item \label{corappA-1} If $h_{1}(p)$ and $h_{2}(p)$ are compactly supported, continuous, and have the same ``generalized'' time-one flow (in $\gclDHamc{T^*T^n}$) then $h_{1}=h_{2}$.
\item \label{corappA-2} If $h_{1}(p),h_{2}(p)$ are continuous and compact-supported,  and have generalized flows (in $\gclDHamc{T^*T^n}$) $\varphi_{1}^t, \varphi_{2}^t$, then $h_{1}(p)+h_{2}(p)$ has generalized flow  $\varphi_{1}^t\circ \varphi_{2}^t$. As a consequence $\varphi_{1}^{t}\circ \varphi_{2}^{t}=\varphi_{2}^{t}\circ \varphi_{1}^{t}$. 
In particular if $\varphi^t$ is the flow associated to $h(p)$, $\varphi^{-t}$ is the flow associated to $-h(p)$.
\end{enumerate} 
\end{corollary}
\begin{proof} 
Proof of  \ref{corappA-1}. 

Indeed, according to the above Proposition, we have $$\gamma (\varphi_{1}^t (\varphi_{2}^t)^{-1}) = t \operatornamewithlimits{osc}_{p} (h_{1}-h_{2})$$ since according to the above proposition this is true for smooth  $h_{1},h_{2}$, and we conclude by density of compactly supported smooth functions in the set of  compactly supported continuous functions. Therefore $\varphi_{1}^1=\varphi_{2}^1$ implies $h_{1}=h_{2}$. 

 Proof of \ref{corappA-2}. 

Let $h_{k,1}(p), h_{k,2}(p)$ be smooth sequences $C^0$-converging to $h_1(p), h_2(p)$ respectively. This implies that these sequences  $\gamma$-convergence. Now the corresponding time-one flows, $\varphi_{k,1}^1,\varphi_{k,2}^1$ commute so that $h_{k,1}(p)+h_{k,2}(p)=h_k(p)$ has flow
$\varphi_{k,1}^1\circ \varphi_{k,1}^1$, and since the $\gamma$-limit of $\varphi_{k,1}^1\circ \varphi_{k,1}^1$ is $ \varphi_1\circ \varphi_2$, we get that $h(p)=h_1(p)+h_2(p)$. 
The commutativity of the addition, implies the commutativity of the flows. Finally, the flow of the zero Hamiltonian being the identity, the last claim follows.

\end{proof}

\section{Some ``classical'' inequalities}\label{Appendix-B}

Our goal here is to prove the following results:

\begin{proposition} \label{B1}
Let $S_{1}(x,\xi), S_{2}(x,\eta)$ be two \GFQI     and $S_{1,x}(\xi)=S(x,\xi), S_{2,x}(\eta)=S_{2}(x,\eta)$.

Then $$ \vert c(1_{x}, S_{1,x})-  c(1_{x}, S_{2,x})\vert \leq \gamma (S_{1},S_{2}) \dispdot $$
\end{proposition} 
 \begin{proposition} \label{B2} For  a Hamiltonian isotopy $\varphi^{t }$ on $T^*T^n$, and $L$ Hamiltonian isotopic to the zero-section,
 we have $$\gamma(\varphi^{1} (L),L)\leq \gamma(\varphi^1).$$
\end{proposition} 

Note that the isotopy may be assumed to be compactly supported, since we may truncate the Hamiltonian outside the compact set $\bigcup_{t\in[0,1]}\varphi^{t}(L)$.

\begin{proposition} \label{Prop-B3} Let $S$ be a \GFQI,  $\alpha \in H^{k}(M)$ and $a \in H_k(M)$. Then 
 $$c(\alpha,S)=\inf \left \{ c(a,S) \mid a \in H_{k}(M),  \langle \alpha, a \rangle \neq 0\right\} .$$
  $$c(a,S)=\sup\left\{ c(\alpha,S) \mid  \alpha \in H^{k}(M) \langle \alpha, a \rangle \neq 0\right\} .$$
  In particular $c(1,S)=c([pt]_M, S)$ and $c(\mu_M, S)=c([M],S)$. 
  
As a result, we have for any nonzero $\alpha$ in $H^{*}(M)$ $$\vert c(\alpha , \varphi_{1})- c(\alpha, \varphi_{2}) \vert \leq \gamma (\varphi_{1},\varphi_{2}).$$
\end{proposition} 
\begin{proof}[Proof of proposition \ref{B1}]
It is important to notice that all formulas or inequalities we shall use in the proof are established in \cite{Viterbo-STAGGF}  for any \GFQI    , and not only those associated to an embedded Lagrangian submanifold.

Indeed, we have according to \cite{Viterbo-STAGGF}, prop. 3.3, p. 693 the formula $$c(u\cdot v, S_{1}\oplus S_{2}) \geq c(u,S_{1})+c(v,S_{2})$$ where $S_{1}\oplus S_{2}(x,\xi_{1},\xi_{2})=S_{1}(x,\xi_{1})+S_{2}(x,\xi_{2})$.  We then apply this inequality  to the generator $u=v=1_{x}$  of $H^0(\{ x \})$ (which is its own Poincar\'e dual) and we get

$$c(1_{x}, S_{1,x}\ominus S_{2,x}) \geq c(1_{x},S_{1,x})+ c(1_{x}, -S_{2,x})$$

But by \cite{Viterbo-STAGGF}, prop. 2.7, p. 692,  we have

$$c(1_{x}, -S_{2,x})=-c(1_{x},S_{2,x})$$

thus

$$c(1_{x}, S_{1,x}\ominus S_{2,x}) \geq  c(1_{x},S_{1,x})- c(1_{x}, S_{2,x})$$

Similarly we have 
$$c(1_{x}, S_{1,x}\ominus S_{2,x}) \leq c(1_{x},S_{1,x})+ c(1_{x}, -S_{2,x})=c(1_{x},S_{1,x})- c(1_{x}, S_{2,x})$$

Now since $(S_1\ominus S_2)_x(\xi_1,\xi_2)=S_{1,x}(\xi_1)-S_{2,x}(\xi_2)$ and $\gamma (L_1,L_2)=\gamma (S_1\ominus S_2)$,  we have by the reduction inequality (\cite{Viterbo-STAGGF}, prop. 5.1 p. 705)

$$c(1_{x}, S_{1,x}\ominus S_{2,x}) \leq \gamma (S_{1}\ominus S_{2})= \gamma (L_{1},L_{2})$$

and our claim follows.

\end{proof}

\begin{proof}[Proof of proposition \ref{B2}]

Indeed, $\varphi(0_{T^n})$ is the symplectic  reduction of $$\widetilde\Gamma (\varphi) = \{(q,p,Q,P) \mid (Q,P)=\varphi (q,p)\}$$ by the coisotropic submanifold, $N_p=\{p=0\}$. 
The symplectic  map $$ (q,p,Q,P) \longrightarrow (Q,p,p-P, Q-q)=(u,v,U,V)$$ sends $N_p$ to $N_v=\{v=0\}$. The reduction inequality in \cite{Viterbo-STAGGF} (prop. 5.1, page 705) and the fact that  $\varphi(L)= \widetilde \Gamma(\varphi)) \cap N_v / N_v^\omega$, then implies
$$\gamma (\varphi (L)) \leq \gamma (\widetilde \Gamma (\varphi)) = \gamma (\varphi)$$ 

Now if $L=\rho (0_{T^n})$, we have

$$\gamma (\varphi (L), L) = \gamma (\varphi(\rho(0_{T^n})), \rho(0_{T^n})) = \gamma (\rho^{-1}\varphi \rho(0_{T^n}) ) \leq \gamma (\rho^{-1}\varphi \rho) = \gamma (\varphi)$$
where the second and last equality follow by symplectic invariance of $\gamma$, while the inequality has been proved above. 
\end{proof} 
 \begin{proof}[Proof of Proposition \ref{Prop-B3}] Let us denote by $T$ the Thom isomorphism, and by $T^{*}$ its homological counterpart. 
 Let $a$ in $H_{k}(M)$ and $T^{*}a$ it image in $H_{k+d}(E^{+\infty},E^{-\infty})$. Similarly, let $\alpha \in H^{k}(M)$ and $T\alpha$ its image in $H^{k+d}(E^{+\infty},E^{-\infty})$. 
 Now, considering the $(k+d)$-cohomology group as the dual of the $(k+d)$-homology group (we work with coefficients in a field), we have the following diagram
$$ \xymatrix { H_{k+d}(E^{\lambda},E^{-\infty}) \ar^{i_{\lambda}^{*}(Tu)}@{->}[dr] \ar^{i_{\lambda}}[r]& H_{k+d}(E^{+\infty},E^{-\infty})\ar^{Tu}[d]\\   & {\mathbb R} }.
 $$
 
and we have $i_{\lambda}^{*}(T\alpha)\neq 0$ if and only if there is $T^{*}a \in \Image (i_{\lambda})$ such that $\langle T\alpha, T^{*}a \rangle \neq 0$.  
 
 Since $\langle T\alpha, T^{*}a \rangle = \langle \alpha, a\rangle$ we have the inequality $\lambda \geq c(a,L)$ if and only if there exists $\alpha$ such that 
  $\langle \alpha, a\rangle \neq 0$, and  $\lambda \geq c(\alpha,L)$ : this follows immediately from the universal coefficient theorem ($H^k(M)$ is the dual of $H_k(M)$). In other words we proved  that 
  
\begin{equation*}\thetag{*}\hskip 2cm c(a,L)=\sup\{ c(\alpha , L) \mid \langle \alpha, a \rangle \neq 0\}\end{equation*}
This proves the first statement.  Similarly we have 
 the inequality $\lambda \geq c(\alpha,S)$ if and only if there exists $a$ such that 
  $\langle \alpha, a\rangle \neq 0$, and  $\lambda \geq c(a,S)$, so 
 $$c(\alpha,S)=\inf\{ c(a , S) \mid \langle \alpha, a \rangle \neq 0\}$$ 
Now by \cite{Viterbo-STAGGF} (prop. 2.7, page 692),  if $b \in H_{n-k}(M), \alpha \in H^{k}(M)$ are Poincar\'e dual classes\footnote{This means for all $\beta \in H^{n-k}(M)$ we have $\langle \beta , b \rangle = \langle \alpha\cup \beta , [M]\rangle$}, we have the identity $c(\alpha,S)=-c(a,\overline S)$. Moreover from the same paper (prop. 3.3 page 693) we see that $c(\alpha\cdot \beta,S_{1}\ominus S_{2})\geq c(\alpha,S_{1})+c(\beta,\overline S_{2})$.

So for  $b \in H_{n-k}(M)$ and $\alpha \in H^k(M)$ be Poincar\'e dual classes
we may write, using \thetag{*}
\begin{gather*}  c(\alpha,S_{1})-c(\alpha,S_{2})= c(\alpha,S_{1})+ c(b ,\overline S_{2}) =\\ c(\alpha,S_{1})+\sup\{c(\beta,\overline S_{2}) \mid \langle \beta , b \rangle\neq 0\}=\sup\{c(\alpha,S_{1})+c(\beta,\overline S_{2}) \mid \alpha\cup \beta \neq 0\}\leq \\ \sup \{ c(\alpha\cdot \beta , S_{1}\ominus S_{2}) \mid \alpha\cdot \beta \neq 0\} \end{gather*}
Since $\alpha\in H^{k}(M), \beta \in H^{n-k}(M)$, $\alpha\cdot \beta\neq 0$ implies that $\alpha\cdot  \beta $ is a multiple of the orientation class $\mu$, hence, the last term in the above equals to $c(\mu, S_{1}\ominus S_{2})$. 
We finally proved

$$c(\alpha,S_{1})-c(\alpha,S_{2})   \leq c(\alpha\cdot \beta, S_{1}\ominus S_{2})=c(\mu, S_{1}\ominus S_{2})$$
Using the fact that $c(\mu, S_{2}\ominus S_{1})=-c(1,S_{1}\ominus S_{2})$, we get exchanging $S_1$ and $S_2$
$$c(\alpha,S_{2})-c(\alpha,S_{1})   \leq c(\mu, S_{2}\ominus S_{1})= - c(1, S_1\ominus S_2)
$$
that is $$c(\alpha,S_{1})-c(\alpha,S_{2})   \geq c(1, S_{1}\ominus S_{2})$$ so that finally 
$$   c(1,S_{1}\ominus S_{2}) \leq c(\alpha,S_{1})-c(\alpha,S_{2})  \leq c(\mu, S_{1}\ominus S_{2})$$
Since in our case $S_1, S_2$ are \GFQI for $\overline\Gamma (\varphi_1), \overline \Gamma(\varphi_2)$, we have 
$c(1, S_1\ominus S_2) \leq 0 \leq c(\mu, S_{1}\ominus S_{2})$, so the above inequality implies 
\begin{gather*}  \vert c(\alpha,\varphi_{1})-c(\alpha,\varphi_{2})\vert   =\vert c(\alpha,S_{1})-c(\alpha,S_{2})   \vert \leq \\ c(\mu, S_{1}\ominus S_{2})-c(1,S_{1}\ominus S_{2})=\gamma(S_{1},S_{2})=\gamma (\varphi_1,\varphi_2)
\end{gather*} 

 \end{proof} 
\section{A different type of homogenization}
 As has been pointed out to me by M. Bardi and F. Cardin, the Hamilton-Jacobi homogenization is often applied to the sequence $H(x, \varepsilon ^{-1}x,p)$: we seek the limit as $ \varepsilon $ goes to zero of
 
 \begin{equation}
\tag{$HHJ_{ \varepsilon }$}
\left\lbrace
\begin{array}{ll}
& \frac{\partial}{\partial t} u (t,x) + H (x, \varepsilon ^{-1}x, \frac{\partial}{\partial
x}u (t,x)) = 0\\
& u (0,x) = f(x)\dispdot
\end{array}
\right.
\end{equation}

This  seems to be more general than the case we deal with. However we prove here that this problem can be reduced to the case we studied.

Let indeed $K(x,y,p_{x},p_{y})$ be the Hamiltonian on $T^*(T^n\times T^n)$ defined by $K(x, y,p_{x},p_{y})=H(x,y,p_{x}+p_{y})$.

\begin{remark} 
Note that even if $H$ is compactly supported, $K$ is not,  since the map $(p_x,p_y) \mapsto p_x+p_y$ is not proper.  However $ K$ is $C^2$ bounded, and this is enough to carry on symplectic homogenization as in Sections 6,7 and 8.  
\end{remark} 

We claim that the equation
\begin{equation*} \tag{1}  \frac{\partial}{\partial t} u (t,x) + H (x, x, \frac{\partial}{\partial
x}u (t,x)) = 0 \end{equation*}

is satisfied by $u(t,x)=v(t,x,x)$ where $v$ is the variational solution of

\begin{equation*}\tag{2}
\frac{\partial}{\partial t} v(t,x,y) + K (x, y, \frac{\partial}{\partial
x}v (t,x,y), \frac{\partial}{\partial
y}v (t,x,y)) = 0 \dispdot
\end{equation*}

Indeed this can be rewritten for $x=y$ as

$$
 \frac{\partial}{\partial t} v (t,x,x) + H (x, x, \frac{\partial}{\partial
x}v (t,x,x)+\frac{\partial}{\partial
y}v (t,x,x)) = 0
$$

and since $$  \frac{\partial}{\partial
x}u (t,x)= \frac{\partial}{\partial
x}v (t,x,x)  + \frac{\partial}{\partial
y}v (t,x,x) $$

 this proves our claim.

 Now we may replace $K$ by $$K_\varepsilon (x,y,p_x,p_y)=H(x,\frac{y}{\varepsilon},p_x+p_y)=H_\varepsilon (x,y,p_x+p_y)$$
 and we shall get solutions of the equation
 $$  \frac{\partial}{\partial t} u_\varepsilon (t,x) + H (x, \varepsilon^{-1}x, \frac{\partial}{\partial
x}u_\varepsilon (t,x)) = 0.$$

 However we have to prove that

 \begin{proposition}
 If $v$ is a variational solution of \thetag{2}, then $u(t,x)=v(t,x,x)$ is variational solution of
 \thetag{1}.
 \end{proposition}
 \begin{proof}

 Indeed it is enough to prove that if $L$ is a Lagrangian submanifold containing the isotropic submanifold
 $$I_f=\{(t,-K(x,y,\frac{\partial f}{\partial x},\frac{\partial f}{\partial y}), \frac{\partial f}{\partial x},\frac{\partial f}{\partial y})\}$$
 and such that $L$ is contained in

 $$\{(t,\tau , x,y, p_x,p_y) \mid \tau +K(t,x,y,p_x,p_y)=0\}$$
 then its symplectic reduction by $x=y$ is a Lagrangian submanifold, $L_{\Delta}$, contained in
  $$\{(t,\tau , x, p_x) \mid \tau +H(t,x,x,p_x)=0\}$$

 and $L_{\Delta}$ contains
 
 $$I_g=\left \{(t,-H(x,x,\frac{\partial g}{\partial x}(x)), \frac{\partial g}{\partial x}(x)) \right \}$$
 where $g(x)=f(x,x)$.

 This is rather straightforward to check, and is a consequence of the commutation relation $\{x-y,p_x + p_y\}=0$.
 
 Finally we notice that if $S(t,x,y,\xi)$ is a \GFQI     for $L$, then $v(t,x,y)=c(1_{t,x,y},S)$, while the reduction $L_{\Delta}$ of $L$ by $x=y$ has 
 for \GFQI     the function $S_{\Delta}=S(x,x,\xi)$, so that  $u(t,x)=c(1_{t,x},S_{\Delta})=c(1_{t,x,x},S)=v(t,x,x)$. This concludes our proof. 
  \end{proof}
 \begin{corollary}
 Let $H(x,y,p)$ be a compactly supported function on $T^n\times T^n \times ({\mathbb R})^n$. Set $K(x,y,p_x,p_y)=H(x,y,p_x+p_y)$, and let
 $\overline K (x,p_x,p_y)$ be the homogenization of $K$ with respect to the $y$ variable (i.e. $\overline K (x,p_x,p_y)$ is the
 $\gamma$-limit of $K_{\varepsilon}(x,y,p_x,p_y)=K(x,\frac{y}{\varepsilon},p_x,p_y)$). Then  $\overline H(x,p)=\overline K(x,p,0)$
 is such that a sequence  $(v_{ \varepsilon })_{ \varepsilon >0}$ of variational solutions of
 $$
 \frac{\partial}{\partial t} v (t,x) + H (x, {\varepsilon}^{-1} x, \frac{\partial}{\partial
x}v (t,x)) = 0
$$

$C^{0}$-converges to a variational solution of

$$
 \frac{\partial}{\partial t} v (t,x) + {\overline H} (x, \frac{\partial}{\partial
x}v (t,x)) = 0
$$
 \end{corollary}

\begin{proof}
The only missing fact is to show that $\overline K(x,p_x,p_y)$ can be written as $\overline H(x,p_x+p_y)$ (if this is the case, we recover
$\overline H$ by $\overline H(x,p)=\overline K(x,p,0)$). We claim that this amounts to proving that the $x_{j}$ being coordinates on the first torus, and $y_{j}$ the same coordinate on the second torus, the function $x_{j}-y_{j}$ commutes\footnote{Of course $x_{j}-y_{j}$ is only defined in $S^{1}$, but this is equivalent to claiming that for any function $f$ on $S^{1}$, $f(x_{j}-y_{j})$ commutes with $\overline K$, provided it commutes
with $K_{\varepsilon}$.} with $\overline K$, knowing it commutes
with $K_{\varepsilon}$. The result follows from the fact that this commutation relation goes to the $\gamma$-limit. If $\varphi_{j}^{t}$ is the flow of $x_{j}-y_{k}$, and 
$\psi_{ \varepsilon }^{t}$ the flow of $K_{ \varepsilon }$, the relation $\psi_{ \varepsilon }^{t}\circ \varphi_{j}^{t}\circ \psi_{ \varepsilon }^{-t}=\varphi_{j}^{t}$ 
implies for $\overline\psi^{t}=\gamma-\lim_{ \varepsilon \to 0}\psi^{t}$ that $\psi^{t}\circ \varphi_{j}^{t}\circ \psi^{-t}=\varphi_{j}^{t}$.

\end{proof}

\begin{lemma}
Let $H$ be a compactly supported smooth Hamiltonian and $(K_j)_{j\geq 1}$ a sequence of Hamiltonians commuting with $H$ and  $\gamma$-converging to a continuous Hamiltonian $K_\infty$.
Then $H$ and $K_\infty$ $\gamma$-commute in the following sense
$$K_\infty(\varphi_H^t)=K_\infty$$  where $\phi_H^t$ is the flow
of $H$. In particular if  the $K_j$ commute with the functions $x_j-y_j$ then so does $K_\infty$, and $K_\infty (x,p_x,p_y)$ only depends on
$(x,p_x+p_y)$.
\end{lemma}

\begin{proof}
Indeed, the flow of $K_j$ being $\psi_j^t$ we have that $ \varphi^t_H \psi_j^s \varphi^{-t}_H$ $\gamma$-converges to
$ \varphi^t_H \psi_\infty^s \varphi^{-t}_H$, hence our first result.
Now the flow of $x_j-y_j$ is given by $$(x,y,p_x,p_y) \longrightarrow  (x,y,p_x+te_j,p_y+te_j)$$

Now a continuous Hamiltonian commuting with the functions $x_j-y_j$ only depends on $p_x+p_y$. Indeed it is easy to show that for all $t$,
$K_j(x,y,p_x+te_j,p_y-te_j)=K_j(x,y,p_x,p_y)$ hence  $\gamma$-converges to  both $\overline K(x,p_x+te_j,p_y-te_j)$ and $\overline K(x,p_x,p_y)$. Thus by uniqueness of the limit,  for all $t$ and $j$, 
$\overline K(x,p_x,p_y)=\overline K(x,p_x+te_j, p_y-te_j)$. This is equivalent to the property that $\overline K$ is a function of $(x,p_x+p_y)$.
\end{proof}

\section{Generating function for Euler-Lagrange flows} \label{Appendix-Lag}

Let $N$ be a compact manifold. The goal of this appendix is to prove that if the Hamiltonian $H(t,q,p)$ is strictly convex in $p$ and its flow is complete, then $\varphi^{t}(0_{N})$ has a \GFQI     equal to a  positive definite quadratic form at infinity. 
The same property holds for the graph of $\varphi^{t}$. We remind the reader that using the Legendre transform, the flow of $H$ can be identified with the Euler-Lagrange flow associated to the 
Legendre dual of $H$, $L(t,q,\xi)$, where $$L(t,q,\xi)= \sup \{ \langle p,\xi \rangle - H(t,q,p) \mid p \in T_{q}^{*}N\}.$$ 
The definition of Tonelli Lagrangians can be found in Definition \ref{Tonelli-Lagrangian}.

\begin{proposition} \label{Prop-Appendix-Lag}
Let $L(t,q,\xi)$ be a  Tonelli Lagrangian defined on ${\mathbb R} \times TN$, $1$-periodic in $t$. Let $\varphi^{t}$ be the flow of the corresponding Hamiltonian $H(t,q,p)$ defined on $ {\mathbb R} \times T^{*}N$. Then $\Lambda=\varphi^{1}(0_{N})$ has a generating function equal to a positive definite quadratic form outside a compact set.  As a consequence, if we denote by $S(q,\xi)$ this function, we have $c(1_{q},\Lambda)=\inf_{\xi \in {\mathbb R} ^{q}}S(q,\xi)$ and $c(1,\Lambda)=\inf_{(q,\xi)\in T^{n}\times {\mathbb R}^{q}} S(q,\xi)$. 
 \end{proposition} 
 \begin{proof} Using Brunella's idea of embedding $T^{*}N$ into $T^{*} T^{n}$ (cf. \cite{Brunella}) we only need to prove this for a Lagrangian defined on the torus (this is the only case we use here anyway). 
 
 Thus consider a general Lagrangian $L(t,q,\xi)$ and the corresponding Hamiltonian $H(t,q,p)$. If we look for intersection points $\varphi_{}^{1}(0_{N})\cap 0_{N}$, we can modify $ H$ outside a compact set, so that $\widehat H(t,q,p)= C\vert p \vert ^{2}$ outside a compact set,  $\widehat H$ is still strictly convex and for $t$ in $[0,1]$,  $\varphi^{t}(0_{N})$ is unchanged. There is thus a corresponding  ``truncated'' Lagrangian $\widehat L$. We still denote by $L$ the truncated Lagrangian in the sequel. Note that this is automatically a Tonelli Lagrangian. 
 
 Consider the set $W=\{(q_{1},q_{2}) \in {\mathbb R}^{n}\times {\mathbb R}^{n}\}/\simeq $ where $(q_{1},q_{2})\simeq (q_{1}+\nu, q_{2}+\nu)$ for $ \nu \in {\mathbb Z}^{n}$ 
 and we have the projection $W \to T^{n}$ on the first component, with fiber $ {\mathbb R} ^{n}$. We sometimes denote by $(q,x)$ an element in $W$ with projection $q$,
 and  $x=q_{2}-q_{1} \in {\mathbb R}^{n}$ which is well defined in $ {\mathbb R} ^{n}$. 
 
Consider for $(q_{1},q_{2}) \in W$  $$A(q_{1},q_{2})=\inf \left \{\int_{0}^{1}L(s,q(s),\dot q(s)) dt \mid q:[0,1] \to  {\mathbb R}^{n},\; q(0)=q_{1}, q(T)=q_{2}\right\}$$

The existence of a critical point of the  energy $E(q)=\int_{0}^{1}L(s,q(s),\dot q(s)) dt $ on the set  $\mathcal P(q_{1},q_{2})=\{ q:[0,1] \to  {\mathbb R}^{n} \mid q(0)=q_{1}, q(1)=q_{2}\}$  of paths connecting $q_{1}, q_{2}$ is equivalent, by the standard methods of the  calculus of variations, to the existence of a point in $\varphi^{1}(\{q_{1}\}\times {\mathbb R} ^{n}) \cap \{q_{2}\}\times {\mathbb R} ^{n}$. For $L(t,q,p)=L^{0}(t,q,p)=\frac{C}{2} \vert p \vert ^{2}$, the corresponding maps is given by  $\varphi_{0}^{1}(q,p)=(q+Cp,p)$, hence $\varphi_{0}^{1}(\{q_{1}\}\times {\mathbb R} ^{n})=\{ (q_{1}+Cp,p) \mid p \in {\mathbb R}^{n}\}$ and this intersects transversally  at a unique point all other $\{q_{2}\}\times {\mathbb R} ^{n}$. This will still 
hold for a $C^{1}$-small compactly supported perturbation of $\varphi_{0}^{1}$, hence for a general $L$ as above, provided $t$ is small enough, the
minimizer realizing $A(q_{1},q_{2})$ is unique. 

\begin{lemma}
Let $L$ be a Lagrangian such that $L(t,q,\xi)=C \vert \xi \vert^{2}$ for $\xi$ large enough. Then for $t$ small enough, we have that for any $q_{1},q_{2}$, the function
$E_{t}(\gamma)=\int_{0}^{t}L(s,q(s),\dot q(s)) dt$ has a unique critical point  $\gamma : [0,t] \to {\mathbb R} ^{n}$, depending continuously on $(q_{1},q_{2})$. Moreover this unique critical point is a minimum and for $ \vert q_{2}-q_{1} \vert $ large enough, the critical value $A_{t}(q_{1},q_{2})= \frac{C}{2t} \vert q_{2}-q_{1}\vert ^{2} $ is quadratic positive definite in $q_{2}-q_{1}$. 
\end{lemma}  The existence of a minimum of $E_{t}$ on ${\mathcal P}(q_{1},q_{2})$ implies  that if  uniqueness and transversality hold,  the only critical point of $E$ is a  minimum, and  $A_{t}$ is a smooth function of $(q_{1},q_{2})$. 
\begin{proof} For simplicity we may assume $C=1$. 
Let $\varphi^{t}$ be the flow associated to $H$, the Legendre dual of $L$, and $\pi: T^{*} {\mathbb R} ^{n} \to {\mathbb R}^{n}$ be the projection. We claim that for $t$ small enough, $\zeta^{t}=\pi\circ \varphi^{t} : \{q\}\times {\mathbb R} ^{n} \to {\mathbb R}^{n} $ is a diffeomorphism. Indeed, we claim that 
\begin{enumerate} 
\item \label{A}For $ \vert p \vert $ large enough (independently of $t$), $\varphi^{t}(q,p)=q+tp$
\item \label{B} For $t$ small enough $D\zeta^{t}(q,p)$ is invertible.
\end{enumerate} 

The first claim is clear: if $H(t,q,p)= \vert p \vert ^{2}$ for $ \vert p \vert \geq r$, then for $\vert p \vert \geq r$, we have $\varphi^{t}(q,p)=q+tp$. 
The second claim  only needs to be checked for $ \vert p \vert \leq r$. 
Let us compute \begin{gather*}  \frac{d}{dt} (D\varphi^{t}(q,p))_{\mid t=t_{0}}=D( \frac{d}{dt} \varphi^{t}(q,p))_{\mid t=t_{0}}=D (X_{H}(\varphi^{t_{0}}(q,p)))=\\  DX_{H}(\varphi^{t_{0}}(q,p))D\varphi^{t_{0}}(q,p) \end{gather*} 
so if $M(t)=D\varphi^{t}(q,p)$ we have $$ \frac{d}{dt}M(t)= DX_{H}(\varphi^{t_{0}}(q,p))M(t)$$ and $$D(\pi \circ \varphi^{t})(q,p)_{\mid \{0\}\times {\mathbb R} ^{n}}=D\pi (\varphi^{t}(q,p)) D\varphi^{t}(q,p)_{\mid \{0\}\times {\mathbb R} ^{n}}= D\pi \circ M(t)_{\mid \{0\}\times {\mathbb R} ^{n}}$$
Now $$M(t)=M(0)+t DX_{H}(q,p)+ o(t)$$ and $X_{H}(q,p)= \begin{pmatrix} \frac{\partial H}{\partial p}, &  -\frac{\partial H}{\partial q} \end{pmatrix}$ so that 
$$DX_{H}(q,p)= \begin{pmatrix} \frac{\partial^{2} H}{\partial p\partial q} &-\frac{\partial^{2} H}{\partial q^{2}}  \\ \frac{\partial^{2} H}{\partial p^{2}}  &-  \frac{\partial^{2} H}{\partial p\partial q} \end{pmatrix}$$ 
and $$D\pi \circ DX_{H}(q,p)_{\mid \{q\}\times {\mathbb R} ^{n}}=\frac{\partial^{2} H}{\partial p^{2}} $$
As a result, denoting $S(t)= D\pi \circ M(t)_{\mid \{0\}\times {\mathbb R} ^{n}}$ we get $$S(t)=S(0)+ t \frac{\partial^{2} H}{\partial p^{2}} (q,p)+ o(t)$$ and since $S(0)=0$ and  the reminder $o(t)$ is uniform on any  set bounded in $p$ (since it will be bounded, by Taylor's formula  by higher derivatives of $H$ which are periodic in $q$), we get that for $t$ small enough, $S(t)$ will be invertible for $ \vert p \vert \leq r$. 
This concludes the proof of \ref{B}. 

We thus know that the map $\pi\circ \varphi^{t}$ is equal to $q+ t \cdot \Id$ outside a compact set, and has invertible differential. It is a classical calculus exercise to conclude that such a  map is a diffeomorphism.  This means that given $q_{1}$ for every $q_{2}$ there is a unique critical point of $\gamma \mapsto  E_{t}(\gamma) $, which is necessarily a minimum, since we know by the Tonelli condition implies  that there is a minimum. Moreover the map $q_{2}\mapsto P(q_{2})$ where $P(q_{1},q_{2})$ is the unique point  such that $\pi\varphi^{t} (q_{1},P(q_{1},q_{2}))=q_{2}$. Finally for $q_{2}-q_{1}$ large enough, the path $q(t)=q_{1}+ \frac{s}{t}(q_{2}-q_{1})$ is a critical point of $E$ on ${\mathcal P}(q_{1},q_{2})$.

 \end{proof} 

Now let $N$  be chosen large enough so that $\varphi^{\frac{k+1}{N} }\varphi^{- \frac{k}{N} }$ satisfy the conclusions of the above lemma (i.e. $t=1/N$ is small enough in the sense of the previous lemma). Let $A_{k}(q,q')$ be the minimum of the energy corresponding to $L(t,q,\xi)$ for $t \in [ \frac{k}{N}, \frac{k+1}{N}]$.  Then, consider the 
function
focused
$$S(q_{1};q_{2},q_{3},...,q_{N})=A_{0}(q_{1},q_{2})+A_{1}(q_{2},q_{3})+....+A_{N-1}(q_{N-1},q_{N}).$$
One easily checks that $S$  is a generating function for $\varphi^{1}(0_{T^{n}})$. Moreover it is asymptotically quadratic in the sense that it satisfies the assumptions of proposition 1.6 page 441 of \cite{Viterbo-Montreal}, and this proves that $S$ can be deformed into a \GFQI     equal to a positive definite quadratic form generating the same Lagrangian (see loc. cit. prop 1.6 and the proof of theorem 1.7). 
\end{proof} 

\section{Relationship with  \texorpdfstring{\cite{M-V-Z}}{MVZ}}
The paper [M-V-Z] contains a slightly different approach to symplectic homogenization : it is focused on the behavior of the Lagrangians $L_{p_0}$ by iterations of $\varphi_H$. They also make use of Floer homology, but this does not really matter, as we could use Generating Function homology (see [Traynor], and [V2] for the equivalence of the two).  We here briefly relate their definition with ours. 
Their version of homogenization (see the proof of theorem 1.3) is defined by  $$\mu_p(\varphi)=\lim_{k\to \infty} \frac{1}{k} c_+(\tau_{-p}\varphi^k (\tau_{p}(0_N))$$
We claim
\begin{proposition}\label{Prop-Appendix-E}
 Let $\varphi$ be the time-one flow of a compactly supported smooth Hamiltonian $H$. Let $\overline H$ be defined in our Main Theorem.  Then 
we have $\mu_p(\varphi)=\overline H(p)$.
\end{proposition} 
\begin{lemma} 
Let $H(p)$ be an integrable Hamiltonian. Then we have setting $L_{p_0}=\{(x,p_0) \mid x \in T^n\}$ that $FH^*(\varphi_H(L_{p_0}),L_{p_0};a,b)= H^*(T^n)$  if $H(p_0)\in ]a,b]$ and vanishes otherwise. 
 \end{lemma} 
 \begin{proof} 
 Indeed, the flow is given by $(x,p) \mapsto (x+t\nabla H(p), p)$ and we have that $FH^*(\varphi_H(L_{p_0}),L_{p_0};a,b)=FH^*(L_{p_0},L_{p_0},H;a,b)$ where the second homology is obtained by taking trajectories of $X_H$ and  the action of a trajectory is $\int_0^1 pdx - H dt =p_0 \nabla H(p_0) -H(p_0)$. Note that for $p_0=0$ we get $-H(0)$. For a general $p_0$, replacing $H$ by $K_{p_0}(p)=H(p+p_0)$ we get that $FH^*(L_{0},L_{0},K_{p_0};a,b)$ is non zero if and only if $-K(0)=-H(p_0)\in ]a,b]$. Note that if $\tau_{p_0}$ is the vertical translation by $p_0$, we have $\varphi_{K_{p_0}} =T_{-p_0}\varphi_H T_{p_0}$. 
  \end{proof} 
  \begin{proof} [Proof of Proposition \ref{Prop-Appendix-E}]
  Now let $H$ be a smooth compact supported Hamiltonian. Setting $\varphi_k= \rho_k^{-1}\varphi_H^k\rho_k$, we have according to our Main Theorem that  $\varphi_k$ $\gamma$-converges to $\overline \varphi$  and since $\tau_{p_0}$ and $\rho_k$ commute, 
 the sequence $ \frac{1}{k} c_+(\tau_{-p_0}{\varphi_k} \tau_{p_0}(0_N), 0_N)$ converges to $c_+(\overline \varphi (L_{p_0}))=\overline H(p_0)$. This uses the fact that the map 
$$ {\DHam}(T^*N) \times  {\mathcal L}(T^*N) \longrightarrow   {\mathcal L}(T^*N) $$
given by $$(\varphi, L) \mapsto \varphi(L)$$ extends by continuity to 
 $$ \gclDHam{T^*N} \times \widehat {\mathcal L}(T^*N) \longrightarrow  \widehat {\mathcal L}(T^*N) $$
(see \cite{Humiliere}, prop. 4.3). 
\end{proof}

\renewcommand \vol {\rm vol. }


\begin{thebibliography}{AAAAAA}
\addcontentsline{toc}{section}{References}

\bibitem[Aar]{Aarnes}
J.F. Aarnes,
\newblock{\em Quasi-states and quasi-measures.}
\newblock { Adv. Math., } \vol 86 (1991), pp. 41-67.

\bibitem[Ac-Butt]{Acerbi-Buttazzo}
E. Acerbi and G. Buttazzo,
 \newblock{\em On the limits of periodic Riemannian metrics.}
\newblock { Journal d'Analyse Math\'ematique.}
\vol 43(1983), pp. 183-201, DOI: 10.1007/BF02790183

\bibitem[Alv-Bar1]{Alv-Bar1}
 O. Alvarez,  and M. Bardi,
 \newblock{\em Viscosity solutions methods for singular perturbations in deterministic  and stochastic control.}
 \newblock  {SIAM J. Control Optim.,}  \vol 40 (2001/02), pp. 1159--1188

\bibitem[Alv-Bar2]{Alv-Bar2}
 O. Alvarez,  and M. Bardi,
 \newblock{\em Singular perturbations of nonlinear degenerate parabolic PDEs: a  general convergence result. }
 \newblock  {Arch. Ration. Mech. Anal.,}  \vol 170  (2003), pp. 17--61

\bibitem[Ban]{Bangert}
V. Bangert,
\newblock{\em Closed Geodesics on Complete Surfaces.}
\newblock { Math. Ann., } \vol 251 (1980), pp 83--96.

\bibitem[Ba-CD]{Bardi-Capuzzo-Dolcetta}
M. Bardi and I. Capuzzo-Dolcetta,
\newblock{\em Optimal Control and Viscosity Solutions of Hamilton-Jacobi-Bellman Equations.}
\newblock{Modern Birkh\"auser classics, Birkh\"auser, 1997.}
\newblock{ \url{https://doi.org/10.1007/978-0-8176-4755-1}}

\bibitem[Ba]{Barles}
G. Barles,
\newblock{\em Solutions de viscosit\'e des \'equations de Hamilton-Jacobi.}
\newblock{Springer-Verlag Berlin Heidelberg, 1994.}

\bibitem[Be]{Benci}
V. Benci,
\newblock {\em Talk at the Workshop of Symplectic Geometry.}
\newblock {M.S.R.I., Berkeley, CA, USA}, August 1988.

\bibitem[Bern 1]{Bernard}
P.~Bernard,
\newblock{\em The action spectrum near positive definite invariant tori. }
\newblock { Bull. Soc. Math. France,} \vol 131(2003), pp. 603-616.

\bibitem[Bern 2]{Bernard2}
P.~Bernard,
\newblock{\em Symplectic aspects of Aubry Mather Theory.}
\newblock  { Duke Math. Journal,} \vol 136 (2007) no. 3, 401-420.

\bibitem[Birk]{Birkhoff}
G.D. Birkhoff,
\newblock{\em Dynamical Systems.}
\newblock Amer. Math. Soc. Colloquium Publications, 1927 (reprinted 1966).

\bibitem[Bis]{Bisgaard}
M.R. Bisgaard,
\newblock{\em Mather theory and Symplectic rigidity.}
\newblock{Journal of Modern Dynamics},vol. 15 (2019) pp. 165-207. 

\bibitem [B-M]{B-M}
 N. N. Bogoliubov and Y. A. Mitropolski,
 \newblock {\em  Asymptotic methods in the theory of nonlinear
oscillations. }
 \newblock  {Moscow}, 1958. English translation, Gordon and Breach, New York, 1964.

\bibitem[B-C-B]{Couderc}
A. Boudaoud, Y. Couder and M. Ben Amar,
\newblock{ \em A self-adaptative oscillator.}
\newblock { The European Physical  Journal B,} \vol 9 (1999), pp. 159-165.

\bibitem[Bou]{Boudaoud}
A. Boudaoud,
\newblock  { \em De la corde au film de savon: de l'auto-adaptation dans les syst\`emes vibrants.}
\newblock { Images de la Physique,} 2002, pp. 78-83.
\newblock {Formerly available from \url{http://www.cnrs.fr/publications/imagesdelaphysique/2002.htm}, may now be retrieved from \url{https://www.math.ens.fr/~viterbo/Imphy.pdf}}

\bibitem[Br]{Braides}
A. Braides,
\newblock{\em Gamma-convergence for Beginners.}
\newblock Oxford University Press, 2002

\bibitem [Bru]{Brunella}
M. Brunella,
\newblock {\em On a theorem of Sikorav.}
\newblock{L'Enseignement Math \'ematique,} \vol 37 (1991), pp. 83-87.

\bibitem[B-S]{B-S}
L. Buhovsky and S. Seyfaddini,
 \newblock{\em Uniqueness of generating Hamiltonians for topological Hamiltonian flows.}
\newblock{J. Symplectic Geom.} \vol 11 (2013), pp. 37-52.


\bibitem[C-V]{Cardin-Viterbo}
F. Cardin and C. Viterbo,
\newblock {\em Commuting Hamiltonians and multi-time Hamilton-Jacobi equations. }
\newblock { Duke Math Journal,}  \vol 144, (2008), pp. 235-284.


  \bibitem[Ch1]{Chaperon1} M. Chaperon,
 \newblock{\em Lois de conservation et g\'eom\'etrie symplectique.}
   \newblock{  C. R. Acad. Sci. Paris. S\'erie I, Math.},  312 (1991), pp.345-348.
   
\bibitem[Ch2]{Chaperon2}
M. Chaperon,
\newblock { \em Une id\'ee de type  ``g\'eod\'esiques bris\'ees''  pour les syst\`emes Hamiltoniens.}
\newblock  { C.R. Acad. Sci. Paris, S\'erie I,} \vol 298(1984), pp.293-296.


\bibitem [Che]{Chekanov}
Y. Chekanov,
\newblock {\em  Critical points of quasifunctions and generating families of legendrian manifolds.}
\newblock  { Functional Analysis and its Applications}, Volume 30(1996), Number 2, pp.118-128, DOI: 10.1007/BF02509451.
\newblock{(Translated from Funktsional'nyi Analiz i Ego Prilozheniya, \vol 30, No. 2, pp. 56-69, April-June, 1996.)}

\bibitem[Conc]{Concordel}
M.~Concordel,
\newblock  {\em Periodic homogenization of Hamilton-Jacobi equations: Additive eigenvalues and variational formula.}
\newblock { Indiana Univ. Math. J.,} \vol 45(1996), pp. 1095-1118.

\bibitem[CIPP]{CIPP}
G. Contreras, R. Iturriaga, G. Paternain and M. Paternain, 
\newblock{\em Lagrangian Graphs, minimizing Measures and Ma{\~n}e's critical values.}
\newblock{GAFA, Geom. funct. anal.}, \vol 8(1998), pp. 788–809.


\bibitem[C-L]{Crandall-Lions}
M.G. Crandall and P.-L. Lions,
\newblock{\em Viscosity solutions of Hamilton--Jacobi equations. }
\newblock{Trans. Amer. Math. Soc.}, \vol 277 (1983), pp.1-43.

\bibitem[Dal M]{Dal Maso}
G. Dal Maso,
\newblock{\em An introduction to $\Gamma$-convergence.}
\newblock {Birkh\"auser}, 1993.

\bibitem[de G]{de Giorgi}
E. de Giorgi,
\newblock { \em Sulla convergenza di alcune successioni di integrali del tipo dell'area.}
\newblock {Rend. Mat.,} \vol 8(1975), pp 277-294.

\bibitem[de G-F]{de Giorgi-Franzoni}
E. de Giorgi and T. Franzoni,
\newblock   {\em Su un tipo di convergenza variazionale.}
\newblock { Atti Accad. Naz. Lincei. Rend. Cl. Sc. Fis. Mat. Natur., } \vol 58 (1975), 842-850.

\bibitem[DM-G-Z]{DM-G-Z}
G. De Marco, G., Gorni and  G. Zampieri, 
\newblock{\em  Global inversion of functions: an introduction.}
\newblock{NoDEA},  \vol 1(1994), pp. 229-248.

\bibitem[El]{Eliashberg}
Y. Eliashberg,
\newblock{\em  New Invariants of Open Symplectic and Contact Manifolds.}
\newblock  { Journal of the American Mathematical Society,} \vol 4, No. 3. (1991), pp. 513-520.

\bibitem[E-P]{E-P}
M. Entov and L. Polterovich,
\newblock{\em Quasi-states and symplectic intersections.}
\newblock { Commentarii Mathematici Helvetici,} \vol81, 2006, pp 75-99.

\bibitem[Ev]{Evans}
L.C. Evans,
\newblock{\em The perturbed test function method for viscosity solutions of nonlinear PDE.}
\newblock { Proc. Roy. Soc. Edinburgh Sect. A,} \vol 111 (1989), no. 3-4, 359--375.

\bibitem[Fathi1]{Fathi1}
A. Fathi, 
\newblock{\em Th\'eor\`eme KAM faible et th\'eorie de Mather sur les syst\`emes Lagrangiens.}
\newblock{C.R. Acad. Sci. Paris, S \'erie I} \vol 324 (1997) pp. 1043-1046. 

\bibitem[Fathi2]{Fathi2}
A. Fathi, 
\newblock{\em Sur la convergence du semi-groupe de Lax-Oleinik.}
\newblock{C.R. Acad. Sci. Paris, S \'erie I} \vol 327 (1998) pp. 267-270.

\bibitem[Fathi3]{Fathi3}
A. Fathi, 
\newblock{\em Weak KAM theorem in Lagrangian Dynamics.}
\newblock{Version 10.}  
\newblock{Available from   
\url{https://www.math.u-bordeaux.fr/~pthieull/Recherche/KamFaible/Publications/Fathi2008_01.pdf}
}
\bibitem[Fl1]{Floer1} A. Floer,
\newblock{\em The unregularized gradient flow of the symplectic action.}
\newblock { Comm. Pure Appl. Math.,} \vol 41 (1988), 775-813.

\bibitem[Fl2]{Floer2} A. Floer,
\newblock{\em Morse theory for Lagrangian intersections.}
\newblock  { J. Differential Geom.,} \vol 28 (1988), no. 3, 513-547.

\bibitem[Ger]{Gershgorin}
S. Gerschgorin, 
\newblock{\em {\"U}ber die Abgrenzung der Eigenwerte einer Matrix.}
\newblock{ Izv. Akad. Nauk. USSR Otd. Fiz.-Mat. Nauk}, \vol 6 (1931), pp. 749-754.

\bibitem[Gr]{Gromov}
M. Gromov,
\newblock {\em Metric Structures for Riemannian and Non-Riemannian Spaces.}
\newblock Birkh\"auser, 1999 (3rd ed. 2007).

\bibitem[G-V]{Guillermou-Vichery}
S. Guillermou and N. Vichery,
\newblock{\em On the Viterbo conjecture.}
\newblock{(in preparation), 2021}

\bibitem[Ha]{Hadamard}
J. Hadamard,
\newblock{\em Sur les transformations ponctuelles.}
\newblock{Bulletin de la S. M. F.}, \vol  34 (1906), pp. 71-84.

\bibitem[Hofer]{Hofer}
H. Hofer, 
\newblock {\em On the topological properties of symplectic maps. }
\newblock  {Proceedings of the Royal Society of Edinburgh,} \vol 115 (1990), pp. 25--38. 
        
\bibitem[Hu]{Humiliere}
V.~Humili\`ere,
\newblock {\em On some completions of the space of Hamiltonians maps.}
\newblock {Bulletin de la Soci\'et\'e Math\'ematique de France,} \vol 136 (2008), pp. 373-404.

\bibitem[Hu2]{Humiliere 2}
V.~Humili\`ere,
\newblock {\em Continuit\'e en topologie symplectique, PhD, {\'E}cole polytechnique, July 2008.}
\newblock {\url{https://www.theses.fr/2008EPXX0005}}

\bibitem[H-L-S]{H-L-S}
V. Humili\`ere, R. Leclercq and S. Seyfaddini, 
\newblock{\em New Energy-Capacity-like inequalities and uniqueness of continuous Hamiltonians.}
\newblock{Comment. Math. Helv.}, \vol 90 (2015), pp. 1-21. 

\bibitem[K-B]{K-B}
N.M. Krylov and N.N. Bogoliubov,
\newblock{\em  Introduction to nonlinear mechanics.} 
 \newblock Moscow,1937. English translation, Princeton Univ. Press, Princeton, New Jersey, 1947.

\bibitem[Laud]{Laudenbach-Bismut}
F. Laudenbach, 
\newblock{\em On the Thom-Smale complex.}
 \newblock{ Appendix to J.-M. Bismut and W. Zhang. An extension of a theorem by Cheeger and M\"uller.} 
 \newblock{Ast\'erisque}, \vol 205 (1992) pp. 219-233. 
  
\bibitem [Lau-Sik]{Laudenbach-Sikorav} F. Laudenbach and J.-C. Sikorav, 
\newblock { \em Persistance d'intersection avec la section nulle au cours d'une isotopie hamiltonienne dans un fibr\'e cotangent.}
\newblock{Inventiones Math.} \vol 82 (1985) pp. 349-357. 

\bibitem[L-P-V]{L-P-V}
P.-L. Lions, G. C. Papanicolaou and S. R. S. Varadhan,
\newblock  {\em Homogenization of Hamilton-Jacobi Equations.}
\newblock  {Unpublished Preprint, 1987.}
\newblock{Available from \url{http://localwww.math.unipd.it/~bardi/didattica/Nonlinear_PDE_%20homogenization_Dott_%202011/LPV87.pdf}}


\bibitem [Ma]{Ma}
 J. Mather,
 \newblock{\em Action minimizing measures for positive definite Lagrangian systems.}
 \newblock  {Math. Z.}, \vol 207 (1991), pp. 169-207.

\bibitem[Mon]{Mondrian}
P. Mondrian,
\newblock{\em New-York City (1942)}
\newblock{Centre Pompidou, MNAM-CCI. Paris. }
\newblock{\url{https://www.centrepompidou.fr/cpv/resource/c5pRBL/rdyjdr9}}

\bibitem[M-V-Z]{M-V-Z}
A. Monzner, N. Vichery and F. Zapolsky
\newblock{ \em Partial quasi-morphisms and quasi-states on cotangent bundles, and symplectic homogenization.}
\newblock  {Journal of Modern Dynamic}, \vol 6 (2012),  pp. 205-249. 

\bibitem[M-Z]{M-Z}
A. Monzner and F. Zapolsky
\newblock{ \em A comparison of symplectic homogenization and Calabi quasi-states.}
\newblock  {Journal of Topology and Analysis}, \vol 3 (2011),  pp. 243-263. 

\bibitem[Oh]{Oh-spectrum1}
Y.-G. Oh.,
\newblock{\em Construction of spectral invariants of {H}amiltonian paths on
  closed symplectic manifolds},
\newblock{The breadth of symplectic and Poisson geometry.}
\newblock{Progress in  Math.,
  vol. 232 (2005)}, Birkh{\"a}user, pp. 525-570.
  
\bibitem[O-V]{Viterbo-Ottolenghi}  
A.  Ottolenghi, C. Viterbo,
\newblock {\em Solutions g\'en\'eralis\'ees pour l' \'equation de Hamilton-Jacobi dans le cas d'\'evolution.}
 \newblock { Preprint, 1995}.  
\newblock {Available from \url{http://www.math.ens.fr/~viterbo/Ottolenghi-Viterbo.pdf}}

\bibitem [P-P-S]{P-P-S}
G. P. Paternain, L. Polterovich, K.F. Siburg,
\newblock {\em Boundary rigidity for Lagrangian submanifolds, non-removable intersections, and Aubry-Mather theory.}
\newblock {Moscow Math. J.}, \vol 3 (2003), pp. 593-619.

\bibitem[Roos]{Roos}
V. Roos,
\newblock{\em Solutions variationnelles et solutions de viscosit\'e.}
\newblock{PhD thesis, Universit\'e de Paris-Dauphine, 2017.}
\newblock{\url{http://www.theses.fr/2017PSLED023}} and
\newblock{\url{https://basepub.dauphine.fr/handle/123456789/16992}} 

\bibitem[R2]{Roos2}
V. Roos,
\newblock{\em Variational and viscosity operators for the evolutionary Hamilton-Jacobi equation.}
\newblock{ Communications in Contemporary Mathematics}, \vol 21(2019), pp. 1850018-1/76.
\newblock{\url{https://doi.org/10.1142/S0219199718500189}}

\bibitem [S-V]{S-V}
 J. A. Sanders and F. Verhulst,
 \newblock{\em  Averaging methods in nonlinear dynamical systems.}
\newblock {Applied Mathematical Sciences, no. 59}, 
\newblock{Springer-Verlag}, New York, 1985.

\bibitem[Schw]{Schwarz}
M.~Schwarz.
\newblock On the action spectrum for closed symplectically aspherical
  manifolds.
\newblock {\em Pacific J. Math.}, \vol 193(2000), pp. 419-461.

\bibitem[Sey]{Seyfaddini}
S. Seyfaddini. 
\newblock{\em Descent and $C^0$-rigidity of spectral invariants on monotone symplectic manifolds.}
\newblock{J. Topol. Anal.}, \vol 4 (2012), pp.481-498.


\bibitem[She]{Shelukhin}
E. Shelukhin, 
\newblock{\em Symplectic cohomology and a conjecture of Viterbo.}
\newblock{\url{https://arxiv.org/pdf/1904.06798.pdf}}

\bibitem[Sik]{Sikorav-pc} 
J.-C. Sikorav, 
\newblock {\em Talk given at Paris 7 seminar.} 
\newblock { (1990). }

\bibitem[Sik2]{Sikorav}
J.-C. Sikorav
\newblock{ \em Rigidit\'e symplectique dans le cotangent de $T^n$.}
\newblock { Duke Math. J.,} \vol 59, Number 3 (1989), pp. 759-763.

\bibitem[Sorr-V]{Sorrentino-V}
A. Sorrentino and C. Viterbo,
\newblock{\em Action minimizing properties and distances on the group of Hamiltonian diffeomorphisms.}
\newblock{Geometry \& Topology}, \vol 14(2010), pp. 2383-2403. 

\bibitem[Th\'e]{Theret}
David Th{\'e}ret.
\newblock {\em A complete proof of {V}iterbo's uniqueness theorem on generating
  functions.}
\newblock { Topology Appl.}, \vol 96(1999), pp.249-266.

 \bibitem[Tr]{Traynor}
L. Traynor,
\newblock {\em  Symplectic Homology via Generating Functions.}
\newblock { Geometric and Functional Analysis,} \vol 4(1994), pp. 718-748.

\bibitem[Va]{Varga}
R. Varga,
\newblock{\em Ger\v sgorin and His Circles.}
\newblock{Springer Series in Computational Mathematics}, \vol 36.
\newblock{Springer-Verlag, Berlin, Heidelberg (2004)}

\bibitem[V1]{Viterbo-STAGGF} C.~Viterbo, \newblock {\em Symplectic  topology as the geometry of generating
functions.} \newblock  {Mathematische Annalen}, \vol 292(1992), pp. 685-710.

\bibitem[V2]{Viterbo-HJ}
C. Viterbo,
\newblock {\em Solutions d'\'equations d'Hamilton-Jacobi et g\'eom\'etrie symplectique},
\newblock{S\'eminaire X-EDP, \'Ecole polytechnique.}
\newblock{\url{ http://www.numdam.org/item/SEDP_1995-1996____A22_0/}}


 \bibitem[V3]{Viterbo-Montreal}
C. Viterbo,
\newblock  {\em Symplectic topology and Hamilton-Jacobi equations. }
\newblock { ``Morse theoretic methods in nonlinear analysis and in symplectic topology.''}
\newblock{Proceedings of the NATO Advanced Study Institute, Montreal, Canada, July 2004.}
\newblock{NATO Science
Series II: Mathematics, Physics and Chemistry 217. 
}
\newblock{Dordrecht: Springer}, 2006, pp. 439-459. 

\bibitem[V4]{Viterbo-unique}
C. Viterbo,
\newblock{\em On the uniqueness of generating Hamiltonian for continuous limits of Hamiltonians flows}
\newblock { International Mathematics Research Notices} (2006), 9 pages, 
\newblock{\url{https://doi.org/10.1155/IMRN/2006/34028}}
 and Erratum  4 pages, 
 \newblock{\url{htpps://doi.org/10.1155/IMRN/2006/38784}}

 \bibitem[V5]{NCMT}
C. Viterbo,
\newblock  {\em Non-convex Mather theory.} (submitted to Duke Math. J.)
\newblock {\url{https://arxiv.org/abs/1807.09461}}

\bibitem[V6]{Viterbo-stochastic}
C. Viterbo
\newblock{\em Stochastic homogenization of variational solutions of Hamilton-Jacobi equations.}
\newblock{\url{https://arxiv.org/abs/2105.04445}}

\bibitem[V7]{Viterbo-support}
C. Viterbo
\newblock{\em Inverse reduction inequalites for spectral numbers and applications.}
\newblock{\url{https://arxiv.org/abs/2203.13172}} 

\bibitem[WQ1]{WQ1}
Q. Wei
\newblock{\em Solutions de viscosit\'e des \'equations de Hamilton-Jacobi et minmax it\'er\'es.}
\newblock{PhD thesis, Universit\'e de Paris 7} (2013).
\newblock{\url{https://tel.archives-ouvertes.fr/tel-00963780/}}

\bibitem[WQ2]{WQ2}
Q. Wei
\newblock{\em Viscosity solution of the Hamilton-Jacobi equation by a limiting minimax method.}
\newblock{Nonlinearity,} \vol 27(1)(2013), pp. 17--41. 

\bibitem [Zh]{Zhukovskaia} T. Zhukovskaya,
\newblock{\em Singularit\'es  de minimax et solutions faibles d'{E}quations aux d\'eriv\'ees partielles.}
\newblock {Th\`ese de Doctorat, Universit \'e de Paris 7, 1993.}
\newblock{Available from \url{https://www.math.ens.fr/~viterbo/Joukovskaia-these.pdf}}

\bibitem [Zh 2]{Zhukovskaia-2} T. Zhukovskaya,
\newblock{\em Metamorphoses of the Chaperon-Sikorav weak solutions of Hamilton-Jacobi equations}
\newblock {\em  J. Math. Sci.} 82 (1996), no. 5, 3737-3746.
\newblock{Translated from Itogi Nauki i Tekhniki, Seriya Sovremenaya Matematika i Ee Prilozheniya. Tematicheskie Obzory. \vol 20, Topologiya-3, 1994.}).
\newblock{\url{https://doi.org/10.1007/BF02362583}}
\end{thebibliography}
\end{document}